\documentclass[12pt]{article}
\usepackage[utf8]{inputenc}

\RequirePackage[a4paper,margin=2.54cm]{geometry}

\newcommand{\addressumushort}{Mathematics and Mathematical Statistics, Ume{\aa}~University, Sweden}

\newcommand{\addressjushort}{Mechanical Engineering, J\"onk\"oping University, Sweden}

\usepackage{amstext, amsmath, amssymb, graphicx, array, todonotes}
\usepackage{stmaryrd}
\usepackage[update,prepend]{epstopdf}
\usepackage{subfigure}
\usepackage[margin=20pt,font=small,labelfont=bf]{caption}

\RequirePackage{enumitem}
\setlist[itemize]{leftmargin=*}
\setlist[itemize,2]{label=$\circ$, leftmargin=*}

\usepackage{mathtools}

\usepackage{cite}

\newcommand{\bfI}{\boldsymbol I}

\newcommand{\mcP}{\mathcal{P}}

\newcommand{\mcK}{\mathcal{K}}
\newcommand{\mcE}{\mathcal{E}}
\newcommand{\mcN}{\mathcal{N}}

\newcommand{\mcT}{\mathcal{T}}

\usepackage{mathabx}
\newcommand{\tn}{\vvvert}

\newcommand{\Gammah}{{\Gamma_h}}
\newcommand{\nablas}{\nabla_\Gamma}
\newcommand{\nablash}{\nabla_{\Gamma_h}}

\newcommand{\IR}{\mathbb{R}}

\newcommand{\Ps}{{P}}
\newcommand{\Psh}{{P}_h}
\newcommand{\Qs}{{Q}}
\newcommand{\Qsh}{{Q}_h}

\newcommand{\asym}{a_{\text{sym}}}

\newcommand{\Ds}{D_\Gamma}
\newcommand{\Dsh}{D_\Gammah}
\newcommand{\DPs}{{D_\Gamma}}
\newcommand{\DPsh}{{D_\Gammah}}
\newcommand{\epss}{{\epsilon_\Gamma}}
\newcommand{\epssh}{{\epsilon_{\Gammah}}}
\newcommand{\epsPs}{{\epsilon_\Gamma}}
\newcommand{\epsPsh}{{\epsilon_{\Gammah}}}
\newcommand{\nh}{{n_h}}
\newcommand{\thh}{{t_h}}
\newcommand{\Honet}{H^1_{\mathrm{tan}}(\Gamma)}
\newcommand{\Hst}{H^s_{\mathrm{tan}}}

\newcommand{\projs}{\mathcal{P}}

\newcommand{\muh}{|B|} 

\numberwithin{equation}{section}

\newtheorem{lem}{Lemma}[section]

\newtheorem{thm}{Theorem}[section]
\newtheorem{rem}{Remark}[section]

\newenvironment{proof}{\noindent \newline {\bf Proof.}}
{\hfill $\blacksquare$ \newline}

\begin{document}


\title{\bf Analysis of Finite Element Methods for Vector Laplacians on Surfaces}

\author{ 
Peter Hansbo, {Mats~G.~Larson} and
{Karl~Larsson}
}

\date{}

\maketitle
\begin{abstract} 
We develop a finite element method for the vector Laplacian based on the covariant derivative of tangential vector fields on surfaces embedded in $\IR^3$. Closely related operators arise in models of flow on surfaces as well as elastic membranes and shells. The method is based on standard continuous parametric Lagrange elements which describe a $\IR^3$ vector field on the surface and the tangent condition is weakly enforced using a penalization term. We derive error estimates that take the approximation of both the geometry of the surface and the solution to the partial differential equation into account. In particular we note that to achieve optimal order error estimates, in both energy and $L^2$ norms, the normal approximation used in the penalization term must be of the same order as the approximation of the solution. This can be fulfilled either by using an improved normal in the penalization term, or by increasing the order of the geometry approximation. We also present numerical results using higher-order finite elements that verify our theoretical findings.
\end{abstract}





\section{Introduction}

In this contribution we develop a finite element method for the 
vector Laplacian on a surface.
While there are several natural Laplacians acting on vector fields on surfaces we in this work consider the \emph{rough Laplacian} which
is a second order elliptic operator based on covariant derivatives. In contrast, another natural Laplacian is the Hodge Laplacian which is based on exterior calculus, see \cite{Holst}, and which differs from the rough Laplacian by a zeroth order term depending only on the curvature of the surface.

The method is based on continuous parametric Lagrange elements with geometry and solution approximations which are piecewise polynomial of orders $k_g$ and $k_u$, respectively. 
Instead of defining an approximation space for tangent vector fields on the surface $\Gamma$ we seek solutions which are full vector fields  $\Gamma\rightarrow\mathbb{R}^3$ and
weakly enforce the tangential condition using a suitable penalty 
term, similar to our work on the Darcy problem, see \cite{Darcy}. 
Note, however, that the Darcy problem does not involve any gradients 
of the velocity vector and is therefore easier to deal with. This 
approach leads to a convenient implementation without the need for 
special finite element spaces.

We prove a priori error estimates in the energy and $L^2$ norm and 
we find that in order to obtain optimal order convergence in both norms it 
is necessary to use a discrete normal in the penalty term of order $k_u+1$. For isoparametric finite elements this translates into a geometry approximation of the normal in the penalty term that is one degree higher 
than of the normal to the discrete surface $\Gamma_h$.
Somewhat curiously, there is no loss of order in $L^2$ due to the fact that the covariant derivative is obtained by projecting the componentwise directional derivative onto the tangent plane and that the approximation order of the projection is only $h^{k_g}$. To prove this, however, requires the use of non-standard techniques which we developed in \cite{LaLa17}.

\paragraph{Related Work.}
Finite elements for partial differential equations on surfaces is now a rapidly developing field that originates from the seminal work of Dziuk  \cite{Dz88} where surface finite elements for the Laplace--Beltrami operator was 
first developed. Most of the research is, however, focused on problems 
with scalar unknowns, see the recent review article \cite{DzEl13} and the references therein, which simplifies the differential calculus since the 
covariant derivative of a vector field, or more generally a tensor field, is 
not needed. Models of flow on surfaces as well as membranes and shells, however, involve vector unknowns, see for instance  \cite{HaLa14} (linear) and \cite{HaLaLa14} (nonlinear), for membrane models formulated using 
the same approach as used in this paper. Furthermore, we employ higher 
order elements similar to the approaches presented in \cite{De09,LaLa17,Darcy,nedelec}.
Concurrent to the present work, similar formulations for vector Laplace operators on surfaces, also using tangential differential calculus, were studied in \cite{JaOlRe17} motivated by their use in methods posed in an embedding space, and later such a method (TraceFEM) for a vector Laplacian problem was presented in \cite{GrJaOlRe17}. As in the present work the formulation in \cite{GrJaOlRe17} assumes a full vector field on the surface but instead of using a penalty term to enforce the field to be tangential a Lagrange multiplier approach is used. In addition our analysis includes the geometry approximation.

\paragraph{Paper Outline.}
The remainder of this paper is organized as follows: In Section~\ref{section:vec-lap-on-surf} we introduce the vector Laplacian and results concerning the continuous problem; in Section~\ref{section:fem} we introduce the finite element method; in Section~\ref{section:prel-results} we recall some basic results regarding lifting and extension of functions between the discrete and continuous surfaces, present a non-standard geometry approximation estimate and introduce the interpolant; in Section~\ref{section:error-estimates} we derive a sequence of necessary lemmas leading up to the a priori error estimate, and finally in Section~\ref{section:numerics} we present numerical examples confirming our theoretical findings.


\section{Vector {L}aplacians on a Surface} \label{section:vec-lap-on-surf}
In this section we present the tools we need to work with vector Laplacians on surfaces in the setting of tangential differential calculus, which allows us to employ the Cartesian coordinates of the embedding $\IR^3$ space. We first in Section~\ref{section:the-surface} define the surface and its assumptions; in Sections~\ref{section:tensors}--\ref{sec:tangential-calculus} we introduce the notations needed to describe tensor fields on the surface and derivatives and covariant derivatives of such fields; in Section~\ref{section:function-spaces} we present the suitable Sobolev spaces. As these first five sections involve numerous definitions we for clarity and compactness present them in the form of bullet lists. In Section~\ref{section:basic-lemmas} we establish some lemmas fundamental to the analysis on surfaces, in particular a Poincaré inequality. Finally, in Section~\ref{section:vector-laplacians} we introduce our model variational problem, which involves certain vector Laplacians on a surface.

\subsection{The Surface} \label{section:the-surface}
\begin{itemize}
\item Let $\Gamma$ be a smooth compact surface embedded in $\IR^3$ without boundary 
and let $\rho$ be the signed distance function, negative on the inside and positive 
on the outside. The exterior unit normal to the surface $\Gamma$ is given by $n = \nabla \rho$.

\item Let $p:\IR^3 \rightarrow \Gamma$ be the closest point mapping onto $\Gamma$. Then 
there is a $\delta_0>0$ such that $p$ maps each point in $U_{\delta_0}(\Gamma)$ to 
precisely one point on $\Gamma$, where 
$U_\delta(\Gamma) = \{ x \in \IR^3 : |\rho(x)| < \delta\}$ is the open tubular neighborhood 
of $\Gamma$ of thickness $\delta>0$. 

\item
As $\rho$ is a signed distance function within $U_{\delta_0}(\Gamma)$ the unit normal to $\Gamma$ naturally extend to $U_{\delta_0}(\Gamma)$ through its original definition $n(x) = \nabla\rho$.

\item
For each function $u:\Gamma \rightarrow \IR^m$, 
$m=1,2,\dots,$ we define the componentwise extension $u^e$ to the neighborhood 
$U_{\delta_0}(\Gamma)$ by the pull back $u^e = u \circ p$.

\item The curvature tensor (or second fundamental form) is defined on $U_{\delta_0}(\Gamma)$ by 
\begin{equation}\label{eq:curvature-tensor}
\kappa = \nabla \otimes \nabla \rho
\end{equation}
and may be expressed in the form
\begin{equation}\label{Hform}
\kappa(x) = \sum_{i=1}^2 \frac{\kappa_i^e}{1 + \rho(x)\kappa_i^e} a_i^e \otimes a_i^e
\end{equation}
where $\kappa_i$ are the principal curvatures with corresponding orthonormal 
principal curvature vectors $a_i$, see \cite[Lemma 14.7]{GiTr01}. 
\end{itemize}

\subsection{Tensors} \label{section:tensors}
\begin{itemize}
\item Let $V,W$ be finite dimensional vector spaces with bases 
$\{e_i\}_{i=1}^m$ respectively $\{f_i\}_{i=1}^n$. The tensor product $V\otimes W$ is the vector 
space spanned by all pairs $(e_i,f_j)$ of basis vectors, denoted by $e_i \otimes f_j$, 
and there is a bilinear product $\otimes: V \times W \rightarrow V \otimes W$ defined by 
\begin{equation}
v \otimes w = \left(\sum_{i=1}^m v_i e_i\right) \otimes \left( \sum_{j=1}^n w_j f_j \right) 
=   \sum_{i=1}^m  \sum_{j=1}^n v_i w_j (e_i \otimes f_j )
\end{equation}
The dimension of $V\otimes W$ is $\text{dim}(V\otimes W)=\text{dim}(V)\text{dim}(W)$.
\item If $V$ and $W$ are inner product spaces, $V\otimes W$ is an inner product space 
with product 
\begin{equation}
(a\otimes b, v\otimes w)_{V\otimes W} = (a,v)_V (b,w)_W
\end{equation} 
and the inner product norm is given by
\begin{equation}
\| v \otimes w \|_{V\otimes W} = \| v \|_V \| w \|_W
\end{equation}

\item The dual space of $V$ denoted by $V^*$ is the space of all linear functionals 
$\lambda: V \rightarrow \IR$. The dual basis $\{\lambda_j\}_{j=1}^n$ is defined by the identity 
$\lambda_j(e_i) = \delta_{ij}$. When $V$ is an inner product space, there is for each $\lambda \in V^*$ 
a unique vector $\xi_\lambda$ such that $\xi_\lambda(v) = (v, \xi_\lambda)_V$, $\forall v \in V$.

\item Tensors of type $(k,l)$ are elements in the tensor product space
\begin{equation}
W^{k,l} = \underbrace{V\otimes\cdots \otimes V}_{\text{$k$ copies}} 
\otimes 
\underbrace{V^* \otimes \cdots \otimes V^*}_{\text{$l$ copies}}
\end{equation} 

\item If $\{e_i\}_{i=1}^m$ is an orthonormal basis in $V$, then $\{e_i\}_{i=1}^m$ is also the 
corresponding dual basis in $V^*$. If $Q:V\rightarrow V$ is an orthogonal mapping 
$\widetilde{e}_i= Qe_i$ is also an orthonormal basis in $V$ and hence also the corresponding 
dual basis in $V^*$. 
\item For $v$ in $V$ let $[v]$ denote the array of coefficients in the expansion $v = \sum_i v_i e_i$. 
If $v = \sum_i v_i e_i = \sum_i \widetilde{v}_i \widetilde{e}_i =  \sum_i \widetilde{v}_i Q e_i$ 
we find that $v_j =  \sum_i \widetilde{v}_i (Q e_i,e_j)_V = \sum_i \widetilde{v}_i Q_{ji}$, $j=1,\dots, m$, 
and thus in matrix form $[v] = Q [\widetilde{v}]$ or $[\widetilde{v}] = Q^{-1} [v] = Q^T[v]$. The 
same transformation rules hold for the dual space $V^*$ and thus we do not have to distinguish 
between $V$ and $V^*$ and we can restrict our attention to tensors of type $k+l$ of the form 
\begin{equation}
W^{k+l} = \underbrace{V\otimes\cdots \otimes V}_{\text{$k+l$ copies}} 
\end{equation}
\item Let $v\in V^k$ and $w\in V^l$, for $n=1,\dots,\text{min(k,l)}$ we define the 
$n$-contraction $v\cdot_n w \in V^{k+l-2n}$ by 
\begin{equation}
(\otimes_{i=1}^k v_i ) \cdot_n (\otimes_{j=1}^l w_j )
= 
\Pi_{i=1}^n (v_{k-n+i}, w_i)_V  (\otimes_{i=1}^{k-n} v_i ) \otimes (\otimes_{j=1}^{l-n} w_j )
\end{equation}
Special cases include $n=l=1$ or $n=k=1$ and $k=l=2$ where we use the simplified notation  
\begin{equation}
v\cdot w \in V^{k-1}, \qquad v:w \in \IR 
\end{equation}

\end{itemize}

\subsection{Tensor Fields}

\paragraph{Vector Fields.} 

\begin{itemize}
\item Let $\{e_i \in \IR^3\}_{i=1}^3$ be a Cartesian basis, i.e., a fixed orthonormal basis, 
in the embedding space $\IR^3$.
\item  For $x\in U_{\delta_0}(\Gamma)$ let $P(x)=I - n(x) \otimes n(x)$ be the projection onto the tangential plane $T_x(\Gamma)$ and $Q = I - P$ the projection onto the normal line.
\item The projected Cartesian basis $\{ p_i = \Ps e_i :\Gamma \rightarrow T_x(\Gamma) \}_{i=1}^3$ 
spans the tangential plane $T_x(\Gamma)$ but is not a basis for $T_x(\Gamma)$ since the 
vectors in the set are linearly dependent. Note however that for $b \in T_x(\Gamma)$ we have the 
unique expansion $b = \sum_{i=1}^3 b_i e_i$, which induces the canonical expansion  
$b = \sum_{i=1}^3 b_i P e_i = \sum_{i=1}^3 b_i p_i$. Furthermore, inner products and norms are 
clearly independent of the choice of expansion in the projected basis.

\item Define: (a) The space of general smooth vector fields  
\begin{equation}
\mcT^1 = \{ a = \sum_{i=1}^3 a_i e_i : a_i \in C^\infty(\Gamma) \}
\end{equation}
(b) The space of tangential smooth vector fields
\begin{equation}
\mcT^1_{\mathrm{tan}} = \{ a = \sum_{i=1}^3 a_i p_i : a_i \in C^\infty(\Gamma) \}
\end{equation}
\end{itemize}

\paragraph{Tensor Fields.} 
\begin{itemize}
\item Define: (a) The vector space of smooth $m$ tensor fields on $\Gamma$,
\begin{equation}
\mcT^m = \{X = \sum_{i_1,\dots ,i_m =1}^3 X_{i_1,\dots, i_m} e_{i_1} \otimes\dots \otimes e_{i_m}, 
X_{i_1,\dots, i_m} \in C^\infty(\Gamma, \IR)
 \}
\end{equation} 
(b) The vector space of smooth tangential $m$ tensor fields on $\Gamma$,
\begin{equation}\label{eq:tan-tensor-field}
\mcT^m_{\mathrm{tan}} = \{X = \sum_{i_1,\dots, i_m =1}^3 X_{i_1,\dots, i_m} p_{i_1} \otimes\dots \otimes p_{i_m}, 
X_{i_1,\dots, i_m} \in C^\infty(\Gamma, \IR)
 \}
\end{equation} 

\item The projection $\projs:\mcT^m \rightarrow \mcT^m_{\mathrm{tan}}$, is defined by
\begin{equation} \label{eq:def-proj}
\projs\left( \sum_{i_1,\dots, i_m =1}^3 X_{i_1,\dots, i_m} e_{i_1} \otimes\dots \otimes e_{i_m}  \right)
= 
\sum_{i_1,\dots, i_m =1}^3 X_{i_1,\dots, i_m} p_{i_1} \otimes\dots \otimes p_{i_m}
\end{equation}

\end{itemize}

\subsection{Tangential Calculus}
\label{sec:tangential-calculus}

\paragraph{Tangential Derivatives.}
\begin{itemize}
\item
The directional derivative of $u \in \mcT^1$, in the direction of $a \in \mcT^1$, 
is defined by
\begin{equation}\label{eq:partial-a}
\partial_a u
=
(a \cdot \nabla) u^e = (u^e \otimes \nabla)\cdot a
\end{equation}
where $a \cdot\nabla = \sum_{i=1}^3 a_i \partial_i$
and $u^e\otimes\nabla$ is the Jacobian of $u^e$.

\item Define the tangential gradient operator 
$\nablas = \sum_{j=1}^n e_j  \partial_{p_j}$  
and the total derivative of a vector field
\begin{equation}
u \otimes \nablas
=
\sum_{j=1}^3 (\partial_{p_j} u) \otimes e_j
=
\sum_{i,j=1}^3 (\partial_{p_j} u_i) e_i \otimes e_j
=
(u^e \otimes \nabla)P
\end{equation}
We note that 
\begin{equation} \label{eq:partial-a-tan}
\partial_a u = (u \otimes \nablas)\cdot a\qquad \forall a \in \mcT^1_{\mathrm{tan}}
\end{equation}

\item More generally,  for $X \in \mcT^m$ we define in the same way the directional derivative
\begin{equation}\label{eq:tensor-gen-directional-der}
\partial_a X = \sum_{i_1,\dots, i_m =1}^3 (\partial_a X_{i_1,\dots, i_m}) e_{i_1} \otimes\dots \otimes e_{i_m}
\end{equation}
and the total derivative $X\otimes \nablas \in \mcT^{m+1}$,
\begin{equation}\label{eq:tensor-gen-tot-der}
X \otimes \nablas = 
\sum_{j=1}^3 
\sum_{i_1,\dots, i_m =1}^3 (\partial_{p_j} X_{i_1,\dots, i_m}) e_{i_1} \otimes\dots \otimes e_{i_m} \otimes e_j
\end{equation}
and we note that 
\begin{equation}
\partial_a X = (X \otimes \nablas )\cdot a\qquad \forall a \in \mcT^m_{\mathrm{tan}}
\end{equation}
since $a = a_i e_i = a_i p_i$ and $a_i = e_i \cdot a$.

\item Higher order derivatives of $X \in \mcT^m$ are obtained by repeated 
application of (\ref{eq:tensor-gen-tot-der}), 
\begin{equation}
(\nablas)^k X
=
X \otimes \nablas^k
= 
X \,\underbrace{\otimes\,\nablas\otimes\cdots\otimes\nablas}_{\text{$k$ gradients}}
\end{equation}
which gives $(\nablas)^k X \in \mcT^{m+k}$ of the form 
\begin{equation}
(\nablas)^k X
= 
\sum_{j_1,\dots, j_k=1}^3 
\sum_{i_1,\dots, i_m =1}^3 
(\partial_{p_{j_k}} \dots \partial_{p_{j_1}} X_{i_1,\dots, i_m}) (e_{i_1} \otimes\dots \otimes e_{i_m}) 
\otimes 
(e_{j_1} \otimes \dots \otimes e_{j_k})
\end{equation}

\end{itemize}


\paragraph{Covariant Derivatives.}
\begin{itemize}
\item 
For $u \in \mcT^1_{\mathrm{tan}}$ 
we define the covariant derivative of $u$ in the direction $a$ by
\begin{equation}\label{eq:cov-der-def}
D_a u = \Ps \partial_a u
\end{equation}
\item 
Writing $u=\sum_{i=1}^3 u_i p_i$ we have using the product rule
\begin{align}\label{eq:cov-der-a}
D_a u
&=
\Ps \partial_a u
= 
\Ps \partial_a \sum_{i=1}^3 u_i p_i
=
\sum_{i=1}^3 (\partial_a u_i) p_i + u_i \Ps(\partial_a p_i)
\end{align}
We note that the covariant derivative includes a lower order term multiplied by 
a projected directional derivative of a tangent basis vector $p_i$. Writing 
$a= \sum_{j=1}^3 a_j p_j$ we have  $\partial_a p_i 
= \sum_{j=1}^3 a_j \partial_{p_j} p_i $ and using the identity $p_i = e_i - n_i n$ 
we find that
\begin{equation}
\partial_{p_j} p_i  
= \partial_{p_j} ( e_i - n_i n ) 
= -p_j \cdot \kappa_i n + n_i \kappa\cdot p_j  
= -\kappa_{ij} n   - n_i \kappa_j    
\end{equation}
where $\kappa = n \otimes \nabla = \nabla^2 \rho$ is the tangential curvature tensor, 
see (\ref{eq:curvature-tensor}), with elements $\kappa_{ij}$ and columns (and rows) 
$\kappa_j$. Thus $\Ps (\partial_{p_j} p_i) = -  n_i \kappa_j$ and expanding the right 
hand side in the Cartesian basis we obtain
\begin{equation}
\Ps( \partial_{p_j} p_i ) = \sum_{k=1}^3 \gamma_{ij,k} p_k 
\end{equation} 
where the coefficients $\gamma_{ij,k} = - n_i \kappa_{jk}$ correspond to the Christoffel 
symbols of the Levi--Civita connection. 

Furthermore, note that in the case of the canonical expansion $u =  \sum_{i=1}^3 u_i p_i 
= \sum_{i=1}^3 u_i e_i$ we have the simplified identity  
\begin{equation}
D_a u = P \partial_a u 
= P \left( \sum_{i=1}^3 (\partial_a u_i) e_i \right) 
= \sum_{i=1}^3 (\partial_a u_i) p_i
\end{equation}
Using the fact $\sum_{i=1}^3 u_i n_i = 0$, we also note that the second term in the right 
hand side of (\ref{eq:cov-der-a}) is indeed zero since
\begin{equation}
\sum_{i=1}^3 u_i \sum_{j=1}^3 a_j P (\partial_{p_j} p_i) 
= 
-\sum_{i=1}^3 u_i \sum_{j=1}^3 a_j n_i \kappa_j 
=
-\left(\sum_{i=1}^3 u_i n_i \right) \left( \sum_{j=1}^3 a_j \kappa_j\right) = 0 
\end{equation}

\item Define the total covariant derivative $\DPs u \in \mcT^2_{\mathrm{tan}}$,
\begin{align}
\DPs u
&=
\sum_{i=1}^3 (D_{p_i} u)\otimes p_i 
=
\Ps (u \otimes \nablas)
=
P (u^e \otimes \nabla) P
\end{align}
and we note that 
\begin{equation}
D_a u = (D_\Gamma u) \cdot a  \qquad \forall a \in \mcT^1_\mathrm{tan}
\end{equation}
In contrast to $u \otimes \nablas$, $\DPs u $ is a tangential tensor. 

\item The symmetric part of $\DPs u $ is defined by
\begin{equation}
\epsPs(u) = \frac{1}{2}\left( \DPs u + (\DPs u)^T \right)
\end{equation} 
which is the tangential strain tensor used in modeling of solids and fluids, see \cite{HaLa14}.

%

\item The covariant derivative $D_a X \in  \mcT^m_{\mathrm{tan}}$ of a tangential tensor 
$X \in  \mcT^m_{\mathrm{tan}}$ in the direction $a\in \mcT^1_{\mathrm{tan}}$ is defined by 
\begin{align}
D_a X &=
\mcP(\partial_a X)
\\&=
\ \sum_{\mathclap{i_1,\dots,i_m=1}}^3 \
(\partial_a X_{i_1 \cdots i_m}) (p_{i_1} \otimes \cdots \otimes p_{i_m})
+
X_{i_1 \cdots i_m} \projs\left( \partial_a(p_{i_1} \otimes \cdots \otimes p_{i_m}) \right)
\end{align}
where the projection $\mcP$ of a tensor field is defined in \eqref{eq:def-proj}. We use the product rule $\partial_a ( Y \otimes Z ) = (\partial_a Y ) \otimes Z 
+ Y \otimes (\partial_a Z)$, $X \in \mcT^m$, $Y\in \mcT^n$, to compute the second term. 
The total covariant derivative 
$\Ds X \in \mcT^{m+1}_{\mathrm{tan}}$, is defined by
\begin{align}
\Ds X = \sum_{j=1}^3 (D_{p_j} X) \otimes p_j
\end{align}
and note that since $p_j \cdot a = P e_j \cdot a = e_j \cdot P a = e_j \cdot a = a_j$ we have 
\begin{equation}
D_a X = (\Ds X) \cdot a
\end{equation}
for all tangential vector fields $a$. 

\item Iterating this definition we can represent covariant derivatives 
of order $m$ as
\begin{align}
(\Ds)^m u = \underbrace{\Ds \cdots \Ds}_{\mathclap{\text{$m$ covariant derivatives}}} u
\end{align}

\end{itemize}

\subsection{Function Spaces} \label{section:function-spaces}

For $\omega\subset\Gamma$ let $(\cdot,\cdot)_\omega$ and $\| \cdot \|_{L^2(\omega)}$ denote the usual $L^2$ inner product and norm on $\omega$ and let $\| \cdot \|_{L^\infty(\omega)}$ denote the usual $L^\infty$ norm on $\omega$.
We define the following Sobolev spaces:
\begin{itemize}
\item $H^s(\omega)$, with $\omega \subset \Gamma$, denotes the standard 
Sobolev spaces of scalar or vector valued functions with componentwise 
derivatives and norm
\begin{equation} \label{eq:H1norm}
\| v \|^2_{H^s(\omega)} 
=  
\sum_{j=0}^s \| (\nablas)^j v  \|^2_{L^2(\omega)} 
\end{equation}
\item  $\Hst(\omega)$, with $\omega \subset \Gamma$, denotes the Sobolev space of tangential vector 
fields with covariant derivatives and norm
\begin{equation} 
\| v \|^2_{\Hst(\omega)} 
=  
\sum_{j=0}^s \| (\DPs)^j v \|^2_{L^2(\omega)} 
\end{equation}
\end{itemize}
We employ the standard notation $L^2(\omega) = H^0(\omega)$ and 
$\| v \|_{L^2(\omega)} = \| v \|_\omega$.

\subsection{Basic Lemmas} \label{section:basic-lemmas}
We here prove three fundamental lemmas. In Lemma~\ref{lem:constant-fields} we show that
the kernel of the covariant derivative of a tangential vector field is empty, a fact then used in the proof of Lemma~\ref{lem:poincare-cont}, which is a Poincaré inequality. Finally, in Lemma~\ref{lem:H-Ht-equivalence} we show that Sobolev norms based on tangential respectively covariant derivatives are equivalent.

\begin{lem}\label{lem:constant-fields} If $v \in \Honet$ satisfies 
$\DPs v  = 0$ then $v = 0$.
\end{lem}
\begin{proof} {\bf Step 1.} \emph{Claim: if $v\in \mcT^1_{\mathrm{tan}}$ is a smooth tangential vector 
field which is covariantly constant, $\DPs v =0$, there is a point $x \in \Gamma$ 
such that $v(x)=0$.}  

To verify this claim we introduce the Riemannian curvature tensor, see \cite{doCarmo}, which 
is the mapping 
$R:  \mcT^1_{\mathrm{tan}}\times \mcT^1_{\mathrm{tan}} \times \mcT^1_{\mathrm{tan}} \rightarrow \mcT^1_{\mathrm{tan}}$ defined by
\begin{equation}\label{eq:riemannian-curv}
R(a,b,v) 
= D_a D_b v - D_a D_b v - D_{[a,b]} v 
\end{equation}
where $D_a v = \Ps \partial_a v$ is the covariant derivative in the direction of the tangential 
vector field $a$ and $[a,b]$ is the tangent vector field given by the Lie bracket
\begin{equation}
[a,b] = \partial_b a - \partial_a b
\end{equation}
where we recall that $\partial_b a = (a \otimes \nablas)\cdot b$, see \eqref{eq:partial-a-tan}.
To see that the Lie bracket is indeed a tangential vector field we note that since $n\cdot a=0$ 
we have $0= \partial_b (n\cdot a) = (\partial_b n ) \cdot a + n \cdot (\partial_b a) 
= b \cdot \kappa \cdot a +  n \cdot (\partial_b a)$ and thus 
$n \cdot (\partial_b a) = -b \cdot \kappa \cdot a $, from which it follows that $n\cdot [a,b] = 0$.

All derivatives in (\ref{eq:riemannian-curv}) cancel so that $R(a,b,v)$ is a 
tangential vector field which does not depend on any derivatives of $v$.  
In the case of an embedded codimension one surface in $\IR^3$ we have 
the identity
\begin{equation}\label{eq:riemannian-curvature-surface}
R(a,b,v) = (b\cdot \kappa \cdot v) \kappa \cdot a 
-  (a \cdot \kappa \cdot v) \kappa \cdot b
\end{equation}
where $\kappa$ is the curvature tensor of $\Gamma$, and we note in particular that 
there are no derivatives of $v$. 
To verify \eqref{eq:riemannian-curvature-surface} we first
recall the directional and covariant derivatives introduced in Section~\ref{sec:tangential-calculus}, i.e., 
\begin{equation}
\partial_a v = (v \otimes \nablas) \cdot a, \qquad D_a v  = P \partial_a v
\label{eq:directional-covariant-notation}
\end{equation}
for a tangential vector field $a$. We then have
\begin{align}
D_a D_b v &= \Ps \partial_a \Ps \partial_b v
\\
 &= \Ps \partial_a (\partial_b v - (n \cdot \partial_b v ) n )
 \\
  &= \Ps \partial_a \partial_b v 
-  
   \Ps ( \partial_a (n \cdot \partial_b v ) n + (n \cdot \partial_b v ) \partial_a n  )
   \\
   &=  \Ps \partial_a \partial_b v 
-  
(n \cdot \partial_b v ) \kappa\cdot a  
   \\
   &=  \Ps \partial_a \partial_b v 
+ 
(b \cdot \kappa \cdot v ) \kappa\cdot a  
\end{align}
Here we used that identities 
\begin{equation}
\Ps n = 0, 
\qquad 
\Ps\kappa = \kappa, 
\qquad 
\partial_a n = \kappa \cdot a,
\qquad
n \cdot \partial_b v = - b \cdot \kappa \cdot v 
\end{equation}
where the last formula follows from the fact that $v \cdot n =0$, 
which leads to
\begin{equation}
0 = \partial_b (n \cdot  v  ) =  n \cdot (\partial_b v)    
+ \partial_b n \cdot v 
= n \cdot (\partial_b v)    
+ b \cdot \kappa \cdot v 
\end{equation}
We thus obtain
\begin{align}
R(a,b,v) &= D_a D_b v - D_b D_a v  - D_{[a,b]} v
\\
 &=  \underbrace{\Ps (\partial_a \partial_b v - \partial_b \partial_a v - \partial_{[a,b]} v )}_{\bigstar}
 \\ \nonumber &\qquad
 +  
 (a\cdot \kappa \cdot v ) \kappa\cdot b  - (b \cdot \kappa \cdot v ) \kappa\cdot a
 \\
 &=  (a\cdot \kappa \cdot v ) \kappa\cdot b  - (b \cdot \kappa \cdot v ) \kappa\cdot a
\end{align}
Here we used the identity 
\begin{equation}
\partial_a \partial_b v_i = \partial_b \partial_a v_i + \partial_{\partial_a b} v_i 
\end{equation}
for each component $v_i$ in $v$, to conclude that 
\begin{align}
\partial_a \partial_b v - \partial_b \partial_a v 
&= \partial_{\partial_a b}  v -  \partial_{\partial_b a}  v
=\partial_{[a,b]} v
\end{align}
and thus $\bigstar = 0$.
This concludes the verification of (\ref{eq:riemannian-curvature-surface}).

Next let $\{t_i\}_{i=1}^2$ be a smooth orthonormal basis to $T_x(\Gamma)$ in the vicinity 
of a point $x \in \Gamma$, i.e., all tangential vector fields can be written as a linear combination $v = \sum_{i=1}^2 v_i t_i$ with coordinate functions $v_i$. We then have the identity  
\begin{equation}\label{eq:riemann-a}
R(t_2,t_1,t_1) \cdot t_2  = R(t_1,t_2,t_2) \cdot t_1 = K
\end{equation}
where $K = \kappa_1 \kappa_2$ is the Gauss curvature
and it also holds
\begin{equation}\label{eq:riemann-b}
R(t_2 , t_1, t_2)\cdot t_2 = R(t_1 , t_2, t_1)\cdot t_1 = 0
\end{equation}
and we get the corresponding identities if we interchange $t_1$ and $t_2$.
In verification of \eqref{eq:riemann-a} we directly obtain
\begin{align}
R(t_2,t_1,t_1) \cdot t_2 &= (t_1 \cdot \kappa \cdot t_1) (t_2 \cdot \kappa \cdot t_2) 
-  (t_2 \cdot \kappa \cdot t_1) (t_2 \cdot \kappa \cdot t_2) = \det (\kappa) = 
\kappa_1 \kappa_2
\end{align}
where $\text{det}(\kappa)$ is the determinant of the $2\times 2$ tangential part of $\kappa$ and we used the fact that the matrix $T=[t_1, t_2]$ 
is orthogonal and $\det(T^T \kappa T) = \det \kappa$.
For \eqref{eq:riemann-b} we get
\begin{align}
R(t_1 , t_2, t_1)\cdot t_1 
=  ( t_2 \cdot \kappa \cdot t_1) (t_1\cdot \kappa \cdot t_1) 
-  (t_1 \cdot \kappa \cdot t_1) (t_2 \cdot \kappa \cdot t_1) 
= 0
\end{align}
and we note that both verifications hold also if we switch $t_1$ and $t_2$.

If $\Ds v = 0$ we have $D_a v = 0$ for all tangential vector fields 
$a$ and thus we can conclude that $R(a,b,v) = 0$ for all tangential 
vector fields $a,b$. Expanding $v$ in the orthonormal frame we also 
have the identity 
\begin{align}
0 = R(a,b,v)\cdot w &= R(a,b,\sum_{i=1}^2 v_i t_i)\cdot w
= \sum_{i=1}^2 v_i R(a,b, t_i)\cdot w
\end{align}
Setting $a = t_1$, $b=t_2$, and  $w=t_1$  we get
\begin{equation}
0= v_2 K
\end{equation}
and setting $a = t_2$, $b=t_1$, and  $w=t_2$  we get
\begin{equation}
0= v_1 K
\end{equation}
We can therefore conclude that in a point with nonzero Gauss curvature 
a covariantly constant vector field must be zero. For any closed compact 
smooth surface embedded in $\IR^3$ there is at least one point 
$x\in \Gamma$ where $K\neq 0 $, see \cite[Theorem 4, p. 88]{Thorpe}, and 
thus $v(x)=0$ which concludes the verification of the claim in Step 1.

\paragraph{Step 2.}
\emph{Claim: if $v\in \mcT^1_{\mathrm{tan}}$ is a smooth tangential vector 
field which is covariantly constant, $\DPs v =0$, and there exists a point $x\in\Gamma$ such that $v(x) = 0$, then $v(y)=0$ for all $y\in \Gamma$.}

We will use so called parallel transport of vectors along curves to verify this claim. First using the 
fact that a closed compact manifold is geodesically complete in the sense that each point 
$y \in \Gamma$ is connected to $x$ by a geodesic, i.e., a length minimizing curve: $\gamma:I \ni t \rightarrow \gamma(t) \in \Gamma$ where $I=[a,b]$ is an interval in $\IR$ and $\gamma(a) = x$, $\gamma(b) = y$. 
Consider now the transport problem: find $w \in \{T_{x}(\Gamma) : x \in \gamma\}$  
such that 
\begin{equation}\label{eq:parallel-transport}
\text{$D_{\dot{\gamma}} w = 0$ on $\gamma$}, \qquad w(a)=v(x)
\end{equation}
where $\dot{\gamma} =\frac{d\gamma}{dt}$ is the tangent vector to $\gamma$. We note that 
$\frac{d w \circ \gamma}{d t} = \partial_{\dot{\gamma}} w$ and thus 
$D_{\dot{\gamma}} w = P \frac{d w\circ \gamma}{d t}$. Setting $w\circ \gamma(t) 
= \sum_{i=1}^2 w_i(t) t_i \circ \gamma(t)$ we get 
\begin{equation}
0 = P \frac{d w\circ \gamma}{d t} =  \sum_{i=1}^2 \frac{dw_i(t)}{dt} t_i\circ \gamma(t)
+ w_i(t) (D_{\dot{\gamma}(t)} t_i)
\end{equation}
and using the fact that $\{t_i\}_{i=1}^2$ is orthonormal we obtain 
\begin{equation}
\frac{dw_i(t)}{dt}
+ \sum_{j=1}^2  w_j(t) (D_{\dot{\gamma}(t)} t_j)\cdot t_i
\end{equation}
which is a standard system of linear ordinary differential equations with a unique solution 
since the coefficients are smooth. We say that $w$ is the parallel transport of $v$ along 
the curve $\gamma$. Now let $w_1$ and $w_2$ be solutions to (\ref{eq:parallel-transport}) 
with initial data $v_1$ and $v_2$, we then have 
\begin{equation}
\frac{d (w_1\cdot w_2)}{dt} 
= \frac{d w_1}{d t} \cdot w_2 + w_1 \cdot \frac{ d w_2}{dt} 
=(P\frac{d w_1}{d t} ) \cdot w_2 + w_1 \cdot ( P \frac{ d w_2}{dt} )  
=0
\end{equation} 
where we used the fact that $w_1$ and $w_2$ are tangent vectors to insert $P$. Thus 
the scalar product of $w_1$ and $w_2$ is constant along $\gamma$ and in particular 
we have $\| w(t) \|_{\IR^3} = \|v(x)\|_{\IR^3}$. We conclude that $v(y)=0$, since $v(x)=0$ and 
$v(y)$ is obtained by parallel transport of $v(x)$ along $\gamma$  since $D_\Gamma v =0$ on 
$\Gamma$ implies $D_{\dot{\gamma}} v =0$ on $\gamma$.

\paragraph{Step 3.} Using the fact that smooth tangent vector fields are dense in $\Honet$ 
the desired result follows.
\end{proof}

\begin{lem}[Poincaré Inequality]\label{lem:poincare-cont} For all $v \in \Honet$ there is a constant 
such that
\begin{equation}\label{eq:poincare-cont}
\| v \|_\Gamma \lesssim \| \Ds v \|_\Gamma 
\end{equation}
\end{lem}
\begin{proof} Assume that (\ref{eq:poincare-cont}) does not hold. Then there is 
a sequence $\{v_k\}_{k=1}^\infty$ in $\Honet$ such that 
\begin{equation}
\| v_k \|_\Gamma \geq k \| \Ds v_k \|_\Gamma 
\end{equation}
Setting $w_k = v_k/\| v_k \|_\Gamma$ we obtain 
\begin{equation}
\| \Ds w_k \|_\Gamma \leq k^{-1}  
\end{equation}
and therefore $\{w_k\}_{k=1}^\infty$ is bounded in $\Honet$. Using Rellich's compactness theorem, see \cite[Ch. 4, Prop. 4.4]{MR1395148}, there is a subsequence $\{w_{k_j}\}_{j=1}^\infty$ and a tangential vector field $w \in L^2(\Gamma)$ such that
\begin{equation}
w_{k_j} \rightarrow w\qquad \text{in $L^2(\Gamma)$}
\end{equation} 
Then $\| w \|_\Gamma = 1$ and $\| \Ds w \|_\Gamma = 0$ but this is a contradiction 
in view of Lemma \ref{lem:constant-fields}.
\end{proof}

\begin{lem}[Sobolev Norm Equivalence]\label{lem:H-Ht-equivalence} For all tangential 
vector fields $v \in H^m_{\mathrm{tan}}(\Gamma)$, and $m=1,2,\dots$ there are 
constants such that
\begin{align}\label{eq:GradCov-equiv}
\| \nablas^m v \|_\Gamma 
&\lesssim 
\sum_{k=0}^m \| \Ds^k v \|_\Gamma
\\ \label{eq:GradCov-equiv2}
\| \Ds^m v \|_\Gamma
& \lesssim
\sum_{k=0}^m \| \nablas^k v \|_\Gamma 
\end{align}
and as a consequence
\begin{equation}\label{eq:H-Ht-equivalence}
\| v \|_{H^m(\Gamma)} \sim \| v \|_{H^m_{\mathrm{tan}}(\Gamma)}
\end{equation}
\end{lem}
\begin{proof}
Let $X\in\mcT^n_{\mathrm{tan}}$ be a smoothly varying tangential tensor on $\Gamma$.
\paragraph{Bound (\ref{eq:GradCov-equiv}).}
Taking $k$ derivatives on $X$,
adding and subtracting a projection on the innermost derivative and using the triangle inequality we 
obtain
\begin{align}\label{eq:gkdsj}
\| \nablas^k X \|_\Gamma &=  \| \nablas^{k-1}(\projs (\nablas X) + (I -\projs) (\nablas X) ) \|_\Gamma
\\ \label{eq:gkdsj2}
&\leq \| \nablas^{k-1}(\Ds X ) \|_\Gamma + \|\nablas^{k-1}((I -\projs) (\nablas X)) \|_\Gamma
\end{align}
where we in the first term use the definition of the covariant derivative $\Ds X = \projs (\nablas X)$. 
Next we show that the second term is actually of lower order. Expressing $X$ using the canonical expansion in the spanning set, see \eqref{eq:tan-tensor-field}, we by the product rule have the total derivative
\begin{align}
\nablas X
&=
\sum_{\mathclap{i_1,\dots,i_n,j=1}}^3 \
\partial_{p_{j}} \left(X_{i_1 \cdots i_n}\right)
\underbrace{ p_{i_1} \otimes \cdots \otimes p_{i_n}  \otimes p_{j}}_{\in\mcT^{n+1}_{\mathrm{tan}}}
+
X_{i_1 \cdots i_n}
\partial_{p_{j}}(p_{i_1} \otimes \cdots \otimes p_{i_n} ) \otimes p_{j}
\end{align}
As the first term in this sum is tangential we after subtraction of a projection get the expression
\begin{align}
(I-\projs) (\nablas X)
&=
\sum_{\mathclap{i_1,\dots,i_n,j=1}}^3 \
X_{i_1 \cdots i_n}
(I - \projs)\left(
\partial_{p_{j}}(p_{i_1} \otimes \cdots \otimes p_{i_n} ) \otimes p_{j}
\right)
\end{align}
which has no derivatives acting on the coordinates of $X$. 
Furthermore, we note that we have the identities
\begin{equation}
(I-P) \partial_{p_j} p_i = - \kappa_{ij} n = - \sum_{k=1}^3 \kappa_{ij} n_k e_k \, , 
\qquad 
p_j = \sum_{l=1}^3 p_{jl} e_l
\end{equation}
where the expansion coefficients are smooth since $\Gamma$ is smooth. Thus there are smooth 
functions $\alpha_{i_1\dots i_n,k_1\dots k_n,l}$  such that 
\begin{equation}
 (I - \projs)\left(\sum_{j=1}^3
\partial_{p_{j}}(p_{i_1} \otimes \cdots \otimes p_{i_n} ) \otimes p_{j} \right)
= \alpha_{i_1\dots i_n,k_1\dots k_n,l} e_{k_1} \otimes \cdots \otimes e_{k_n} \otimes e_l 
\end{equation}
Defining the smooth $2n +1$ tensor $A$ by 
\begin{equation}
A = \alpha_{i_1\dots i_n,k_1\dots k_n,l} (e_{i_1}\otimes \dots \otimes e_{i_n} ) 
 \otimes (e_{k_1} \otimes \cdots \otimes e_{k_n})  \otimes e_l 
\end{equation}
we have the identity 
\begin{equation}
e_{i_1}\otimes \dots \otimes e_{i_n} \cdot_n A = \alpha_{i_1\dots i_n,j;k_1\dots k_n,l} e_{k_1} \otimes \cdots \otimes e_{k_n} \otimes e_l 
\end{equation}
and thus using the canonical expansion of the tangential tensor $X$ we obtain the identity
\begin{equation}
X \cdot_n A = (I-\projs) (\nablas X)
\end{equation}
Using the product rule we get
\begin{align}
\| \nabla_\Gamma^{k-1}  (I-\projs) (\nablas X) \|_\Gamma 
&=
\| \nabla_\Gamma^{k-1}  (X\cdot_n A)  \|_\Gamma
\\
&\lesssim 
\sum_{l=0}^{k-1} \| \nabla_\Gamma^{l} X \|_\Gamma \underbrace{ \| \nabla_\Gamma^{k-1-l} A  \|_{L^\infty(\Gamma)}}_{\lesssim 1}
\\
&\lesssim 
\sum_{l=0}^{k-1} \| \nabla_\Gamma^{l} X \|_\Gamma 
\end{align}
Combined with \eqref{eq:gkdsj}--\eqref{eq:gkdsj2} this yields 
\begin{align}\label{eq:sobeqv-f}
\| \nablas^k X \|_\Gamma
&\lesssim \| \nablas^{k-1}(\Ds X ) \|_\Gamma + \sum_{l=0}^{k-1}  \|\nablas^{l} X \|_\Gamma
\end{align}
with a constant depending only on $\Gamma$. Inequality \eqref{eq:GradCov-equiv} now follows by 
induction. For $k=1$ estimate \eqref{eq:GradCov-equiv} follows directly from (\ref{eq:sobeqv-f}). 
Assuming that \eqref{eq:GradCov-equiv} holds for $k-1$ we have the estimate
\begin{align}
\| \nablas^k X \|_\Gamma
&\lesssim \| \nablas^{k-1}(\Ds X ) \|_\Gamma + \sum_{l=0}^{k-1}  \|\nablas^{l} X \|_\Gamma
\\
&\lesssim 
\sum_{l=0}^{k-1}  \|\Ds^{l} (D_\Gamma X) \|_\Gamma + \sum_{l=0}^{k-1}  \|\Ds^{l} X \|_\Gamma
\\
&\lesssim 
\sum_{l=0}^{k}  \|\Ds^{l} X \|_\Gamma
\end{align}
and thus  \eqref{eq:GradCov-equiv} holds for $k$ as well.
\paragraph{Bound (\ref{eq:GradCov-equiv2}).}
By adding and subtracting $\nablas X$ inside the gradients and applying the triangle inequality we have
\begin{align}\label{eq:bdfse}
\| \nablas^{k-1} \Ds X \|_\Gamma &\leq \| \nablas^{k-1}(\nablas X ) \|_\Gamma + \|\nablas^{k-1}((I -\projs) (\nablas X)) \|_\Gamma
\\&
\lesssim
\| \nablas^{k} X \|_\Gamma
+
\|\nablas^{k-1} X \|_\Gamma
\end{align}
where we use the same lower order bound on the second term in \eqref{eq:bdfse} as above.
The inequality readily follows by iterating this formula, starting with $k=1$ and $X=\Ds^{m-1}v$,
and applying the Poincaré inequality $\|v\|_\Gamma \lesssim \| v\otimes \nablas\|_\Gamma$. This Poincaré inequality clearly holds as we by Lemma~\ref{lem:poincare-cont} have
\begin{align}
\| v \|_\Gamma^2 \lesssim \| \Ds v \|_\Gamma^2
\leq
\| \Ds v \|_\Gamma^2 + 
\underbrace{
\| (I-P)( v\otimes\nablas) \|_\Gamma^2
}_{\geq 0} 
= \| v\otimes \nablas  \|_\Gamma^2
\end{align}
by the orthogonality between tangential and non-tangential tensors.
\end{proof}

\subsection{Vector Laplacians} \label{section:vector-laplacians}

\paragraph{Standard Formulation.}
We consider the variational problem: Find $u \in \Honet$ such that 
\begin{equation}\label{eq:problem}
a(u,v) = l(v) \qquad \forall v \in \Honet
\end{equation}
where the forms are given by
\begin{equation}
a(u,v) = (\DPs(u_t),\DPs(v_t))_\Gamma ,  \qquad 
l(v) = (f,v_t)_\Gamma
\end{equation}
and $f$ is a given tangential vector field in $H^{-1}_{\mathrm{tan}}(\Gamma)$. Here we introduced the notation
\begin{align}
u_t = P u
\end{align}
which at first sight seems superfluous as $u,v,f$ are already tangential. However, the added projections make the above forms well defined also for functions in $H^1(\Gamma)$ which will allow us to deal with these forms and its discrete counterpart in a systematic fashion.

Using the Poincaré inequality (see Lemma \ref{lem:poincare-cont}) together with the Lax--Milgram lemma we 
conclude that this problem has a unique solution $u \in \Honet$.

\paragraph{Elliptic Shift.}
For smooth surfaces we have the following elliptic shift property 
\begin{equation}\label{eq:elliptic-shift}
\| u \|_{H^{s+2}_{\mathrm{tan}}(\Gamma)} 
\leq C \left( \| f \|_{H^{s}_{\mathrm{tan}}(\Gamma)} + \| u \|_{H^{s}_{\mathrm{tan}}(\Gamma)}\right)
\end{equation} 
with a positive constant $C=C(\Gamma,s)$, see \cite[Fundamental Inequality 6.29]{Wa83}.
For $f\in L^2(\Gamma)$, i.e., $s=0$, this shift implies
\begin{equation}\label{eq:elliptic-stab}
\| u \|_{H^{2}_{\mathrm{tan}}(\Gamma)}
\lesssim
\| f \|_\Gamma
\end{equation}
as we by the Poincaré inequality \eqref{eq:poincare-cont} and the Cauchy--Schwarz inequality have
\begin{align}
\| u \|_\Gamma^2 \lesssim \| \Ds u \|_\Gamma^2 = a(u,u) = (u,f)_\Gamma \leq \| u \|_\Gamma \| f \|_\Gamma
\end{align}
which allows us to bound the lower order term in \eqref{eq:elliptic-shift} by $\| f \|_\Gamma$.

\paragraph{Symmetric Formulation.}
While we focus our presentation on the 
standard formulation \eqref{eq:problem}
we will also briefly consider a problem based on the symmetric 
part of the covariant derivative $\epss(u_t)$.
However, in contrast to the standard formulation where the kernel of the full covariant derivative $\ker(\Ds)$ by Lemma~\ref{lem:constant-fields} is empty, the kernel of the symmetric part of the covariant derivative $\ker(\epss)$ is finite dimensional albeit non-trivial and consists of so-called Killing vector fields. Simple examples include surfaces with rotational symmetries where restrictions of three dimensional rigid body rotations induce Killing vector fields on the surface.
To avoid having to deal with the peculiarities of this non-trivial kernel we for the symmetric formulation consider the following problem that also includes a zeroth order term:
Find $u \in \Honet$ such that 
\begin{equation}\label{eq:problem-sym}
\asym(u,v) = l(v) \qquad \forall v \in \Honet
\end{equation}
where the bilinear form is given by
\begin{equation}\label{eq:symmetric-form}
\asym (u,v) = (\epss(u_t),\epss(v_t))_\Gamma  + (u_t,v_t)_\Gamma
\end{equation}
where we note the presence of a zeroth order term.
To prove existence and uniqueness in this symmetric formulation we,
in addition to the results above, require a Korn's inequality $\| \Ds(u) \|_\Gamma \lesssim \| \epss(u) \|_\Gamma$ for $u\in \Honet$.
For a proof of such a Korn's inequality and further discussion on Killing vector fields, see \cite{JaOlRe17}.

\section{The Finite Element Method} \label{section:fem}

In this section we present the finite element method. First, in Section~\ref{section:fem-surface} we introduce the discrete surface approximation in the form of a parametric triangulation fulfilling certain assumptions. In Section~\ref{section:fem-space} we define the parametric finite element space on the discrete surface. The finite element method is presented in Section~\ref{sec:method} where we also consider some variations of the method.

\subsection{Triangulation of the Surface} \label{section:fem-surface}

\paragraph{Parametric Triangulated Surfaces.} 
Let $\widehat{K}\subset \IR^2$ be a reference triangle and 
let $P_{k_g}(\widehat{K})$ be the space of polynomials of 
order less or equal to $k_g$ defined on $\widehat{K}$. Let 
$\Gamma_{h,k_g}$ be a triangulated surface in $\IR^3$ with quasi 
uniform triangulation $\mcK_{h,k_g}$ and mesh parameter 
$h\in (0,h_0]$ such that each triangle $K$ can be described via a mapping $F_{K,k_g}:\widehat{K} \rightarrow K$
where $F_{K,k_g} \in [P_{k_g}(\widehat{K})]^3$. 
Concretely, the construction of a higher-order surface triangulation is based on first generating a regular piecewise linear triangle surface mesh $\mcK_{h,1}$. We then equip each facet element $K_1 \in \mcK_{h,1}$ with the standard $k_g$:th order Lagrange basis $\{\varphi_i\}$ associated with nodes $\{x_i\}$ on $\overline{K}_1$. The higher-order geometry approximation is then defined as the Lagrange interpolant of the closest point mapping $p$, i.e.,
\begin{equation}
K_{k_g} = \left\{ x = \sum_i p(x_i) \varphi(x') \, : \, x' \in K_1  \right\}
\end{equation}
which gives us $\mcK_{h,k_g}$. Note that this is precisely the construction of the higher-order geometry approximation used in \cite{De09}.

Let $n_{h,k_g}$ be the elementwise defined normal to $\Gamma_{h,k_g}$. 
For brevity we use the notation $\mcK_h = \mcK_{h,k_g}$,
 $\Gamma_h = \Gamma_{h,k_g}$ and $n_h = n_{h,k_g}$.
We let geometric quantities derived from $\Gammah$ be indicated by subscript $h$, for example the discrete curvature tensor $\kappa_h$ and the projections $\Psh = I - \Qsh$ and $\Qsh=n_h \otimes n_h$ onto the discrete tangential plane respectively onto the discrete normal line.

\paragraph{Geometry Approximation Assumption.} We assume that the family 
$\{\Gamma_{h,k_g}, h \in (0,h_0]\}$ approximates 
$\Gamma$ in the following ways:
\begin{itemize}
\item $\Gamma_{h,k_g} \subset U_{\delta_0}(\Gamma)$ and 
$p:\Gamma_{h,k_g} \rightarrow \Gamma$ is a bijection.
\item The following bounds hold:
\begin{equation}\label{eq:geombounds}
\| p \|_{L^\infty(\Gamma_{h})}\lesssim h^{k_g+1}, 
\quad
\| n\circ p - n_h \|_{L^\infty(\Gamma_{h})}\lesssim h^{k_g},
\quad
\| \kappa \circ p - \kappa_h \|_{L^\infty(\Gamma_{h})}\lesssim h^{k_g-1}
\end{equation}
\end{itemize}
Here and below we let $a \lesssim b$ denote $a \leq C b$ with a constant $C$ independent of the mesh parameter $h$.

From \eqref{eq:geombounds} we can derive bounds for approximation of other geometric quantities, for example $\| \Ps-\Psh \|_{L^\infty(\Gamma_h)} \lesssim h^{k_g}$,
$\| \Ps\cdot\nh \|_{L^\infty(\Gamma_h)} \lesssim h^{k_g}$, and
$\| 1 - n\cdot\nh \|_{L^\infty(\Gamma_h)} \lesssim h^{k_g+1}$, see eg. \cite{De09}.

\paragraph{Broken Sobolev Spaces.}
As $\Gammah$ is only piecewise smooth we introduce the broken Sobolev space
$H^s(\mcK_h)$ on $\Gammah$ of scalar or vector valued functions with norm
\begin{equation} \label{eq:H1normh}
\| v \|^2_{H^s(\mcK_h)} 
=  
\sum_{j=0}^s \| (\nablash)^j v  \|^2_{L^2(\mcK_h)}
\end{equation}
which we note is analogously defined to \eqref{eq:H1norm} albeit on the discrete surface $\Gammah$.
Here we introduced the convention that when $\mcK_h$ is the domain of integration, element-wise integration over $\Gammah$ is implied, i.e., $\| \cdot \|_{\mcK_h}^2 = \sum_{K\in\mcK_h} \| \cdot \|_K^2$.

We also have the corresponding broken space on the exact surface $\Gamma$ denoted $H^s(\mcK_h^l)$ where $\mcK_h^l$ is defined as follows:
For any parametric triangle $K\in\mcK_{h,i}$, $1\leq i \leq k_g$, we define the lifted triangle $K^l\in\Gamma$ by $K^l=\{p(x) : x\in K\}$.
Let $\mcK_h^l=\bigcup_{K\in\mcK} K^l$ and let the norm of $H^s(\mcK_h^l)$ be given by \eqref{eq:H1norm}.
Clearly, $H^1(\Gamma) \subset H^1(\mcK_h^l)$.
In Section~\ref{section:extandlift} we also introduce the corresponding notation for the lifting of functions on $\Gammah$ onto $\Gamma$.

\subsection{Parametric Finite Element Spaces}  \label{section:fem-space}

Let
\begin{equation}\label{spaceVh}
V_{h,k_u,k_g} = \{ v: {v \vert_K \circ F_{K,k_g} \in P_{k}(\hat K),\; \forall K\in\mcK_{h,k_g}};\; v \; \in C^0(\Gamma_h)\}
\end{equation}
be the space of parametric continuous piecewise polynomials 
of order $k_u$ mapped with a mapping of order $k_g$. For 
brevity we use the simplified notation 
\begin{equation}
V_h = [V_{h,k_u,k_g}]^3
\end{equation}
Note that $V_h\subset H^1(\mcK_h) \cap C^0(\Gammah)$.

\subsection{Formulation of the Method} \label{sec:method}

\paragraph{Tangential Condition.}
While our sought solution is a vector field which is tangential to the surface, this condition is not built into our approximation space $V_h$. Instead we choose to enforce this tangent condition weakly by adding a term $s_h$, defined below, to the bilinear form which penalizes the normal component together with a suitable $h$ scaling. However, as seen in the analysis below, in order to achieve optimal order estimates, in both energy and $L^2$ norms, when using isoparametric finite elements we need to define this penalty term using a normal approximation which is at least one order higher than of the normal to the discrete surface $\Gamma_h$. We denote this normal approximation $\widetilde{n}_h$ and assume
\begin{align}
\label{eq:nhtilde-assumption}
\| n\circ p - \widetilde{n}_h \|_{L^\infty(\Gamma_{h})}\lesssim h^{k_p} \ , \quad k_p \geq k_g
\end{align}
In the case $k_p=k_g$, i.e., when $\widetilde{n}_h$ is of the same approximation order as $n_h$, we choose $\widetilde{n}_h=n_h$.
For $k_p\geq k_g+1$ we instead construct $\widetilde{n}_h$ by taking the node-wise interpolation of the exact normal $n$ using a Lagrange basis of order $k_p-1$ and normalizing this quantity.
While this construction clearly fulfills \eqref{eq:nhtilde-assumption} we discuss other choices in Remark~\ref{rem:better-normal} below.

\paragraph{The Method.}
The finite element method takes the form: Find $u_h \in V_h$ such that
\begin{equation}\label{eq:method}
A_h(u_h,v)= l_h(v)\qquad \forall v \in V_h
\end{equation}
The forms are defined by
\begin{equation}
A_h(v,w) = a_h(v,w) + s_h(v,w)
\end{equation}
with
\begin{align}\label{def:ah}
a_{h}(v,w) &= (\DPsh v_\thh, \DPsh w_\thh)_{\mcK_h}
\\
s_h(v,w) &= \beta h^{-2} (  v_{\widetilde{n}_h} , w_{\widetilde{n}_h} )_{\mcK_h}
\\
l_h(v) &= (f \circ p,v_\thh)_{\mcK_h}
\end{align}
where $\beta>0$ is a parameter. Here we used the notation 
\begin{equation}
v = \Psh v + \Qsh v = v_\thh + v_\nh \nh
\end{equation}
for the decomposition of a general vector field on $\Gammah$ into 
a tangential and a normal fields, and
\begin{equation}
v_{\widetilde{n}_h} = v \cdot {\widetilde{n}_h}
\end{equation}
for the component of a general vector field in the approximate normal direction ${\widetilde{n}_h}$.
The form $s_h$ is added to weakly enforce the tangent condition. 
Note that these forms are defined for $v,w\in H^1(\mcK_h)$
and recall that $V_h\in H^1(\mcK_h)$.

When implementing \eqref{def:ah} it is convenient to use the identity
\begin{align}\label{eq:DSh-vt-id}
\DPsh v_\thh = \Psh ( v \otimes \nablash) - \kappa_h v_{n_h}
\end{align}

\begin{rem}[Consistency]
The method (\ref{eq:method}) is inconsistent due to the geometry approximation where we simply replace $\Gamma$ with $\Gamma_h$ both in the integration domain and in the surface differential operators. As a side effect integration must be performed elementwise as we cannot evaluate the derivative of $P_h v$ over element faces as $P_h$ is discontinuous. An alternate `dG-style' derivation using Green's formula elementwise over $\Gamma_h$ would result in an additional term of the form
\begin{equation} \label{eq:dg-term}
\left(
( \DPsh v_\thh \cdot \nu_h )^+ 
+
( \DPsh v_\thh \cdot \nu_h )^- 
, w
\right)_{\mcE_h} 
\end{equation}
where $\nu_h^\pm$ are the outward pointing element conormals to the neighboring elements $K^\pm \in \mcK_h$ and $\mcE_h$ is the set of faces in $\mcK_h$.
As $(\DPsh v_\thh \cdot \nu_h )^+ + ( \DPsh v_\thh \cdot \nu_h )^- = 0$ is the natural flux conservation law over element edges, the additional term (\ref{eq:dg-term}) is zero.
\end{rem}

\begin{rem}[Symmetric Formulation] The finite element method for the symmetric formulation is 
obtained by replacing $a_h$ with the form
\begin{equation}\label{def:ah-sym}
a_{h,\text{sym}}(v,w) =(\epsPsh(v_\thh), \epsPsh(w_\thh))_{\mcK_h} + (v_\thh,w_\thh)_\Gammah
\end{equation}
where we include a zeroth order term to avoid having to deal with the non-trivial kernel of the symmetric part of the covariant derivative, i.e., $\ker(\epsilon_{\Gamma_h})$.
\end{rem}

\begin{rem}[The Penalty Term]\label{rem:better-normal}
The choice of normal $\widetilde{n}_h$ in the penalty term $s_h$ depends on
available geometry information. When the
triangulation is constructed from a parametrization of the exact surface, for instance via a CAD model, the exact normal in the nodes is typically available and we can construct $\widetilde{n}_h$ based on nodal interpolation as suggested above. In contrast to this, there are applications such as surface evolution problems where we would typically only have access to a discrete triangulated surface and thus $\widetilde{n}_h = n_h$ is a natural choice.
As we will see in the error estimates below, that choice would, however, not give optimal order convergence.
\end{rem}

\begin{rem}[A Lagrange Multiplier Approach] \label{rem:multiplier}
Another natural approach to enforcing the tangent condition is to use Lagrange multipliers, as employed in \cite{GrJaOlRe17}. The problem is then posed as the following saddle point problem:
Find $\{u,\lambda\}\in V_h \times V_{h,k_u,k_g}$ such that
\begin{alignat}{2}
a_h(u,v) + (\lambda,v_{\widetilde{n}})_\Gammah &= l_h(v) & \qquad &\forall v\in V_h
\\
(u_{\widetilde{n}},\mu)_\Gammah &= 0  & \qquad &\forall \mu\in V_{h,k_u,k_g}
\end{alignat}
where we recall that $V_h = [V_{h,k_u,k_g}]^3$. In the numerical results section we briefly consider this alternate approach.
\end{rem}


\section{Preliminary Results} \label{section:prel-results}
In this section we present preliminary results, which are necessary in the analysis, albeit not directly associated to the vector Laplace problem.
To be able to compare functions defined on the continuous surface $\Gamma$ with functions defined on discrete approximations of $\Gamma$ we collect basic results regarding extension and lifting of functions  in Section~\ref{section:extandlift} and equivalences between norms defined on the respective surfaces  in Section~\ref{section:norm-equivalences}. In Section~\ref{section:non-standard} we present a non-standard geometry approximation result adapted from \cite{LaLa17}, which is required in the proofs of our error estimates below.
A suitable interpolant and properties thereof is given in Section~\ref{section:inpln}.

\subsection{Extension and Lifting of Functions} \label{section:extandlift}

We here summarize basic results concerning 
extension and liftings of functions, and we refer to 
\cite{BuHaLa14} and \cite{De09} for further details.

\paragraph{Extension of Scalar Valued Functions.}
Recalling the definition $v^e = v \circ p$ of the extension 
and using the chain rule we obtain the identity 
\begin{equation}\label{eq:tanderext}
\nablash v^e = B^T \nablas v  
\end{equation}
where 
\begin{equation}\label{Bmap}
B = \Ps(I - \rho \kappa)\Psh : T_x(K)\rightarrow
T_{p(x)} (\Gamma)
\end{equation}
and $\kappa = \nabla \otimes \nabla \rho$ is the curvature tensor defined in 
(\ref{eq:curvature-tensor}).
We note that there is $\delta>0$ such that the uniform 
bound 
\begin{equation}\label{eq:kappa-bound}
\|\kappa \|_{L^\infty(U_\delta(\Gamma))}\lesssim 1
\end{equation}
holds. Furthermore, we show below that 
$B:T_{x}(K) \rightarrow T_{p(x)} (\Gamma)$ is  invertible for 
$h \in (0,h_0]$ with $h_0$ small enough, i.e, there is
 $B^{-1}:   T_{p(x)} (\Gamma) \rightarrow  T_{x}(K)$ such that 
 \begin{equation}
 B B^{-1} = P, \qquad B^{-1} B = P_h
 \end{equation}

\paragraph{Lifting of Scalar Valued Functions.}
The lifting $w^l$ of a function $w$ defined on $\Gamma_h$ 
to $\Gamma$ is defined as the push forward
\begin{equation}
(w^l)^e = w^l \circ p = w \quad \text{on $\Gamma_h$}
\end{equation}
For the derivative it follows that 
\begin{equation}
\nablash w = \nablash (w^l)^e = B^T \nablas (w^l) 
\end{equation}
and thus 
\begin{equation}\label{eq:tanderlift}
 \nablas (w^l) = B^{-T} \nablash w 
\end{equation}

\paragraph{Extension and Lifting of Vector Valued Functions}

We employ component-wise lifting and extension of vector valued functions 
which directly give the identities:
\begin{alignat}{2}
\label{eq:ext-der-vector}
v^e \otimes \nablash &= (v \otimes \nablas) B \qquad &&v \in H^1(\mcK_h^l)
\\
\label{eq:lift-der-vector}
v^l \otimes \nablas &= (v \otimes \nablash) B^{-1}  \qquad &&v \in H^1(\mcK_h)
\end{alignat}


\begin{lem}[Estimates Related to $\boldsymbol B$] \label{lem:Bbounds}
We have the following bounds 
\begin{alignat}{2}
\label{B-uniform-bounds}
  \| B \|_{L^\infty(\Gamma_h)} &\lesssim 1,
 \qquad &\| B^{-1} \|_{L^\infty(\Gamma)} 
 &\lesssim 1
\\
\label{eq:B-PPh}
\| \Ps\Psh - B \|_{L^\infty(\Gamma)} 
 &\lesssim h^{k_g+1}, 
 \qquad\quad
 &\| \Psh\Ps -  B^{-1} \|_{L^\infty(\Gamma_h)} 
 &\lesssim h^{k_g+1}
\end{alignat}
For the surface measures on $\Gamma$ and $\Gammah$ 
we have the identity  
\begin{equation}\label{eq:measure}
d \Gamma = \left| P_h (p \otimes \nablas) \right|
= \left|B\right| d \Gammah
\end{equation}
where $\left|B\right| =| \mathrm{det}(B)|$
and we have the estimates
\begin{equation}\label{eq:B-detbound}
\left\| 1 - |B| \right\|_{L^\infty(\Gamma_h)} \lesssim h^{k_g+1}, 
\qquad \left\| |B| \right\|_{L^\infty(\Gamma_h)} \lesssim 1,
\qquad \left\| |B|^{-1} \right\|_{L^\infty(\Gamma_h)} \lesssim 1
\end{equation}
\end{lem}
\begin{proof}
{\bf Estimates (\ref{eq:B-PPh}).} The first estimate follows directly from (\ref{eq:geombounds}) and the bound \eqref{eq:kappa-bound},
\begin{equation}\label{eq:B-PPh-a}
\|B - PP_h\|_{L^\infty(\Gammah)} 
= \| \rho \kappa P_h \|_{L^\infty(\Gammah)}  \lesssim h^{k_g+1}
\end{equation}
For the second estimate we first note that for $\xi \in T_x(K)$ we have  the 
bound
\begin{equation}\label{eq:B-coer}
\| \xi \|_{\IR^3}  \lesssim \| B \xi \|_{\IR^3} 
\end{equation} 
since
\begin{align}
\| B \xi \|_{\IR^3} &= \|P(I- \rho \kappa) P P_h \xi \|_{\IR^3}  
\\
&\gtrsim  \|PP_h \xi \|_{\IR^3} - h^{k_g+1} \|P_h \xi \|_{\IR^3}  
\\
&=  \|(I-Q)P_h \xi \|_{\IR^3}  - h^{k_g+1} \|P_h \xi \|_{\IR^3} 
\\
&\gtrsim  \|P_h \xi \|_{\IR^3}  -  h^{k_g} \|P_h \xi \|_{\IR^3}  - h^{k_g+1} 
\|P_h \xi \|_{\IR^3} 
\\
&\gtrsim  \| \xi \|_{\IR^3} 
\end{align}
for $h\in(0,h_0]$ with $h_0$ small enough. Thus it follows from 
(\ref{eq:B-coer}) that  $B$ is invertible and for 
$\eta \in T_{p(x)}(\Gamma)$ we have the estimate
\begin{align}
\| (B^{-1} - P_h P )  \eta \|_{\IR^3} 
&\lesssim 
 \| B (B^{-1} - P_h P )  \eta \|_{\IR^3}
 \\
 &\lesssim 
 \| (P - B P_h P )  \eta \|_{\IR^3}
 \\
 &= \| (P - B)  \eta \|_{\IR^3}
 \\
 &\lesssim  \| (P - PP_h)\eta \|_{\IR^3} + \|(P P_h - B)  \eta \|_{\IR^3}
 \\
  &\lesssim  \| P Q_h P \eta \|_{\IR^3} + \|(P P_h - B)  \eta \|_{\IR^3}
\\
&\lesssim ( h^{2kg} + h^{k_g + 1} ) \| \eta \|_{\IR^3}
\end{align}
where we first used (\ref{eq:B-coer}) and then 
(\ref{eq:geombounds}) and  (\ref{eq:B-PPh-a}). 
It thus follows that, in the operator norm, 
\begin{equation}
\| (B^{-1} - P_h P ) \|_{\IR^3}  
= 
\sup_{\eta \in T_{p(x)}(\Gamma)} 
 \frac{\| (B^{-1} - P_h P )  \eta \|_{\IR^3} }{\| \eta\|_{\IR^3}} 
\lesssim 
 h^{2kg} + h^{k_g + 1}
 \lesssim 
 h^{k_g + 1}
\end{equation}
for $k_g\geq 1$.

\paragraph{Estimates (\ref{B-uniform-bounds}).} These estimates readily follow by adding and subtracting $P P_h$ respectively $P_h P$, applying the triangle inequality and using the bounds \eqref{eq:B-PPh}.

\paragraph{Estimates (\ref{eq:B-detbound}).} The proof in \cite[Section 3.3]{BuHaLa14} combined with the higher-order geometry bounds \eqref{eq:geombounds} yield these estimates.
\end{proof}

\subsection{Norm Equivalences} \label{section:norm-equivalences}
In order to conveniently deal with extensions and liftings we will 
write $v = v^e$ and $v = v^l$ when there is no risk for confusion. 
In this way we may think of functions as being defined both on $\Gamma$ 
and $\Gammah$ and we can form the sum of function spaces on $\Gamma$ 
and $\Gammah$, for instance, $L^2(\Gamma) + L^2(\Gammah)$ or $H^1_\mathrm{tan}(\Gamma) + V_h$. 
In view of the bounds in Lemma~\ref{lem:Bbounds} and the 
identities (\ref{eq:tanderext}) and (\ref{eq:tanderlift}) 
we obtain the following equivalences for scalar valued functions $v \in H^1(\Gamma) + H^1(\Gammah)$
\begin{align}
\label{eq:normequ}
\| v \|_{L^2(\Gamma)}
\sim \| v \|_{L^2(\Gammah)} \qquad \text{and} \qquad
\| \nablas v \|_{L^2(\Gamma)} &\sim
\| \nablash v \|_{L^2(\Gamma_h)}
\end{align}

\paragraph{Norm Equivalences for Vector Valued Functions.} The above equivalences directly translate to the following equivalences for vector valued functions $v\in H^1(\mcK_h^l) + H^1(\mcK_h)$
\begin{align}
\label{eq:normequ-vec}
\| v \|_{L^2(\mcK_h^l)}
\sim \| v \|_{L^2(\mcK_h)} \qquad \text{and} \qquad
\| v \otimes \nablas \|_{L^2(\mcK_h^l)} &\sim
\| v \otimes \nablash \|_{L^2(\mcK_h)}
\end{align}

\subsection{Non-standard Geometry Approximation} \label{section:non-standard}
To achieve optimal estimates in our proofs below we utilize the $P_h\cdot n$ geometry approximation lemma introduced in \cite[Lemma~3.2]{LaLa17}. We state this lemma below in a slightly extended form and we also supply a proof adapted to higher-order geometry approximations.
\begin{lem}[${\boldsymbol P}_{\boldsymbol h}\cdot {\boldsymbol n}$ Geometry Approximation]
\label{lem:Phn}
For $\chi\in [W_1^1(\Gamma)]^3$ and the approximate surface $\Gammah$ fulfilling the bounds in \eqref{eq:geombounds} it holds
\begin{align}
\label{eq:Phn-est}
\left|
( P_h\cdot n, \chi^e )_\Gammah
\right|
&\lesssim h^{k_g+1}
\| \chi \|_{W_1^1(\Gamma)}
\end{align}
where  $\| \chi \|_{W_1^1(\Gamma)} = \| \chi \|_{L^1(\Gamma)} + \| \chi \otimes\nablas \|_{L^1(\Gamma)}$.
As a consequence the corresponding estimate with $P\cdot n_h$ also holds, i.e.
\begin{align}
\label{eq:Pnh-est}
\left|
( P\cdot n_h, \chi^e )_\Gammah
\right|
&\lesssim h^{k_g+1}
\| \chi \|_{W_1^1(\Gamma)}
\end{align}
\end{lem}
\begin{proof}
\textbf{Estimate (\ref{eq:Phn-est}).\ }
Using Green's formula elementwise  gives the identity
\begin{align}
( P_h \cdot n, \chi^e )_\Gammah
&=
( P_h \nabla\rho, \chi^e )_\Gammah
\\&=
( \nablash \rho, \chi^e )_\Gammah
\\&=
\underbrace{( \rho , \mathrm{tr}(\kappa_h)(n_h\cdot\chi^e) - \nablash \cdot\chi^e )_{\mcK_h}}_{I}
+
\underbrace{( \rho (\nu_h^+ + \nu_h^-), \chi^e )_{\mcE_h}}_{II}
\end{align}
where $\mcE_h$ is the union of the set of (parametrically mapped) faces in $\mcK_h$ and $\nu_h^\pm$ are the conormals of two elements $K_{k_g}^\pm$ sharing a face.

\paragraph{Term $\boldsymbol I$.} By Hölder's inequality and bounds on $n_h$, $\kappa_h$, and $\rho$ we have the estimate
\begin{align}
|I|
&\lesssim
\underbrace{\| \rho \|_{L^\infty(\Gammah)}}_{\lesssim h^{k_g+1}}
\| \mathrm{tr}(\kappa_h)(n_h\cdot\chi^e) - \nablash \cdot\chi^e \|_{L^1(\Gammah)}
\lesssim h^{k_g+1} \| \chi \|_{W_1^1(\Gammah)}
\end{align}
where we recall that $\| \chi^e \|_{W_1^1(\Gammah)} = \| \chi^e \|_{L^1(\Gammah)} + \| \chi^e \otimes\nablash \|_{L^1(\mcK_h)}$.

\paragraph{Term $\boldsymbol I \boldsymbol I$.} Using Hölder's inequality and a trace inequality give
\begin{align}
|II|
&\lesssim
\underbrace{\| \rho \|_{L^\infty(\mcE_h)}}_{\lesssim \| \rho \|_{L^\infty(\Gammah)}}
\| \nu_h^+ + \nu_h^- \|_{L^\infty(\mcE_h)}
\| \chi^e \|_{L^1(\mcE_h)}
\\&\lesssim
\underbrace{\| \rho \|_{L^\infty(\Gamma_h)}}_{\lesssim h^{k_g+1}}
\| \nu_h^+ + \nu_h^- \|_{L^\infty(\mcE_h)}
\underbrace{
\left(
h^{-1}\| \chi^e \|_{L^1(\mcK_h)}
+
\| \chi^e \otimes\nablash \|_{L^1(\mcK_h)}
\right)
}_{\lesssim h^{-1}\| \chi^e \|_{W_1^1(\Gammah)}}
\\&\lesssim
h^{k_g}
\| \nu_h^+ + \nu_h^- \|_{L^\infty(\mcE_h)}
\| \chi^e \|_{W_1^1(\Gammah)}
\end{align}
where it now remains to estimate the conormal term.
Letting $\nu^\pm$ denote the conormal to the lifted triangle $(K_{k_g}^\pm)^l\subset\Gamma$ we note that $\nu^+ + \nu^- = 0$. Hence, by subtracting $\nu^+ + \nu^-$, using the triangle inequality and bounds on the
conormal approximation we have
\begin{align}
\| \nu_h^+ + \nu_h^- \|_{L^\infty(\mcE_h)}
\leq
\| \nu_h^+ - \nu^+ \|_{L^\infty(\mcE_h)}
+
\| \nu_h^- - \nu^- \|_{L^\infty(\mcE_h)}
\lesssim h^{k_g}
\end{align}
and thus $II\lesssim h^{2k_g} \| \chi^e \|_{W_1^1(\Gammah)}$.

The proof of \eqref{eq:Phn-est} is completed by the equivalence $\| \chi^e \|_{W_1^1(\Gammah)} \sim \| \chi \|_{W_1^1(\Gamma)}$ which holds
in view of the bounds in Lemma~\ref{lem:Bbounds}.

\paragraph{Estimate (\ref{eq:Pnh-est}).}
This estimate readily follows by noting that
$P_h \cdot n + P \cdot n_h = (n + n_h)(1 - n\cdot n_h)$ where
$\| 1 - n\cdot n_h \|_{L^\infty(\Gammah)} \lesssim h^{k_g+1}$.
Thus, by adding and subtracting suitable terms,
applying the triangle inequality and Hölder's inequality,
we may, without loosing approximation order,
move over to a term on the form $|(P_h \cdot n,\chi^e)_{\Gamma_h}|$ to which we apply \eqref{eq:Phn-est}.
\end{proof}

\begin{rem}
Clearly, Lemma~\ref{lem:Phn} also holds for $\chi\in H^1(\Gamma)$ as we by the Cauchy--Schwarz inequality have the bound
\begin{align}
\| \chi \|_{W^1_1(\Gamma)} \leq \sqrt{\left|\Gamma\right|}\left(\| \chi \|_{L^2(\Gamma)} + \| \chi\otimes\nablas \|_{L^2(\Gamma)} \right) \leq \sqrt{2 \left|\Gamma\right|} \| \chi \|_{H^1(\Gamma)}
\end{align}
\end{rem}

\subsection{Interpolation} \label{section:inpln}

We now turn to defining the interpolation operator $\pi_{h,1}:[L^2(\Gamma_{h,1})]^d \mapsto [V_{h,k_u,1}]^d$ on the facet surface triangulation
as a Scott--Zhang interpolation operator, see the classical reference \cite{ScZh90} and the extension to triangulated surfaces in \cite{CaDe15}.
The construction of this interpolation operator is as follows.
Let each Lagrange node $x_i$ be associated with a domain $S_i$ which is a triangle $S_i = K\in\mcK_{h,1}$ if $x_i$ is interior to $K$ or a face $S_i = E$ if $x_i$ is interior to $E$. For nodes contained in several faces, i.e., nodes at triangle vertices, $S_i=E$ may be arbitrarily chosen among the faces containing $x_i$.
Let $\{\varphi_i : \mcK_{h,1}\rightarrow\mathbb{R}^d\}$ be the Lagrange basis for $[V_{h,k_u,1}]^d$  and let $\{\psi_i\}$ be the dual basis such that $(\varphi_j,\psi_k)_{S_i}=\delta_{jk}$ where $x_j,x_k$ are nodes associated with $S_i$. The nodal values are then  defined by
\begin{align}
\pi_{h,1} v (x_i) = (v,\psi_i)_{S_i}
\end{align}
and we readily see that $\pi_{h,1}$ is a projection by expanding any $v\in [V_{h,k_u,1}]^d$ in the Lagrange basis.
For $K_1 \in \mcK_{h,1}$ the following interpolation estimate then holds
\begin{align}
\| v - \pi_{h,1}v \|_{H^m(K_1)} \lesssim h^{s-m}\| v \|_{H^s(\mcN_h^l(K_1))},
\quad m\leq s \leq k_u+1, \quad s \geq 1
\end{align}
where $\mcN_h^l(K_1)$ is the patch of elements in $\mcK_{h,1}$ which are node neighbors to $K_1$ lifted onto the exact surface $\Gamma$, see \cite[Theorem~3.2]{CaDe15} for proof.

Next we define the interpolant $\pi_{h,k_g}:[L^2(\Gamma_h)]^d \rightarrow [V_{h,k_u,k_g}]^d$ as follows
\begin{equation} \label{eq:def-phkg}
\pi_{h,k_g} v^e|_{K_{k_g}} = (\pi_{h,1} v^e) \circ G_{K,k_g,1} 
\end{equation}
where $G_{K,k_g,1} = F_{K,1} \circ F_{K,k_g}^{-1}: K_{k_g} \rightarrow K_1$ is a bijection from the curved triangle $K_{k_g}$ 
to the corresponding flat triangle $K_1$.
The interpolant $\pi_{h,k_g}$ inherits the projection property from $\pi_{h,1}$.
As the the higher-order mesh $\mcK_{h,k_g}$ is constructed as the Lagrange interpolant of the closest point mapping $p$ on the facet mesh $\mcK_{h,1}$ we directly get uniform $L^\infty(\mcK_{h,1})$ bounds on $G_{K,k_g,1}^{-1}$ and its derivatives from standard interpolation theory and $W_\infty^{k_u+1}(U_\delta)$ bounds on $p$. This yields the inequality
\begin{align}
\| v^e - \pi_{h,k_g}v^e \|_{H^m(K_{k_g})} 
&\lesssim 
\| v^e - \pi_{h,1} v^e \|_{H^m(K_{1})}  , \quad 0 \leq m \leq k_u + 1
\end{align}
and we thus conclude that the estimate
\begin{equation}\label{eq:interpol}
\| v^e - \pi_{h,k_g}v^e \|_{H^m(K_{k_g})} 
\lesssim h^{s-m} \| v \|_{H^s(\mcN_h^l(K_1))},
\quad 0 \leq m\leq s \leq k_u+1, \quad s \geq 1
\end{equation}
holds for all $K_{k_g}\in \mcK_{h,k_g}$.

When appropriate we simplify the notation and write $\pi_h = \pi_{h,k_g}$.

\begin{rem}[Choice of interpolant]
The choice of interpolant here is rather arbitrary albeit the Scott-Zhang interpolant is a suitable choice as we are interpolating $L^2$ functions. If we were to assume continuity of all functions to be interpolated it is possible to use the Lagrange interpolant instead.
As the present work is on closed surfaces, there is no specific need for the special construction in the Scott-Zhang interpolant for satisfying essential boundary conditions.
\end{rem}

\begin{lem}[Super-approximation and super-stability]
For discrete functions $v \in V_h$ and $\chi \in [W^{k_u+1}_\infty(\Gamma)]^3$ is holds
\begin{align}
\label{eq:superapproximationH1H1}
\left\| \nablash (I - \pi_{h,k_g}) (\chi^e \cdot v) \right\|_{\mcK_h}
&\lesssim
h \| \chi \|_{W^{k_u+1}_\infty(\Gamma)} \| v \otimes \nablas \|_{\Gamma_h}
\\
\label{eq:superapproximationH1L2}
\left\| \nablash (I - \pi_{h,k_g}) (\chi^e \cdot v) \right\|_{\mcK_h}
&\lesssim
\| \chi \|_{W^{k_u+1}_\infty(\Gamma)} \| v \|_{\Gamma_h}
\end{align}
Furthermore we also have the $L^2$ stability estimate
\begin{align}
\label{eq:super-stab}
\left\| \pi_{h,k_g} (\chi^e \cdot v) \right\|_\Gammah
\lesssim
\left\| \chi^e \cdot v \right\|_\Gammah
\end{align}
\end{lem}
\begin{rem}
We call \eqref{eq:super-stab} `super-stability' as the standard $L^2$ stability of the Scott--Zhang interpolant also includes a $H^1$ term on the right hand side.
\end{rem}
\begin{proof}
Let $I_{h,k_g}:C(\Gamma_h)\rightarrow V_{h,k_u,k_g} $ denote the Lagrange interpolant.
As $\pi_{h,k_g}$ is a projection on $V_{h,k_u,k_g}$ the operator $(I - \pi_{h,k_g})I_{h,k_g}$
is zero. Subtracting this zero operator and applying interpolation estimate \eqref{eq:interpol} give
\begin{align}
\| (I - \pi_{h,k_g}) (\chi^e \cdot v) \|_{H^1(K_{k_g})}
&=
\| (I - \pi_{h,k_g})(I - I_{h,k_g}) (\chi^e \cdot v) \|_{H^1(K_{k_g})}
\\
&\lesssim
\| (I - I_{h,k_g}) (\chi^e \cdot v) \|_{H^1(\mcN_h^l(K_1))}
\\
&\lesssim
\sum_{K'\in \mcN_h(K_1)}
\| (I - I_{h,k_g}) (\chi^e \cdot v) \|_{H^1((K')^l)}
\\
&\lesssim
\sum_{K'\in \mcN_h(K_1)}
\| (I - I_{h,k_g}) (\chi^e \cdot v) \|_{H^1(K')}
\\
&\lesssim
\sum_{K'\in \mcN_h(K_1)}
h^{k_u}
\| \chi^e \cdot v \|_{H^{k_u + 1}(K')}
\\&\lesssim
\sum_{K'\in \mcN_h(K_1)}
h^{k_u}
\| \chi^e \|_{[W^{k_u + 1}_\infty(K')]^3}
\| v \|_{H^{k_u + 1}(K')}
\end{align}
where we in the last inequality use the assumption $\chi \in [W^{k_u+1}_\infty(\Gamma)]^3$.
The proof is finalized by the following estimates
\begin{equation}
h^{k_u}
\| v \|_{H^{k_u + 1}(K')}
=
h^{k_u}
\| v \|_{H^{k_u}(K')}
\lesssim
h \| v \|_{H^1(K')}
\lesssim
\| v \|_{K'}
\end{equation}
where we in the equality use that the $(k_u+1)$:th derivative of a polynomial of order $k_u$ is zero and we in the inequalities use two inverse estimates yielding \eqref{eq:superapproximationH1H1} and \eqref{eq:superapproximationH1L2}, respectively.

The stability estimate \eqref{eq:super-stab} follows by mapping each element in $\mcK_h$ associated to a patch of nodal neighbors $\mcN_h(K_1)$ onto a flat reference element $\widehat{K}$, noting that the estimate
\begin{align}
\left\| \pi_{h,1} ( \chi^e\circ F_{K_1,k_g}^{-1}  \cdot v \circ F_{K_1,k_g}^{-1} ) \right\|_{\widehat{K}}
\lesssim
\sum_{K' \in \mcN_h(K_1)}
\left\| \chi^e\circ F_{K',k_g}^{-1}  \cdot v \circ F_{K',k_g}^{-1} \right\|_{\widehat{K}}
\end{align}
holds due to the finite dimensionality of $V_h|_{K_{k_g}}$ and the construction of the interpolant,
and finally mapping back onto the parametrically mapped triangles in $\mcK_h$.
\end{proof}

\section{Error Estimates} \label{section:error-estimates}

In this section we prove a series of theoretical results leading up to the main a priori error estimates. First, in Section~\ref{section:err-norms} we define the energy norm, and in Section~\ref{section:err-cont-coer} we prove coercivity and continuity for the method. In Section~\ref{section:err-interpolation} we show interpolation estimates in the energy norm and in a corresponding continuous norm. Bounds on errors stemming from the geometry approximation via approximate surface differential operators and the change of measure are proven in Section~\ref{section:err-geom}. A Poincaré inequality on the discrete surface and certain $H^1$ type bounds are shown in Section~\ref{section:err-basic-lem}.
Last, in Section~\ref{section:err-est}, we prove the main a priori error estimates; in energy norm (Theorem~\ref{thm:errorest-energy}) and in $L^2$ norm (Theorem~\ref{thm:error-L2}).

\subsection{Norms} \label{section:err-norms}
For a continuous semidefinite form $\alpha$ on 
a Hilbert space $H$ we let $\| v \|^2_\alpha = \alpha(v,v)$ be the 
seminorm associated with $\alpha$ on $H$. We also use the standard 
notation 
\begin{gather} \label{eq:energy-norm}
\tn v \tn_h^2 = \| v \|^2_{A_h} = \| v \|^2_{a_h} + \| v \|^2_{s_h}
=
\| \Dsh v_\thh \|^2_{\mcK_h} + \beta h^{-2} \| v_{\widetilde{n}_h} \|^2_{\Gammah}
\end{gather}
for the discrete energy norm on $H^1(\mcK_h)$.

\begin{rem}[Energy Norm]
That (\ref{eq:energy-norm}) is an actual norm on $H^1(\mcK_h)$ will become evident by the analysis below,
in particular by the Poincaré type inequality in Lemma~\ref{lem:poincare-discrete}.

\end{rem}

\subsection{Coercivity and Continuity} \label{section:err-cont-coer}

\begin{lem} It holds
\begin{alignat}{2}
\label{eq:continuity-Ah}
\tn v \tn_h^2 &\lesssim A_h(v,v) \qquad &&v \in \Honet + V_h
\shortintertext{and}
\label{eq:coercivity-Ah}
| A_h(v,w) | &\lesssim \tn v \tn_h \tn w \tn_h \qquad &&v,w \in \Honet + V_h
\end{alignat}
\end{lem}
\begin{proof} The first inequality holds by definition since $\tn v \tn_h = \| v \|_{A_h}$. The second inequality directly follows by the Cauchy--Schwarz inequality since $A_h(\cdot,\cdot)$ 
is an inner product.
\end{proof}

\subsection{Interpolation} \label{section:err-interpolation}

\begin{lem}[Interpolation in Energy Norm]
For $v\in H_{\mathrm{tan}}^{k_u+1}(\Gamma)$
we have the following interpolation error estimate in the energy norm
\begin{align}
\label{eq:interpol-energy}
\tn v - \pi_h v \tn_h &\lesssim h^{k_u} \| v \|_{H_{\mathrm{tan}}^{k_u+1}(\Gamma)}
\end{align}
and for $\phi\in H_{\mathrm{tan}}^{2}(\Gamma)$ we have the following interpolation estimate in the corresponding continuous norm
\begin{align}
\label{eq:interpol-a}
\| \phi - \pi_h \phi \|_a &\lesssim h \| \phi \|_{H_{\mathrm{tan}}^{2}(\Gamma)}
\end{align}
\end{lem}
\begin{proof} {\bf Estimate (\ref{eq:interpol-energy}).}
This estimate is obtained by the calculation
\begin{align}
\tn v - \pi_h v \tn_h &\lesssim  \|\Dsh (P_h(v- \pi_h v) ) \|_{\mcK_h}
+ h^{-1}\|\widetilde{n}_h\cdot (v - \pi_h v )\|_\Gammah
\\
&\lesssim  \|\Dsh (v- \pi_h v) \|_{\Gammah}
+ \| \nh\cdot(v-\pi_h v) \|_\Gammah
+ h^{-1}\|\widetilde{n}_h\cdot (v - \pi_h v )\|_\Gammah 
\\
&\lesssim\|(v- \pi_h v )\otimes \nablash \|_{\Gammah}
+ \underbrace{(1 + h^{-1})}_{\lesssim h^{-1}}\| v - \pi_h v \|_\Gammah 
\\
&\lesssim  h^{k_u}  \| v \|_{H^{k_u+1}(\Gamma)}
\\
&\lesssim  h^{k_u}  \| v \|_{H_{\mathrm{tan}}^{k_u+1}(\Gamma)}
\end{align}
where we used the interpolation error estimate (\ref{eq:interpol}) 
and at last Lemma \ref{lem:H-Ht-equivalence} to pass to the Sobolev 
norm based on covariant derivatives.

\paragraph{Estimate (\ref{eq:interpol-a}).}
Analogously to the previous calculation we have
\begin{align}
\| \phi - \pi_h \phi \|_a
&\lesssim
\|\Ds (P(\phi- \pi_h \phi)) \|_{\Gamma}
\\&\lesssim 
\|\Ds (\phi- \pi_h \phi) \|_{\Gamma}
+ \| n \cdot (\phi - \pi_h \phi) \|_\Gamma
\\&\lesssim
\|(\phi- \pi_h \phi)\otimes\nablas \|_{\Gamma}
+ \| \phi - \pi_h \phi \|_\Gamma
\\&\lesssim  (h + h^{2})  \| v \|_{H^{2}(\Gamma)}
\\
&\lesssim  h \| v \|_{H_{\mathrm{tan}}^{2}(\Gamma)}
\end{align}
where we used the interpolation error estimate (\ref{eq:interpol}) 
and at last Lemma \ref{lem:H-Ht-equivalence} to pass to the Sobolev 
norm based on covariant derivatives.
\end{proof}

\subsection{Estimates of Geometric Errors} \label{section:err-geom}

Define the geometry error forms
\begin{equation}\label{eq:quad-forms}
Q_a(v,w) = a(v,w) - a_h(v,w),\qquad Q_l(v)= l(v) - l_h(v)
\end{equation}
Before proceeding with the estimates we formulate a useful lemma
\begin{lem}[Operator Difference]\label{lem:Ds-Dsh}
For $h \in (0,h_0]$, with $h_0$ small enough, and $v \in H^1(\mcK_h)$ 
there is a constant such that 
\begin{equation}\label{eq:Ds-bound}
\| (\DPs v^l_t)^e  - \DPsh v_\thh \|_{\mcK_h} 
\lesssim 
h^{k_g} \left( \| v  \|_{H^1(\Gammah)} + h^{-1} \|v_\nh\|_\Gammah \right)
\end{equation}
\end{lem}
\begin{proof} Decomposing $v$ and $v^l$ into tangent and normal components 
on $\Gammah$ and $\Gamma$, 
\begin{equation}\label{eq:normal-tangent-split}
v = \Psh v + \Qsh v = v_\thh + v_\nh \nh, 
\qquad
v^l = \Ps v^l + \Qs v^l = v^l_t + v^l_n n
\end{equation}
we obtain the identities 
\begin{align}\label{eq:normal-tangent-split-derivative}
\DPsh v_\thh &=  \Psh (v \otimes \nablash) - v_{n_h}  \kappa_h, \qquad
\DPs v^l_t =  \Ps ( v^l \otimes \nablas ) - v^l_{n}  \kappa
\end{align}
and thus
\begin{align}
(\DPs v^l_t)^e  - \DPsh v_\thh
&=
\underbrace{((\Ps (v^l \otimes \nablas)^e)  - \Psh (v \otimes \nablash)}_{I} 
- 
\underbrace{( (v^l_n \kappa)^e -  (v_\nh \kappa_h) )}_{II}
\end{align}
\paragraph{Term $\bfI$.} Using (\ref{eq:lift-der-vector}) and 
adding and subtracting suitable terms we obtain
\begin{align}
(\DPs v^l)^e  - \DPsh v
&=
(\Ps v^l \otimes \nablas)^e  - \Psh v \otimes \nablash
\\
&=
\Ps v \otimes \nablash B^{-1}  - \Psh v\otimes \nablash
\\
&=
(\Ps - \Psh) v \otimes \nablash B^{-1}  
\\ \nonumber
&\qquad +
\Psh v \otimes \nablash (B^{-1} -\Ps) 
\\ \nonumber
&\qquad +
\Psh v\otimes \nablash (\Ps - \Psh)
\end{align}
By the triangle inequality we then have the estimate
\begin{align}
\label{eq:DGammaDiff-first}
\| (\DPs v^l)^e  - \DPsh v \|_\Gammah
&\lesssim 
\|\Ps - \Psh\|_{L^\infty(\Gammah)} \|v \otimes \nablash\|_\Gammah 
\|B^{-1}\|_{L^\infty(\Gammah)}  
\\ \nonumber
&\qquad +
\|\Psh\|_{L^{\infty}(\Gammah)} 
\|v \otimes \nablash\|_\Gammah  
\|B^{-1} -\Ps\|_{L^\infty(\Gammah)} 
\\ \nonumber
&\qquad +
\|\Psh\|_{L^\infty(\Gammah)} 
\| v\otimes \nablash \|_{\Gammah} 
   \|\Ps - \Psh\|_{L^\infty(\Gammah)}
\\
&\lesssim 
h^{k_g} \| v \otimes \nablash \|_\Gammah 
\label{eq:DGammaDiff-last}
\end{align}
where we used the bounds (\ref{eq:B-PPh}), 
(\ref{eq:B-detbound}), and $\|\Ps - \Psh \|_{L^\infty(\Gammah)} 
\lesssim h^{k_g}$. 

\paragraph{Term $\bfI\bfI$.} Adding and subtracting suitable terms we obtain
\begin{align}
\| v_n \kappa - v_\nh \kappa_h \|_\Gammah 
&\leq
\| (v_n - v_\nh) \kappa \|_\Gammah +  \| v_\nh (\kappa - \kappa_h) \|_\Gammah
\\
&\lesssim 
h^{k_g} \| v \|_\Gammah +  h^{k_g-1} \| v_\nh  \|_\Gammah
\\
&= 
h^{k_g} ( \| v \|_\Gammah + h^{-1} \| v_\nh \|_\Gammah )
\end{align}

\paragraph{Conclusion.}
Collecting the estimates we obtain
\begin{align}
\|(\DPs v^l_t)^e  - \DPsh v_\thh \|_{\mcK_h}
&\lesssim h^{k_g} ( \| v \otimes \nablash \|_\Gammah +  \| v \|_\Gammah + h^{-1} \| v_\nh \|_\Gammah )
\\
&= 
h^{k_g} (  \| v \|_{H^1(\Gammah)} + h^{-1} \| v_\nh \|_\Gammah )
\end{align}
which is the desired bound.
\end{proof}

\begin{lem}[Geometric Errors]\label{lem:quad} For $v,w \in H^1(\Gammah)$, $h \in (0,h_0]$ with $h_0$ 
small enough, there are constants such that
\begin{align}\label{eq:quad-a}
Q_a(v,w) &\lesssim h^{k_g} (\| v \|_{H^1(\Gammah)} + h^{-1} \| v_\nh \|_\Gammah )  
(\| w \|_{H^1(\Gammah)} + h^{-1} \| w_\nh \|_\Gammah )
\\ \label{eq:quad-l}
Q_l(w) & \lesssim (h^{k_p} + h^{k_g + 1}) \|f \|_\Gamma \| w \|_{\Gammah}
\end{align}
and also, for higher regularity tangential functions $\psi,\phi \in H^2_\mathrm{tan}(\Gamma)$ it holds
\begin{align}\label{eq:quad-a-tan}
Q_a(\psi,\phi) &\lesssim h^{k_g+1} \|\psi\|_{H^2_\mathrm{tan}(\Gamma)}\|\phi\|_{H^2_\mathrm{tan}(\Gamma)}
\\ \label{eq:quad-l-tan}
Q_l(\phi) & \lesssim (h^{2 k_p} + h^{k_g+1}) \|f \|_\Gamma \| \phi \|_{\Gamma}
\end{align}
\end{lem}
\begin{proof} {\bf Estimate (\ref{eq:quad-a}).}  
%
%
Changing domain of integration 
from $\Gamma$ to $\Gammah$ in the first term and adding and subtracting 
suitable terms we obtain
\begin{align}
Q_a(v,w)
&=
(\DPs v^l_t , \DPs w^l_t )_\Gamma 
- (\DPsh v_{t_h} , \DPsh w_{t_h} )_{\mcK_h}
\\
&=
((\DPs v^l_{t})^e ,( \DPs w^l_{t} )^e \muh)_{\mcK_h} 
- (\DPsh v_\thh , \DPsh w_\thh )_{\mcK_h}
\\
&=
(\DPs v^l_t , \DPs w^l_t  (\muh-1))_{\mcK_h}
\\ \nonumber
& \qquad + 
(\DPs v^l_t , \DPs w^l_t )_{\mcK_h}
- (\DPsh v_\thh, \DPsh w_\thh )_{\mcK_h}
\end{align}
Here the first term on the right hand is directly estimated using (\ref{eq:B-detbound}), 
\begin{align}
(\DPs v^l_t , \DPs w^l_t  (\muh-1))_{\mcK_h} 
&\leq
\| \DPs v^l_t \|_{{\mcK_h}} \|\DPs w^l_t \|_{\mcK_h}  
\| 1 - \muh \|_{L^\infty(\Gammah)} 
\\
&\lesssim 
h^{k_g+1} \| v  \|_{H^1(\Gammah)} 
\|w \|_{H^1(\Gammah)}
\end{align}
where we used (\ref{eq:normal-tangent-split-derivative}) to conclude that  
\begin{align}
\| \DPs v^l_t \|_{{\mcK_h}}  &= \| \Ps ( v^l \otimes \nablas ) - v^l_n \kappa \|_{\Gammah} 
\\
&\lesssim \| v^l \otimes \nablas \|_\Gammah  + \| v_n  \|_{\Gammah} 
\\
&\lesssim  \| v \otimes \nablash \|_\Gammah + \| v \|_\Gammah 
\\
&\lesssim \| v \|_{H^1(\Gammah)} 
\end{align}

For the second term we add and subtract suitable terms and employ 
(\ref{eq:Ds-bound}), 
\begin{align}\nonumber
&(\DPs v^l_t, \DPs w^l_t )_{\mcK_h} 
- (\DPsh v_\thh , \DPsh w_\thh )_{\mcK_h}
\\\label{eq:hireg-same}
&\qquad 
=
(\DPs v^l_t - \DPsh v_\thh, \DPs w^l_t )_{\mcK_h}
+ (\DPsh v_\thh , \DPs w^l_t - \DPsh w_\thh )_{\mcK_h}
\\
&\qquad 
\leq 
\| \DPs v^l_t - \DPsh v_\thh \|_{\mcK_h}  \|\DPs w^l_t \|_{\mcK_h}
+ \|\DPsh v_\thh \|_{\mcK_h}  \|\DPs w^l_t - \DPsh w_\thh \|_{\mcK_h}
\\
&\qquad 
\leq 
h^{k_g} (\| v \|_{H^1(\Gammah)} + h^{-1} \| v_\nh \|_\Gammah )  
 \|\DPs w^l_t \|_{\mcK_h}
 \\ \nonumber
 &\qquad \qquad
+ h^{k_g} \|\DPsh v_\thh \|_{\mcK_h}  
( \| w \|_{H^1(\Gammah)} + h^{-1} \| w_\nh \|_\Gammah )
\\
&\qquad 
\leq 
h^{k_g} (\| v \|_{H^1(\Gammah)} + h^{-1} \| v_\nh \|_\Gammah )   
( \| w \|_{H^1(\Gammah)} + h^{-1} \| w_\nh \|_\Gammah )
\end{align}
Combining the estimates we arrive at 
\begin{align}
Q_a(v,w) &\lesssim h^{k_g} (\| v \|_{H^1(\Gammah)} + h^{-1} \| v_\nh \|_\Gammah )   
( \| w \|_{H^1(\Gammah)} + h^{-1} \| w_\nh \|_\Gammah )
\end{align}

\paragraph{Estimate (\ref{eq:quad-l}).} Changing domain of integration from $\Gamma$ to $\Gammah$ and adding and subtracting suitable terms we obtain
\begin{align}
Q_l(v) &= (f,v^l_t)_\Gamma - (f^e,v_\thh)_\Gammah
\\
&= (|B|f^e,P v)_\Gammah - (f^e,v_\thh)_\Gammah
\\
&= (|B|f^e,(P - \Psh)v )_\Gammah + ((|B|-1) f^e,v_\thh)_\Gammah
\\
&\lesssim
\| \Ps - \Psh \|_{L^\infty(\Gammah)} \| f^e \|_{\Gammah} \| v \|_{\Gammah}
+
\left\| |B|-1 \right\|_{L^\infty(\Gammah)} \| f^e \|_{\Gammah} \| v_\thh \|_{\Gammah}
\\
&\lesssim
(h^{k_p} + h^{k_g+1}) \| f \|_{\Gamma} \| v \|_{\Gammah}
\end{align}
where we used (\ref{eq:B-detbound}) followed by the norm 
equivalence (\ref{eq:normequ-vec}).

\paragraph{Estimate (\ref{eq:quad-a-tan}).}
For $\psi,\phi\in H^2_\mathrm{tan}(\Gamma)$ we
have $\psi^e,\phi^e \in H^1(\Gammah)$ whereby estimate \eqref{eq:Ds-bound} holds. Combined with the bound $\| \phi \|_{H^1(\Gammah)} + h^{-1} \|\phi_{n_h}\|_{\Gammah} \lesssim \| \phi \|_{H^1_\mathrm{tan}(\Gamma)}$, which we prove in Lemma~\ref{lem:quad-norm-bounds} below, we have the basic estimate
\begin{align}\label{eq:DDiff-phi}
\| (\DPs \phi)^e  - \DPsh \phi^e_\thh \|_{\mcK_h} 
\lesssim 
h^{k_g} \| \phi \|_{H^1_\mathrm{tan}(\Gamma)}
\end{align}

Without loosing the desired approximation order of $h^{k_g+1}$, we may follow the proof of estimate \eqref{eq:quad-a}, combined with the bound $\|\phi^e\|_{H^1(\Gammah)} \lesssim \|\phi\|_{H_\mathrm{tan}^1(\Gamma)}$, until \eqref{eq:hireg-same} where it remains to bound the term
\begin{align}\label{eq:geom-approx-remaining-term}
(\DPs^e \psi_t - \DPsh \psi^e_\thh, \DPs^e \phi_t )_{\mcK_h}
+ (\DPsh \psi^e_\thh , \DPs^e \phi_t - \DPsh \phi^e_\thh )_{\mcK_h}
\end{align}
For the second integral we by adding and subtracting suitable terms, applying the Cauchy--Schwarz inequality and estimate \eqref{eq:DDiff-phi} have
\begin{align}
(\DPsh \psi^e_\thh , \DPs^e \phi_t - \DPsh \phi^e_\thh )_{\mcK_h}
&\lesssim
h^{2k_g}
\| \psi \|_{H^1_\mathrm{tan}(\Gamma)}
\| \phi \|_{H^1_\mathrm{tan}(\Gamma)}
\\ \nonumber
&\qquad + (\DPs^e \psi_t , \DPs^e \phi_t - \DPsh \phi^e_\thh )_{\mcK_h}
\end{align}
where we note that the remaining integral is transpose symmetric to the first term in \eqref{eq:geom-approx-remaining-term} and thus the following analysis will hold for both these terms.
As $\psi,\phi$ are tangential we have the simplification
\begin{align}
( \Ds^e \psi_t , \Ds^e \phi_t - \Dsh \phi^e_{t_h} )_{\mcK_h}
&=
( \Ds^e \psi , \Ds^e \phi - \Dsh \phi^e_{t_h} )_{\mcK_h}
\\&=
\underbrace{
( \Ds^e \psi , \Ds^e \phi - \Dsh \phi^e )_\Gammah
}_{I}
+
\underbrace{
( \Ds^e \psi , \kappa_h \phi^e_{n_h} )_\Gammah
}_{II}
\end{align}
where the identity \eqref{eq:DSh-vt-id} is used to rewrite $\Dsh \phi^e_{t_h}$ in the second equality.

\paragraph{Term $\boldsymbol I$.}
We begin by expressing $\Dsh \phi^e$ in terms of $\Ds \phi$.
By the closest point extension we have $\phi^e \otimes \nabla = \phi \otimes \nablas$ which yields
the identity
\begin{align}
\Dsh \phi^e
&=
P_h (\phi^e \otimes \nabla) P_h
\\&=
P_h (\phi \otimes \nablas)^e P_h
\\&=
P_h (\underbrace{P (\phi \otimes \nablas)}_{=\Ds\phi})^e P_h
+
P_h ((n\otimes n) (\phi \otimes \nablas))^e P_h
\end{align}
and this allows us to decompose term $I$ into the following two terms
\begin{align}
I
&=
\underbrace{
( \Ds^e \psi , \Ds^e \phi - P_h \Ds^e\phi P_h  )_\Gammah
}_{I_1}
-
\underbrace{
( \Ds^e \psi , P_h ((n\otimes n) (\phi \otimes \nablas))^e P_h )_\Gammah
}_{I_2}
\end{align}
First we consider $I_1$.
Expanding the projections $P_h = I - n_h\otimes n_h$ and recalling that $\Ds\psi,\Ds\phi$ are tangential we by the bound on $P\cdot n_h$ readily get
\begin{align}
|I_1|
&=
|( \Ds^e \psi , \Ds^e \phi - P_h \Ds^e \phi P_h  )_\Gammah|
\\
&=
|( \Ds^e \psi , (n_h\otimes n_h) \Ds^e\phi + \Ds^e\phi (n_h\otimes n_h) - (n_h\otimes n_h) \Ds^e\phi (n_h\otimes n_h)  )_\Gammah|
\\
&\lesssim
(\underbrace{\| P\cdot n_h \|_{L^\infty(\Gammah)}}_{\lesssim h^{k_g}})^2 \| \Ds^e \psi \|_{L^2(\Gammah)} \| \Ds^e \phi \|_{L^2(\Gammah)}
\\
&\lesssim
h^{2k_g}
\| \psi \|_{H^1_{\mathrm{tan}}(\Gamma)}
\| \phi \|_{H^1_{\mathrm{tan}}(\Gamma)}
\end{align}
Next we consider $I_2$. Expanding the rightmost projection $P_h = I - n_h\otimes n_h$ we get two terms where it is sufficient to handle the second term using the bound on $P\cdot n_h$ while for the first term it is necessary to employ the non-standard $P_h\cdot n$ geometry approximation of Lemma~\ref{lem:Phn}.
The calculations follow
\begin{align}
| I_2 |
&\leq
| ( \Ds^e \psi , P_h ((n\otimes n) (\phi \otimes \nablas))^e )_\Gammah |
\\&\qquad\nonumber
+
| ( \Ds^e \psi , P_h ((n\otimes n) (\phi \otimes \nablas))^e n_h\otimes n_h )_\Gammah |
\\
&\lesssim
| ( P_h\cdot n  , \underbrace{(\Ds \psi \cdot (n \cdot (\phi \otimes \nablas)))^e}_{\in H^1(\Gamma) \subset W^1_1(\Gamma)} )_\Gammah |
+
(\underbrace{\| P\cdot n_h \|_{L^\infty(\Gammah)}}_{\lesssim h^{k_g}})^2
\| \Ds \psi \|_{\Gamma} \| \phi\otimes\nablas \|_\Gamma
\\
&\lesssim
h^{k_g+1} \| \Ds \psi \cdot (n \cdot (\phi \otimes \nablas)) \|_{W^1_1(\Gamma)}
+
h^{2 k_g} \| \psi \|_{H^1_{\mathrm{tan}}(\Gamma)}
\| \phi \|_{H^1(\Gamma)}
\\
&\lesssim
h^{k_g+1}
\| \psi \|_{H^2(\Gamma)}
\| \phi \|_{H^2(\Gamma)}
+
h^{2 k_g} \| \psi \|_{H^1_{\mathrm{tan}}(\Gamma)}
\| \phi \|_{H^1(\Gamma)}
\end{align}
In the last inequality use the Cauchy--Schwarz inequality on the $W^1_1(\Gamma)$ norm
to achieve a bound in $H^2(\Gamma)$ norms and
we can then move onto covariant derivatives via Lemma~\ref{lem:H-Ht-equivalence}.

\paragraph{Term $\boldsymbol I \boldsymbol I$.}
If $k_g=1$ this term vanishes as $\kappa_h=0$.
If $k_g \geq 2$ we by adding and subtracting the exact curvature tensor have
\begin{align}
II
&=
( \Ds \psi , (\kappa_h-\kappa) (\phi \cdot n_h) )_\Gammah
+
( \Ds \psi , \kappa (\phi \cdot n_h) )_\Gammah
\\
&=
\underbrace{
( \Ds \psi , (\kappa_h-\kappa) (\phi \cdot (n_h - n)) )_\Gammah
}_{II_1}
+
\underbrace{
( \Ds \psi , \kappa (\phi \cdot (P\cdot n_h)) )_\Gammah
}_{II_2}
\end{align}
where we in the last equality utilize that $\phi$ is tangential to subtract $n$.
By standard geometry approximation bounds we for term $II_1$ directly get the estimate
\begin{align}
\left|
II_1
\right|
&\lesssim
\underbrace{\| \kappa - \kappa_h \|_{L^\infty(\Gammah)}}_{\lesssim h^{k_g-1}}
\underbrace{\| n - n_h \|_{L^\infty(\Gammah)}}_{\lesssim h^{k_g}}
\| \Ds \psi \|_{L^2(\Gammah)}
\| \phi \|_{L^2(\Gammah)}
\\&\lesssim
h^{2 k_g - 1}
\| \psi \|_{H^1_{\mathrm{tan}}(\Gamma)}
\| \phi \|_{L^2(\Gamma)}
\end{align}
which is sufficient as $k_g+1 \leq 2k_g-1$ for $k_g \geq 2$.

For term $II_2$ we again need to utilize the non-standard geometry approximation estimate in Lemma~\ref{lem:Phn} which gives
\begin{align}
| II_2 |
\lesssim
h^{k_g+1}
\| \mathrm{tr}(\kappa \Ds \psi) \phi \|_{W^1_1(\Gamma)}
\lesssim
h^{k_g+1}
\| \psi \|_{H^2(\Gamma)}
\| \phi \|_{H^1(\Gamma)}
\end{align}
where we use the Cauchy--Schwarz inequality on the $W^1_1(\Gamma)$ norm and the bound on $\kappa$ in the last inequality. We can now move onto covariant derivatives via Lemma~\ref{lem:H-Ht-equivalence}.

\paragraph{Conclusion (\ref{eq:quad-a-tan}).}
Collecting all terms and noting that the various norms on $\psi$ and $\phi$ are trivially bounded
by $\|\psi\|_{H_\mathrm{tan}^2(\Gamma)}$ and $\|\phi\|_{H_\mathrm{tan}^2(\Gamma)}$ concludes the proof of \eqref{eq:quad-a-tan}.

\paragraph{Estimate (\ref{eq:quad-l-tan}).}
Mimicking the calculation for \eqref{eq:quad-l} and utilizing that both $f$ and $\phi$ are tangential to $\Gamma$ give
\begin{align}
Q_l(\phi) &= (f,\phi_t)_\Gamma - (f^e,\phi^e_\thh)_\Gammah
\\
&= (|B|f^e,\phi^e)_\Gammah - (f^e,v_\thh)_\Gammah
\\
&= (|B|f^e,(\underbrace{I - \Psh}_{=\nh\otimes\nh})\phi^e )_\Gammah + ((|B|-1) f^e,\phi^e_\thh)_\Gammah
\\
&= (|B| \nh\cdot f^e, \nh\cdot \phi^e )_\Gammah + ((|B|-1) f^e,\phi^e_\thh)_\Gammah
\\
&\lesssim
\| \Ps\cdot\nh \|_{L^\infty(\Gammah)}^2 \| f^e \|_{\Gammah} \| \phi^e \|_{\Gammah}
+
\left\| |B|-1 \right\|_{L^\infty(\Gammah)} \| f^e \|_{\Gammah} \| \phi^e_\thh \|_{\Gammah}
\\
&\lesssim
(h^{2 k_p} + h^{k_g+1}) \| f \|_{\Gamma} \| \phi \|_{\Gamma}
\end{align}
where we finally use the norm 
equivalence (\ref{eq:normequ-vec}).
\end{proof}

\subsection{Basic Lemmas} \label{section:err-basic-lem}

\begin{lem}[Poincaré Inequality on $\boldsymbol\Gamma_{\boldsymbol h}$]\label{lem:poincare-discrete} For $k_g\geq 1$,  
there are constants such that for all $v \in H^1(\mcK_h)$ and $h\in (0,h_0],$ with $h_0$ small enough,
\begin{equation}\label{eq:poincare-discrete}
\| v  \|_\Gammah 
\lesssim \| \DPsh v_\thh \|_{\mcK_h} + \| v_\nh \|_\Gammah 
\end{equation}
\end{lem}
\begin{proof}
Using norm equivalence \eqref{eq:normequ}, splitting in tangent and normal
components and the triangle inequality we have
\begin{align} \label{eq:discpoin-proof-a}
\| v \|_\Gammah &\lesssim \| v \|_\Gamma
\lesssim \| v_t \|_\Gamma + \| v_n \|_\Gamma
\end{align}
For the normal component we have the estimate
\begin{align}
  \| v_n\|_\Gamma &\lesssim   \| (n - \nh )\cdot v \|_\Gamma    
   + \|  \nh  \cdot v \|_\Gamma
  \\ \label{eq:discpoin-proof-b}
  &\lesssim h^{k_g} \| v \|_\Gamma +   \| v_\nh  \|_\Gamma
   \\ \label{eq:discpoin-proof-c}
  &\lesssim h^{k_g} \| v \|_\Gammah +   \| v_\nh  \|_\Gammah
\end{align}
Next the tangent component can be estimated using the Poincar\'e inequality (Lemma \ref{lem:poincare-cont}) on~$\Gamma$
\begin{align}
 \| v_t \|_\Gamma  &\lesssim \| \Ds v_t \|_\Gamma
 \\&\lesssim
 \| \Ds v_t \|_{\mcK_h}
 \\&\lesssim
 \| \Dsh v_{t_h} \|_{\mcK_h} +
 \| \Ds v_t - \Dsh v_{t_h} \|_{\mcK_h}
 \\&\lesssim
 \| \Dsh v_{t_h} \|_{\mcK_h} +
h^{k_g} \| v  \|_{H^1(\Gammah)} + \underbrace{h^{k_g-1}}_{\lesssim 1} \|v_\nh\|_\Gammah
\end{align}
where we changed domain 
of integration from $\Gamma$ to $\Gammah$, added and subtracted 
$\Dsh v_{t_h}$ and used the triangle inequality,
and finally we used Lemma~\ref{lem:Ds-Dsh}.
Combining the two above estimates above in \eqref{eq:discpoin-proof-a}
and using a kickback argument to hide
$h^{k_g} \| v \|_\Gammah$ for all $h \in (0,h_0]$ with $h_0$ small enough, we 
conclude that
\begin{equation}  \label{eq:discpoin-proof-i}
\| v \|_\Gammah 
\lesssim 
 \| \DPsh v_\thh  \|_{\mcK_h} + \|v_\nh \|_\Gammah 
 +  h^{k_g} \| v \otimes \nablash \|_\Gammah 
\end{equation} 

What remains is to handle the last term in \eqref{eq:discpoin-proof-i}.
For $k_g = 1$ we have the special situation that $n_h$ and $\Psh$ are piecewise 
constant which leads to the identity 
$v_\thh \otimes \nablash 
= 
(\Psh v_\thh ) \otimes \nablash 
= 
\Psh ( v_\thh \otimes \nablash )
=
 \DPsh v_\thh 
$, and we have the estimates 
\begin{align}
h^{k_g} \| v \otimes \nablash \|_\Gammah 
&\leq 
h^{k_g} \| v_\thh\otimes \nablash \|_{\mcK_h} 
+ 
h^{k_g} \| (v_\nh \nh)  \otimes \nablash \|_{\mcK_h}
\\
&=
h^{k_g} \| \DPsh v_\thh \|_{\mcK_h} 
+ 
h^{k_g} \|\nablash v_\nh \|_\Gammah
\\
&\lesssim 
h^{k_g} \| \DPsh  v_\thh  \|_{\mcK_h} 
+ 
h^{k_g-1} \| v_\nh \|_\Gammah
\\ \label{eq:discpoin-proof-j}
&\lesssim 
\| \DPsh v_\thh  \|_\Gammah 
+ 
 \| v_\nh \|_\Gammah
\end{align}
where we used the fact that $\nh$ is constant to conclude that $v_\nh$ is a polynomial on 
each element in the mesh  and thus we have the inverse bound $\|\nablash v_\nh\|_\Gammah 
\lesssim h^{-1} \| v \|_\Gammah$. The estimate (\ref{eq:discpoin-proof-j}) together 
with (\ref{eq:discpoin-proof-i}) concludes the proof of (\ref{eq:poincare-discrete}) in 
this case. For $k_g\geq 2$ we may instead use an inverse inequality,
\begin{equation}
h^{k_g} \| v \otimes \nablash \|_\Gammah 
\lesssim 
h^{k_g - 1} \| v \|_\Gammah
\end{equation} 
and conclude the proof of (\ref{eq:poincare-discrete}) by again 
using a kickback argument.
\end{proof}

\begin{lem}[Discrete ${\boldsymbol H}^{\boldsymbol 1}$ Type Bounds]\label{lem:discrete-bounds} For $k_p\geq k_g \geq 1$ and all $v\in V_h$, $h\in(0,h_0]$ with $h_0$ small enough, there are constants such that
\begin{align}
\label{eq:disc-bound-a}
\| \DPsh v_\thh \|_{\mcK_h} + 
h^{-1}\| v_\nh \|_\Gammah
&\lesssim \tn v \tn_h
\\
\label{eq:disc-bound-b}
\| v \otimes \nablash \|_{\Gammah}
&\lesssim 
\tn v \tn_h
\\
\label{eq:disc-bound-c}
\| v \|_{H^1(\Gammah)} + h^{-1} \| v_\nh \|_\Gammah 
&\lesssim \tn v \tn_h
\end{align}
\end{lem}
\begin{proof} {\bf Estimate (\ref{eq:disc-bound-a}).}
By adding and subtracting different normals, the triangle inequality, geometric bounds, and the discrete Poincaré inequality \eqref{eq:poincare-discrete} we obtain 
\begin{align}
h^{-1}\| v_{n_h} \|_\Gammah
&\lesssim
h^{-1} \| v_{\widetilde{n}_h} \|_\Gammah
+
h^{-1} \| (n - \nh)\cdot v \|_\Gammah
+
h^{-1} \| (n - \widetilde{n}_h)\cdot v \|_\Gammah
\\&\lesssim
h^{-1} \| v_{\widetilde{n}_h} \|_\Gammah
+
\underbrace{(h^{k_g-1} + h^{k_p-1})}_{\lesssim 1} \| v \|_\Gammah
\\&\lesssim
h^{-1} \| v_{\widetilde{n}_h} \|_\Gammah
+
\| \DPsh v_\thh \|_{\mcK_h} + \| v_\nh \|_\Gammah 
\end{align}
Hiding the $\| v_\nh \|_\Gammah$ term on the right using a kickback argument gives the bound
\begin{align}
h^{-1}\| v_{n_h} \|_\Gammah
&\lesssim
\| \DPsh v_\thh \|_{\mcK_h}
+
h^{-1} \| v_{\widetilde{n}_h} \|_\Gammah
\end{align}
and estimate \eqref{eq:disc-bound-a} readily follows.

\paragraph{Estimate (\ref{eq:disc-bound-b}).}
We begin with the estimate
\begin{align}
\| v \otimes \nablash \|_\Gammah
&
\leq 
\| \Psh ( v \otimes \nablash )\|_\Gammah 
+ \| \Qsh (v \otimes \nablash) \|_\Gammah
\\
&
\leq 
\tn v \tn_h + 
\underbrace{\| \Qsh (v_t \otimes \nablash) \|_\Gammah}_{I} 
+ \underbrace{\| \Qsh (( v_n n)  \otimes \nablash) \|_\Gammah}_{II}
\\ \label{eq:lemA-Proof-a}
&\lesssim 
\tn v \tn_h  +  h \|  v  \otimes \nablash  \|_\Gammah
\end{align} 
Here we used the orthogonal decomposition $v = v_t + v_n n$, the estimate 
\begin{equation}
\| \Psh ( v \otimes \nablash )\|_\Gammah  
\lesssim \| \DPsh v_\thh \|_{\mcK_h} + \| v_\nh \|_\Gammah \lesssim \tn v \tn_h
\end{equation}
which holds by \eqref{eq:disc-bound-a},
and the estimates
\begin{equation}
I \lesssim \tn v \tn_h, \qquad II \lesssim  h \|  v  \otimes \nablash  \|_\Gammah + \tn v \tn_h 
\end{equation} 
The second term on the right hand side of (\ref{eq:lemA-Proof-a})  can now be hidden in 
the left hand side using a kick back argument, for all $h\in (0,h_0]$ with $h_0$ small enough. 
We now turn to the verification of the estimates of Terms $I$ and $II$.

\paragraph{Term $\bfI$.} 
Starting from the expansion
\begin{equation}
v_t 
= \sum_{i=1}^3 v_i p_i  
\end{equation}
and computing the derivative we obtain
\begin{align}
v_t \otimes \nablash 
&=  \sum_{i=1}^3 p_i \otimes (\nablash v_i) + v_i ( p_i \otimes \nablash)  
\end{align}
Thus we conclude that 
\begin{align}
\Qsh (v_t \otimes \nablash )
&=
\sum_{i=1}^3 \Qsh (p_i \otimes (\nablash v_i)) + \Qsh (v_i ( p_i \otimes \nablash))
\\
&= 
\sum_{i=1}^3 (\Qsh p_i) \otimes (\nablash v_i) + v_i \Qsh ( p_i \otimes \nablash)
\end{align}
and by the bounds $\|\Ps\cdot\nh \|_{L^\infty(\Gammah)} 
\lesssim h^{k_g}$ and $\| \nablash p_i \|_{L^\infty(\Gamma_h)} \lesssim 1$ we have
\begin{align}
\| \Qsh v_t \otimes \nablash  \|_\Gammah
&\lesssim
\sum_{i=1}^3   h^{k_g}  \|\nablash v_i \|_\Gammah +  \| v_i \|_{\Gammah}  
=\bigstar
\end{align}
Next, using the identity $v_i = v \cdot p_i = v \cdot (P e_i)$ 
we may  add and subtract an interpolant and then use super-approximation \eqref{eq:superapproximationH1L2} and an inverse inequality as follows
\begin{align}
 \|\nablash v_i \|_\Gammah
 &\lesssim 
  \|\nablash (I- \pi_h) (v\cdot p_i) \|_{\mcK_h} +  \|\nablash \pi_h v_i \|_{\mcK_h}
  \\
   &\lesssim
  \| v \|_{\mcK_h} +  h^{-1} \|\pi_h v_i \|_{\mcK_h}
    \\
   &\leq 
  \| v \|_{\Gammah} +  h^{-1} \| v_i \|_{\Gammah}
      \\
   &\leq 
(1 +  h^{-1}  )  \| v \|_{\Gammah}  
\end{align}
where we used the $L^2$ super-stability of $\pi_h$ \eqref{eq:super-stab} and the trivial estimate 
$|v_i| \leq |v|$. 
Using the discrete Poincaré inequality  \eqref{eq:poincare-discrete} and the bound \eqref{eq:disc-bound-a},  we 
obtain
\begin{align}
\bigstar
&\lesssim
\underbrace{(1 + h^{k_g - 1} )}_{\lesssim 1} \| v \|_\Gammah 
\lesssim
\|\Dsh v_\thh \|_{\mcK_h} + \| v_\nh \|_\Gammah
\lesssim
\tn v \tn_h 
\end{align}
Thus we conclude that 
\begin{equation}
I \lesssim \tn v \tn_h 
\end{equation}

\paragraph{Term $\bfI\bfI$.}  Proceeding in the same way as above
\begin{align}
\| \Qsh (v_n n )  \otimes \nablash \|_\Gammah 
&\leq 
\|   \nablash v_n  \|_\Gammah +  \| v_n \|_\Gammah \| Q_h \kappa Q_h \|_{L^\infty(\Gammah)}  
\\
&\lesssim
\| \nablash(I-\pi_h) v_n  \|_\Gammah
+
\|  \nablash(\pi_h v_n)   \|_\Gammah
\\ \nonumber
&\qquad + h^{k_g} \| v_n \|_\Gammah \| Q_h P \kappa P Q_h \|_{L^\infty(\Gammah)} 
\\
&\lesssim 
h \|  v  \otimes \nablash  \|_\Gammah
+
h^{-1} \| \pi_h v_n  \|_\Gammah
+ h^{2k_g} \| v_n \|_\Gammah
\\
&\lesssim \label{eq:lemA-ff}
h \|  v  \otimes \nablash  \|_\Gammah
+
\underbrace{(h^{-1} + h^{2k_g} )}_{\lesssim h^{-1}} \| v_n  \|_\Gammah
\end{align}
where we used super-approximation \eqref{eq:superapproximationH1H1}, an inverse inequality, 
and the $L^2$ super-stability of $\pi_h$ \eqref{eq:super-stab}. In the second 
term we replace $n$ by $n_h$, and use (\ref{eq:poincare-discrete}), 
\begin{align}
 \| v_n  \|_\Gammah 
 &\lesssim \| v_{n_h} \|_\Gammah 
 + h^{k_g} \| v \|_\Gammah
 \\
   &\lesssim \| v_{n_h} \|_\Gammah 
+ h^{k_g} ( \| \Dsh v_\thh \|_{\mcK_h} + \| v_\nh \|_\Gammah ) 
\end{align}
Thus we have 
\begin{equation}
h^{-1} \| v_n  \|_\Gammah \lesssim 
\underbrace{(h^{-1} + h^{k_g-1})}_{\lesssim h^{-1}} \| v_{n_h} \|_\Gammah 
 + \underbrace{h^{k_g-1}}_{\lesssim 1}  \| \DPsh v_\thh \|_{\mcK_h} 
 \lesssim \tn v \tn_h
\end{equation}
where we used \eqref{eq:disc-bound-a}.
Collecting the estimates we obtain
\begin{equation}
II \lesssim h \|  v  \otimes \nablash  \|_\Gammah + \tn v \tn_h 
\end{equation} 
which finalizes the proof of estimate \eqref{eq:disc-bound-b}.

\paragraph{Estimate (\ref{eq:disc-bound-c}).}
Using the discrete Poincaré inequality 
\eqref{eq:poincare-discrete} and estimates \eqref{eq:disc-bound-b} and \eqref{eq:disc-bound-a} we obtain
\begin{align}
\| v \|_{H^1(\Gammah)} + h^{-1} \| v_{\nh} \|_\Gammah 
&\lesssim 
\| v \|_\Gammah + \| v \otimes \nablash \|_\Gammah + h^{-1} \| v_{\nh} \|_\Gammah 
\lesssim \tn v \tn_h 
\end{align}
\end{proof}

\begin{lem}[A Continuous ${\boldsymbol H}^{\boldsymbol 1}$ Type Bound] \label{lem:quad-norm-bounds} For $k_g\geq 1$ and $h \in (0,h_0]$ 
with $h_0$ small enough, there are constants such that
\begin{align}\label{eq:quad-norm-cont}
\| v \|_{H^1(\Gammah)} + h^{-1} \| v_\nh \|_\Gammah &\lesssim \| v \|_{H^1_{\mathrm{tan}}(\Gamma)}
, \qquad \forall v \in H^1_{\mathrm{tan}}(\Gamma)
\end{align}
\end{lem}
\begin{proof} We have the estimates
\begin{align}
 \| v \|_{H^1(\Gammah)} + h^{-1} \| v_\nh \|_\Gammah 
 &=  \| v \|_{H^1(\Gamma)} + h^{-1} \| (n - n_h) \cdot v \|_\Gammah 
\\ 
  &=  \| v \|_{H^1(\Gamma)} + h^{k_g -1} \| v \|_\Gammah 
 \\
 &\lesssim ( 1 + h^{k_g -1} ) \| v \|_{H^1_{\mathrm{tan}}(\Gamma)}
\\
&\lesssim  \| v \|_{H^1_{\mathrm{tan}}(\Gamma)}
\end{align}
where we used equivalence of norms \eqref{eq:normequ} and (\ref{eq:H-Ht-equivalence}) 
for the first term and the fact that $v$ is tangential to subtract the exact normal and the 
bound  (\ref{eq:geombounds}) for the error in the normal combined with the Poincaré inequality (Lemma~\ref{lem:poincare-cont}) for the second term.
\end{proof}

\subsection{Error Estimates} \label{section:err-est}

\begin{thm}[Energy Error Estimate]\label{thm:errorest-energy}
Let $u \in H^{k_u+1}_{\mathrm{tan}}(\Gamma)$ be the solution to \eqref{eq:problem} and $u_h$ the solution to \eqref{eq:method}, and assume that the geometry approximation assumptions are fulfilled and $k_p \geq k_g \geq 1$. Then the following estimate holds
\begin{align}\label{eq:errorest-energy}
\tn e \tn_h
&\lesssim 
(h^{k_u} + h^{k_g} + h^{k_p-1})\| u\|_{H^{k_u+1}_{\mathrm{tan}}(\Gamma)}  
\end{align}
for all $h\in(0,h_0]$, with $h_0$ small enough.
\end{thm} 
\begin{proof} Let $e = u - \pi_h u + \underbrace{\pi_h u - u_h}_{e_h}$ and note 
that 
\begin{align}
\tn e \tn_h &\leq \tn u - \pi_h u \tn_h + \tn e_h \tn_h
\\
&\lesssim 
h^{k_u} \| u \|_{H^{k_u + 1}_{\mathrm{tan}}(\Gamma)} + \tn e_h \tn_h
\end{align}
To estimate $ \tn e_h \tn_h $ we add and subtract suitable terms
\begin{align}
\tn e_h \tn_h^2 
&= A_h(e_h,e_h)
\\
&= A_h(\pi_h u - u , e_h) + A_h(u - u_h , e_h)
\\
&= A_h(\pi_h u - u , e_h) + A_h(u,e_h)  - l_h(e_h)
\\
&= A_h(\pi_h u - u , e_h) + a_h(u_\thh,e_{h,\thh})\underbrace{-a(u,e_{h,t}) + l(e_{h,t})}_{=0}  - l_h(e_h) +s_h(u,e_h)
\\
&= A_h(\pi_h u - u , e_h) -  Q_a(u,e_h) + Q_l(e_h)  + s_h(u,e_h)
\\
&\lesssim \tn \pi_h u - u \tn_h \tn  e_h \tn_h 
\\ \nonumber 
&\qquad +  h^{k_g} 
\underbrace{( \| u \|_{H^1(\Gammah)} + h^{-1} \| u_\nh \|_\Gammah )}_{\underset{(a)}\lesssim \| u \|_{H^1_{\mathrm{tan}}(\Gamma)}}   
\underbrace{ ( \|e_h \|_{H^1(\Gammah)} + h^{-1} \| e_{h,\nh} \|_\Gammah )}_{\underset{(b)}\lesssim \tn e_h \tn_h}
\\ \nonumber
&\qquad +  h^{k_g+1} \|f\|_\Gamma \underbrace{\| e_h \|_{\Gammah}}_{\underset{(c)} \lesssim \tn e_h \tn_h} 
\\ \nonumber
&\qquad
+  h^{k_p-1} \|u \|_\Gamma 
\underbrace{\| e_h \|_{s_h}}_{\lesssim \tn e_h \tn_h}
\end{align}
Here we used the identity $l(e_{h,t}) - l_h(e_h) = l(e_h) - l_h ( e_h) = Q_l (e_h)$, which 
holds since $f$ is tangential;  the geometric error bounds in Lemma~\ref{lem:quad};  
the estimates: (a) follows from (\ref{eq:quad-norm-cont}), (b) 
follows from (\ref{eq:disc-bound-c}), and (c) follows from (\ref{eq:poincare-discrete});
%
and 
\begin{align} \label{eq:sh-bound-first}
s_h(u,e_h)
&\leq
\| u \|_{s_h} \|e_h\|_{s_h}
\\
&=
\beta h^{-1} \| \widetilde{n}_h \cdot u\|_{\Gammah} \|e_h\|_{s_h}
\\
&=
\beta h^{-1} \| (\widetilde{n}_h - n) \cdot u\|_{\Gammah} \|e_h\|_{s_h}
\\
&\lesssim \label{eq:sh-bound-last}
h^{k_p - 1} \|u\|_{\Gamma} \|e_h\|_{s_h}
\end{align}
where we used  the fact that 
$u$ is tangential to subtract $n$ and the bound \eqref{eq:nhtilde-assumption}.

Finally, using the interpolation error estimate (\ref{eq:interpol-energy}),
the Poincaré inequality \eqref{eq:poincare-cont}, and 
the trivial inequality $\| f \|_\Gamma\lesssim \| u \|_{H^2_\mathrm{tan}(\Gamma)}$ we obtain
\begin{align}
\tn e_h \tn_h
&\lesssim h^{k_u}\| u\|_{H^{k_u+1}_{\mathrm{tan}}(\Gamma)} 
+ h^{k_g }  \|u\|_{H^{1}_{\mathrm{tan}}(\Gamma)} 
+ h^{k_g+1} \| f \|_\Gamma
+ h^{k_p-1} \| u \|_\Gamma 
\\
&\lesssim (h^{k_u} + h^{k_g} + h^{k_p-1})\| u \|_{H^{k_u+1}_{\mathrm{tan}}(\Gamma)} 
\label{eq:eh-bound-last} 
\end{align}
which concludes the proof.
\end{proof}

\begin{thm}[${\boldsymbol L}^{\boldsymbol 2}$ Error Estimate] \label{thm:error-L2}
Under the same assumptions as in Theorem~\ref{thm:errorest-energy} and $k_p\geq 2$ the following estimate holds
\begin{equation}
\| e \|_\Gamma \lesssim
(h^{k_u+1} + h^{k_g+1} + h^{k_p})\| u \|_{H^{k_u+1}_{\mathrm{tan}}(\Gamma)} 
\end{equation}
\end{thm}
\begin{proof} Splitting the error in a tangential and normal part
\begin{equation}
\| e \|_\Gamma \leq \| e_t \|_\Gamma + \| e_n \|_\Gamma
\end{equation}
Here we have the following estimate of the normal component
\begin{align}
\| e_n \|_\Gamma 
&\lesssim 
\|e\cdot n \|_\Gammah
\\
&\lesssim
\|e\cdot (n-\widetilde{n}_h) \|_\Gammah + \|e\cdot \widetilde{n}_h \|_\Gammah
\\
&\lesssim
h^{k_p} \|e \|_\Gammah + \|e\cdot \widetilde{n}_h \|_\Gammah
\\
&\lesssim
h^{k_p} \|e \|_\Gamma + h \tn e  \tn_h
\\
&\lesssim
 h^{k_p} \|e_t \|_\Gamma + h^{k_p} \|e_n  \|_\Gamma + h \tn e  \tn_h
\end{align}
Using kickback and the energy norm estimate we obtain
\begin{equation}
\| e_n \|_\Gamma 
\lesssim
 h^{k_p}\| e_t \|_\Gamma
 + h (h^{k_u} + h^{k_g} + h^{k_p-1})\| u\|_{H^{k_u+1}_{\mathrm{tan}}(\Gamma)}
\label{eq:enormal-est}
\end{equation}

Next to estimate the tangential part of the error we introduce the dual problem: 
find $\phi \in \Honet$ such that
\begin{equation}
a(v,\phi) = (v,\psi)\qquad \forall v \in \Honet
\end{equation}
where $\psi \in L^2(\Gamma)$ is tangential. 
As $\psi \in L^2(\Gamma)$ we by (\ref{eq:elliptic-stab}) have the elliptic stability
\begin{equation}\label{eq:phi-stab}
\| \phi \|_{H^2_{\mathrm{tan}}(\Gamma)} \lesssim \| \psi \|_\Gamma 
\end{equation}
Setting $v = \psi = e_t$, and adding and subtracting suitable terms we obtain
\begin{align}
\| e_t \|_\Gamma^2
&=
(e_t, \psi)_\Gamma 
\\&= a(e,\phi) 
\\
&= a(e,\phi - \pi_h \phi) + a(e, \pi_h \phi )
\\
&=  a(e,\phi - \pi_h \phi)  + l( \pi_h \phi )  - a(u_h, \pi_h \phi )
\\
&=  a(e,\phi - \pi_h \phi) + l( \pi_h \phi )
\, \underbrace{-\, l_h(\pi_h \phi ) + A_h(u_h,\pi_h \phi )}_{=0}
- a(u_{h}, \pi_h \phi )
\\
&=  a(e,\phi - \pi_h \phi) + Q_l( \pi_h \phi )  -  Q_a(u_h,\pi_h \phi )
 +  s_h (u_h, \pi_h \phi )
\\
&=  a(e,\phi - \pi_h \phi) + Q_l( \pi_h \phi )  + Q_a(e_h,\pi_h \phi ) 
+  s_h (u_h, \pi_h \phi )
-  \underbrace{Q_a(\pi_h u,\pi_h \phi )}_{\bigstar}
\end{align}
where $e_h = \pi_h u - u_h$ as above and we especially indicate the last term $\bigstar$ as the bound for this term does not directly follow from standard calculations.
Using the Cauchy--Schwarz inequality, interpolation estimates, Lemma~\ref{lem:quad}, and bounds which we list and verify below we obtain
\begin{align}
\label{eq:etL2-first}
\| e_t \|_\Gamma^2
&\lesssim
  \| e_t \|_a  \| \phi - \pi_h \phi \|_a  
\\ \nonumber &\quad
  + h^{k_g+1}  \| f \|_\Gamma \tn \pi_h \phi \tn_h 
\\ \nonumber &\quad
+ h^{k_g} \tn e_h \tn_h \tn \pi_h \phi \tn_h
\\ \nonumber &\quad
+ \| u_{h,n} \|_a \| \pi_h \phi \|_a
\\ \nonumber &\quad
+  \| u_h\|_{s_h} \| \pi_h \phi \|_{s_h}
\\ \nonumber &\quad
+ | \bigstar |
\\
&\lesssim
h \left( h^{k_g}\| e_t \|_\Gamma + (1+h)(h^{k_u} + h^{k_g} + h^{k_p-1})\| u\|_{H^{k_u+1}_{\mathrm{tan}}(\Gamma)} \right) \|\phi  \|_{H^2(\Gamma)}
\\ \nonumber &\quad
+ h^{k_g+1} \| f \|_\Gamma \| \phi \|_{H^2_\mathrm{tan}(\Gamma)}
\\ \nonumber &\quad
+ h^{k_g} (h^{k_u} + h^{k_g} + h^{k_p-1})\| u\|_{H^{k_u+1}_{\mathrm{tan}}(\Gamma)} \|\phi\|_{H^2_\mathrm{tan}(\Gamma)}
\\ \nonumber &\quad 
+ \left( h^{k_p}\| e_t \|_\Gamma + h (h^{k_u} + h^{k_g} + h^{k_p-1})\| u\|_{H^{k_u+1}_{\mathrm{tan}}(\Gamma)} \right)  \|  \phi \|_{H^2_\mathrm{tan}(\Gamma)}
\\ \nonumber &\quad
+ (h + h^{k_p -1}) (h^{k_u} + h^{k_g} + h^{k_p-1})\| u\|_{H^{k_u+1}_{\mathrm{tan}}(\Gamma)}  \|  \phi \|_{H^2_{\mathrm{tan}}(\Gamma)}
\\ \nonumber &\quad
+ h^{k_g+1}
\| u \|_{H^2_{\mathrm{tan}}(\Gamma)} \| \phi \|_{H^2_{\mathrm{tan}}(\Gamma)}
\\
&\lesssim\Big( (h^{k_g+1}+h^{k_p})\|e_t\|_\Gamma \label{eq:etL2-last}
\\&\qquad\nonumber
+
\underbrace{(1 +h^{k_p-2} )}_{\lesssim 1 \text{ for } k_p \geq 2} (h^{k_u+1} + h^{k_g+1} + h^{k_p})\| u\|_{H^{k_u+1}_{\mathrm{tan}}(\Gamma)}
\Big)
\underbrace{\| \psi \|_{\Gamma}}_{=\| e_t \|_{\Gamma}
}
\end{align}
where we in the last inequality use $\| f \|_\Gamma\lesssim \| u \|_{H^2_\mathrm{tan}(\Gamma)}$ and the stability estimate \eqref{eq:phi-stab}.
Hiding the $(h^{k_g+1}+h^{k_p})\| e_t \|_\Gamma$ term on the right with a kickback argument and recalling that $k_p \geq 2$, together with the bounds we verify below, completes the proof.

\paragraph{Bounds used in (\ref{eq:etL2-first})--(\ref{eq:etL2-last}).}
We used the following bounds on the error
\begin{align}
 \label{eq:e-equiv}
\| e_t \|_a &\lesssim
h^{k_g} \| e_t\|_\Gamma
+ (1+h) (h^{k_u} + h^{k_g} + h^{k_p-1})\| u\|_{H^{k_u+1}_{\mathrm{tan}}(\Gamma)}
\\ \label{eq:eh-equiv}
\tn e_h \tn &\lesssim (h^{k_u} + h^{k_g} + h^{k_p-1})\| u\|_{H^{k_u+1}_{\mathrm{tan}}(\Gamma)}
\end{align}
on the discrete solution
\begin{align}
\label{eq:uhn}
\| u_{h,n} \|_a &\lesssim 
 h^{k_p}\| e_t \|_\Gamma
 +h (h^{k_u} + h^{k_g} + h^{k_p-1})\| u\|_{H^{k_u+1}_{\mathrm{tan}}(\Gamma)}
\\ \label{eq:sh-cont}
\| u_h \|_{s_h}
&\lesssim (h^{k_u} + h^{k_g} + h^{k_p-1})\| u\|_{H^{k_u+1}_{\mathrm{tan}}(\Gamma)}
\end{align}
on the interpolant
\begin{align} \label{eq:pih-stab-energy}
\tn \pi_h \phi \tn_h 
&\lesssim
\| \phi \|_{H^2_\mathrm{tan}(\Gamma)} 
\\ \label{eq:pih-stab-a}
\| \pi_h \phi \|_a
&\lesssim  \| \phi \|_{H^2_\mathrm{tan}(\Gamma)}
\\ \label{eq:pih-stab-sh}
\| \pi_h \phi \|_{s_h} 
&\lesssim 
(h + h^{k_p -1})  \|  \phi \|_{H^2_{\mathrm{tan}}(\Gamma)}
\end{align}
and on the special term
\begin{align}\label{eq:special-est}
|\bigstar|
= | Q_a(\pi_h u,\pi_h \phi ) |
\lesssim
h^{k_g+1}
\| u \|_{H^2_{\mathrm{tan}}(\Gamma)} \| \phi \|_{H^2_{\mathrm{tan}}(\Gamma)}
\end{align}

\paragraph{Verification of (\ref{eq:e-equiv}).}
By adding and subtracting suitable terms, applying the triangle inequality and using the identity $\Ps e_n\otimes \nablas = (e\cdot n) \kappa$ we get
\begin{align}
\| e_t \|_a 
&\lesssim \| e \|_a + \| e\cdot n \|_{\mcK_h}
\\&\leq \| e_h \|_a + \| u - \pi_h u \|_a +  \| e_n \|_{\Gamma}
\end{align}
where $e_h=\pi_h u - u_h$ as above.
We get the final bound by; on the first term applying equivalence of norms \eqref{eq:normequ} and estimate \eqref{eq:disc-bound-c} yielding
\begin{align}
\| e_h \|_a \lesssim \| e_h \|_{H^1(\Gamma)} &\lesssim \| e_h \|_{H^1(\Gammah)} \lesssim \tn e_h \tn_h
\lesssim (h^{k_u} + h^{k_g} + h^{k_p-1})\| u\|_{H^{k_u+1}_{\mathrm{tan}}(\Gamma)}
\end{align}
where the last inequality comes from the bound \eqref{eq:eh-bound-last};
on the second term applying an interpolation estimate; and on the last term using the bound \eqref{eq:enormal-est} on $\|e_n \|_\Gamma$.

\paragraph{Verification of (\ref{eq:eh-equiv}).} 
This bound directly holds by \eqref{eq:eh-bound-last}.


\paragraph{Verification of (\ref{eq:uhn}).}
By the identity $\Ps u_{h,n}\otimes \nablas = (u_{h}\cdot n) \kappa = ((u_{h}-u)\cdot n) \kappa$ we have
\begin{align}
\| u_{h,n} \|_a 
&=
\| \Ps u_{h,n}\otimes \nablas \|_\Gamma
=
\| ((u_{h}-u)\cdot n) \kappa \|_\Gamma
\lesssim \|e_n \|_\Gamma 
\end{align}
and the bound then follows by using estimate \eqref{eq:enormal-est}.

\paragraph{Verification of (\ref{eq:sh-cont}).}
Applying the triangle inequality and utilizing the fact that $u$ is tangential similarly to the calculation in \eqref{eq:sh-bound-first}--\eqref{eq:sh-bound-last} we have
\begin{align}
\| u_h \|_{s_h} 
&\leq \| u - u_h \|_{s_h} + \| u \|_{s_h}
\\&
\lesssim \tn e \tn_h + h^{k_p-1} \| u \|_\Gamma
\end{align}
Applying the energy error estimate gives the final bound.

\paragraph{Verification of (\ref{eq:pih-stab-energy}).} 
First note that by the triangle inequality, the bound on $n_h$, and an interpolation estimate we have
\begin{align}
\| \Dsh \pi_h \phi \|_\Gammah
&\lesssim
\| (\pi_h \phi) \otimes \nablash \|_\Gammah
\\
&\lesssim
\| \phi \otimes \nablash \|_\Gammah
+
\| (\phi - \pi_h \phi) \otimes \nablash \|_\Gammah
\\
&\lesssim \label{eq:ksvjdn}
\| \phi \otimes \nablas \|_\Gamma
+
h \| \phi \|_{H^2(\Gamma)}
\\
&\lesssim
\| \phi \|_{H^2(\Gamma)}
\end{align}
where we in \eqref{eq:ksvjdn} utilize the equivalence of norms \eqref{eq:normequ}.
Adding and subtracting suitable terms, applying the triangle inequality and interpolation and geometry estimates then yield
\begin{align}
\tn \pi_h \phi \tn_h^2
&= 
\| \Dsh \pi_h \phi \|^2_\Gammah
+ 
h^{-2} \| \widetilde{n}_h \cdot \pi_h \phi \|^2_\Gammah
\\
&\lesssim \| \phi \|^2_{H^2(\Gamma)}
+ 
h^{-2} \| \widetilde{n}_h \cdot (\pi_h \phi - \phi ) \|^2_\Gammah
+ 
h^{-2} \| (\widetilde{n}_h-n) \cdot \phi  \|^2_\Gammah
\\
&\lesssim
\underbrace{( 1 + h^{2k_p - 2 } )}_{\lesssim 1 \text{ for } k_p \geq 1} \| \phi \|^2_{H^2(\Gamma)}
\lesssim
\| \phi \|^2_{H^2_{\mathrm{tan}}(\Gamma)}
\end{align}
where we in the last inequality use Lemma~\ref{lem:H-Ht-equivalence} to move onto covariant derivatives.

\paragraph{Verification of (\ref{eq:pih-stab-a}).}
Adding and subtracting terms and applying an interpolation estimate yield
\begin{align}
\| \pi_h \phi \|_a
&\leq \| \phi - \pi_h \phi \|_a + \| \phi \|_a
\lesssim h \| \phi \|_{H^2(\Gamma)} + \| \phi \|_{\Honet}
\lesssim \| \phi \|_{H^2_\mathrm{tan}(\Gamma)}
\end{align}
where we finally use Lemma~\ref{lem:H-Ht-equivalence} to move onto covariant derivatives.

\paragraph{Verification of (\ref{eq:pih-stab-sh}).}
This bound is established as follows
\begin{align}
\| \pi_h \phi \|_{s_h} 
&= \beta h^{-1} \| \widetilde{n}_h \cdot \pi_h \phi \|_{\Gammah}
\\&\lesssim 
h^{-1} \| \widetilde{n}_h  \cdot (\pi_h \phi - \phi)  \|_{\Gammah} 
+
h^{-1} \| (\widetilde{n}_h - n ) \cdot \phi \|_{\Gammah} 
\\
&\lesssim 
h \|  \phi  \|_{H^2(\Gamma)} 
+
h^{k_p - 1} \| \phi \|_{\Gamma}
\\&
\lesssim 
(h + h^{k_p - 1})\|  \phi \|_{H^2_\mathrm{tan}(\Gamma)} 
\end{align} 
 where we added and subtracted suitable terms, used the fact that $\phi$ is 
 tangential, interpolation and geometry estimates, and finally Lemma~\ref{lem:H-Ht-equivalence}.

\paragraph{Verification of (\ref{eq:special-est}).}
To achieve a bound of the right order for this last term we need to utilize the higher regularity of $u,\phi\in H^2_\mathrm{tan}(\Gamma)$. The bound is obtained by
\begin{align} \label{eq:star-est-a}
| \bigstar | &=
| Q_a(\pi_h u,\pi_h\phi) |
\\&= \label{eq:star-est-b}
| Q_a(\pi_h u - u,\pi_h\phi- \phi)
+
Q_a(u,\pi_h\phi- \phi)
\\&\qquad\nonumber
+
Q_a(\pi_h u - u,\phi)
+
Q_a(u,\phi) |
\\
&\lesssim \label{eq:star-est-c}
h^{k_g}
\left(\| u - \pi_h u \|_{H^1(\Gammah)} + h^{-1} \| (u - \pi_h u)\cdot\nh \|_\Gammah\right)
\\&\qquad\quad \cdot \nonumber
\left(\| \phi - \pi_h\phi \|_{H^1(\Gammah)} + h^{-1} \| (\phi - \pi_h\phi)\cdot\nh \|_\Gammah\right)
\\&\qquad+ \nonumber
h^{k_g}\| u \|_{H_\mathrm{tan}^1(\Gamma)}
\left(\| \phi - \pi_h\phi \|_{H^1(\Gammah)} + h^{-1} \| (\phi - \pi_h\phi)\cdot\nh \|_\Gammah\right)
\\&\qquad+ \nonumber
h^{k_g}
\left(\| u - \pi_h u \|_{H^1(\Gammah)} + h^{-1} \| (u - \pi_h u)\cdot\nh \|_\Gammah\right)
\| \phi \|_{H_\mathrm{tan}^1(\Gamma)}
\\&\qquad+ \nonumber
h^{k_g+1}
\| u \|_{H^2_\mathrm{tan}(\Gamma)}\|\phi\|_{H^2_\mathrm{tan}(\Gamma)}
\\&\lesssim \label{eq:star-est-d}
\underbrace{(h^{k_g+2 k_u} + h^{k_g+k_u} + h^{k_g+1})}_{\lesssim h^{k_g+1} \text{ for }k_u\geq 1}
\| u \|_{H^2_\mathrm{tan}(\Gamma)}\|\phi\|_{H^2_\mathrm{tan}(\Gamma)}
\end{align}
where in \eqref{eq:star-est-b} we add and subtract suitable terms; in \eqref{eq:star-est-c} we apply \eqref{eq:quad-a} to all but the last term to which we instead apply the higher regularity bound \eqref{eq:quad-a-tan}; and in \eqref{eq:star-est-d} we apply the interpolation estimate \eqref{eq:interpol}.
\end{proof}

\section{Numerical Results} \label{section:numerics}

\subsection{Implementation Aspects}

\paragraph{Experimental Set-up.} For our numerical experiments we implemented the variations of the method in {\sc Matlab}
and used its built-in backslash operator, i.e., a direct solver, to solve the resulting sparse linear system of equations. All experiments were run on a computer with 64\,GB memory.

\paragraph{Construction of Geometry Approximations.}
To construct a higher-order geometry approximation $\Gammah$ we started from
a piecewise linear mesh $\mcK_{h,1}$ and composed a parametric mesh $\mcK_{h,k_g}$ by adding nodes for higher-order Lagrange basis functions on each facet $K\in\mcK_{h,1}\subset U_{\delta_0}(\Gamma)$ and mapping the positions of these nodes onto the exact surface $\Gamma$ by the closest point map $p:U_{\delta_0}\rightarrow \Gamma$. In our experiments we consider $1 \leq k_g \leq 5$.

To investigate whether or not convergence is dependent of the mesh structure we also used perturbed meshes, which were generated by randomly moving the mesh vertices in $\mcK_{h,1}$ a distance proportional to $h$ and then mapping the vertices back onto $\Gamma$ by the closest point map.

\paragraph{Penalty Term Normal Approximations.}
The $L^2(\Gammah)$ error estimate (Theorem~\ref{thm:error-L2}) implies that, in order to achieve optimal order convergence, a better approximation of the normal in the penalty term is required.
In our numerical experiments we have access to the true normal on $\Gamma$ and we can thus readily construct approximations of arbitrary order. As described in Section~\ref{sec:method} the improved normal approximations in the penalty term in our implementation are based on node-wise interpolation of the exact normal. Such an implementation is actually very convenient as we are able to deliver an optimal order method using the same order basis functions for the solution, the geometry, and the penalty term normal.

Alternatively, if the normal to the discrete geometry $\Gammah$ is used in the penalty term, a higher-order geometry approximation could be used to achieve optimal order convergence. This alternative, however, requires higher-order basis functions for the geometry approximation. In our experiments below we consider both options.

\subsection{Model Problem and Numerical Example} \label{sec:model-problem}

\paragraph{Geometry.}
The surface of a torus can be expressed in Cartesian coordinates as
\begin{align}
\begin{bmatrix}
x \\
y \\
z
\end{bmatrix}
=
\begin{bmatrix}
(R+r \cos(\theta)) \cos(\phi) \\
(R+r \cos(\theta)) \sin(\phi) \\
r \sin(\theta)
\end{bmatrix}
\end{align}
where $0 \leq \theta,\phi < 2\pi$ are angles and $R,r > 0$ are fixed radii. For our model problem we consider such a geometry with radii $R=1$ and $r=0.6$. Any point on the torus surface can thus be specified using the toroidal coordinates $\{\theta,\phi\}$.
This surface is illustrated in Figure~\ref{fig:mesh1} where we present an example mesh $\mcK_{h,1}$ describing a piecewise linear surface approximation $\Gamma_{h,1}$ of the torus, and in Figure~\ref{fig:mesh2} we illustrate a perturbed version of the same mesh.

\begin{figure}
\centering
\subfigure[Structured mesh]{
\includegraphics[width=0.35\linewidth]{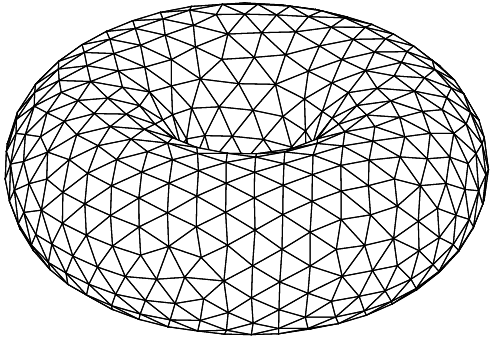}
\label{fig:mesh1}
}
\quad
\subfigure[Perturbed mesh]{
\includegraphics[width=0.35\linewidth]{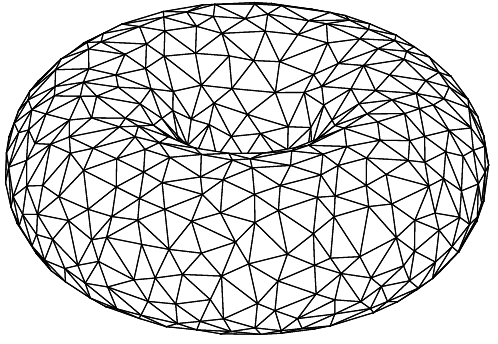}
\label{fig:mesh2}
}
\caption{\textbf{Example meshes.} (a) Structured mesh with mesh size $h=0.25$. (b)~Perturbed version of the same mesh.}
\label{fig:mesh}
\end{figure}

\paragraph{Manufactured Problem.}
We manufacture problems on this geometry from the following ansatz as our analytical tangential vector field solution (expressed in Cartesian coordinates)
\begin{align}
u =
\begin{bmatrix}
- r\sin(3\phi + \theta)\cos(\phi)^2 \sin(\theta) - \cos(\phi + 3\theta)\sin(3\phi)\sin(\phi)(R + r\cos(\theta)) \\
 \cos(\phi + 3\theta)\sin(3\phi)\cos(\phi)(R + r\cos(\theta)) - r\sin(3\phi + \theta)\cos(\phi)\sin(\phi)\sin(\theta) \\
r\sin(3\phi + \theta)\cos(\phi)\cos(\theta)
\end{bmatrix}
\label{eq:ansatz}
\end{align}
and we calculate the corresponding load tangential vector field for both the standard problem \eqref{eq:problem} and the symmetric problem \eqref{eq:problem-sym}. The analytical solution is illustrated in Figure~\ref{fig:field-analytical}.

\begin{figure}
\centering
\subfigure[Analytical solution]{
\includegraphics[width=0.35\linewidth]{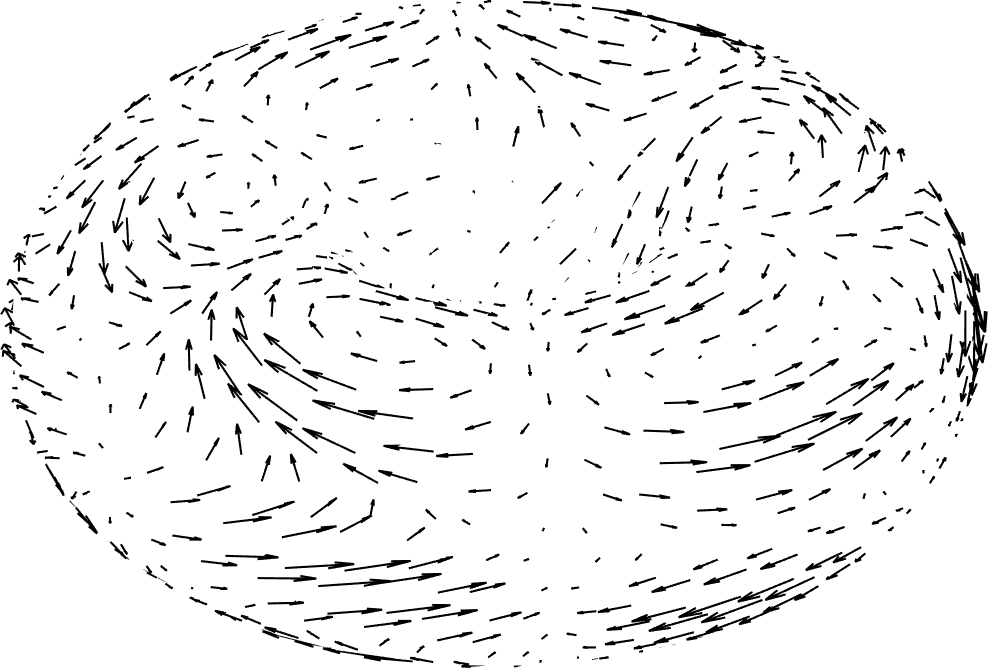}
\label{fig:field-analytical}
}
\quad
\subfigure[Numerical solution]{
\includegraphics[width=0.35\linewidth]{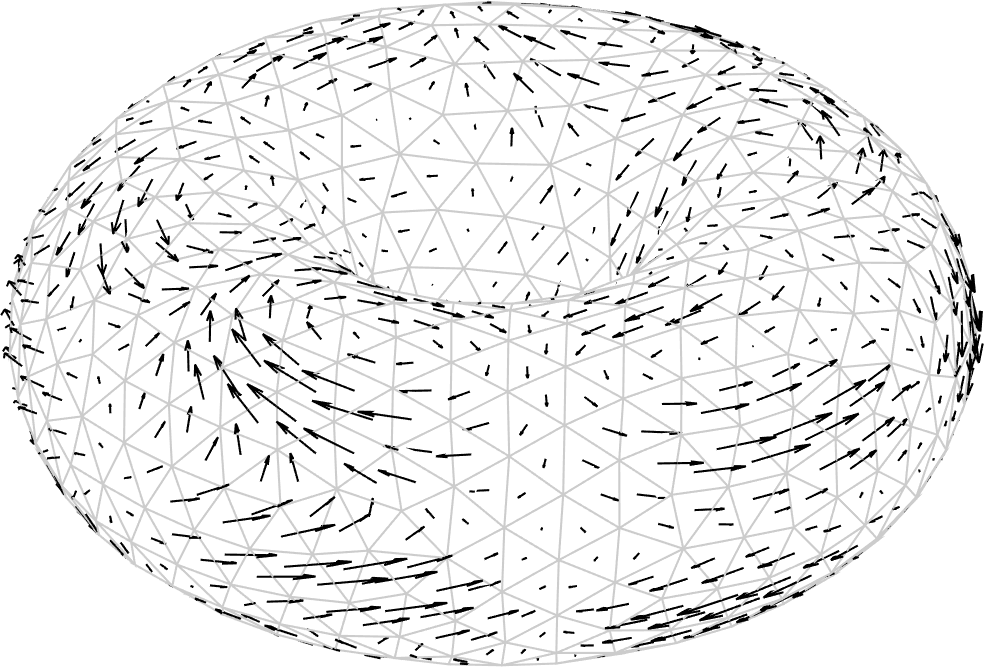}
\label{fig:field-numerical}
}
\caption{\textbf{Solution.} (a) Illustration of the analytical solution \eqref{eq:ansatz} to the model problem. (b) Numerical solution for the standard formulation of the model problem using isoparametric linear finite elements ($h=0.25$).}
\label{fig:field}
\end{figure}

\paragraph{A Numerical Example.}
A numerical solution to the model problem using the standard formulation of the vector Laplacian is shown in Figure~\ref{fig:field-numerical}. In Figure~\ref{fig:error} we present the magnitude of the pointwise error over $\Gamma_h$ using piecewise linear finite elements which varying of the geometry and normal approximations. The results confirm what we can suspect from looking at our estimates; when using isoparametric elements and the normal of $\Gamma_h$ in the penalty term, the dominating error seems to stem from the normal approximation in the penalty term.

\begin{figure}
\centering
\subfigure[$k_g=k_u=1$]{
\includegraphics[width=0.31\linewidth]{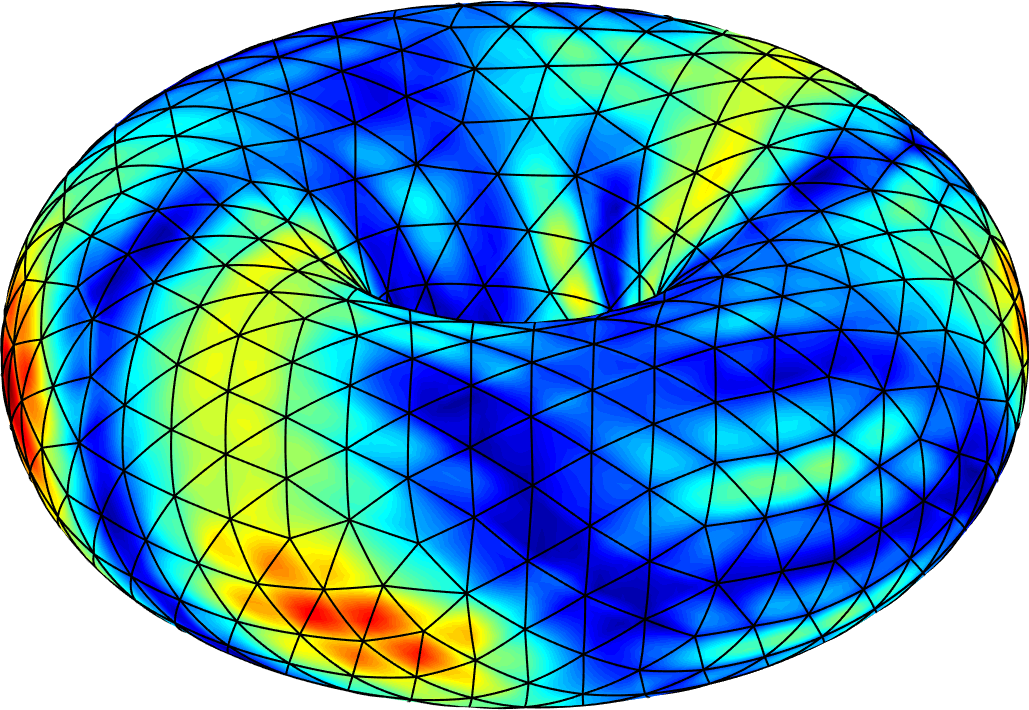}
\label{fig:error-subfig1}
}
\subfigure[$k_g=2$, $k_u=1$]{
\includegraphics[width=0.31\linewidth]{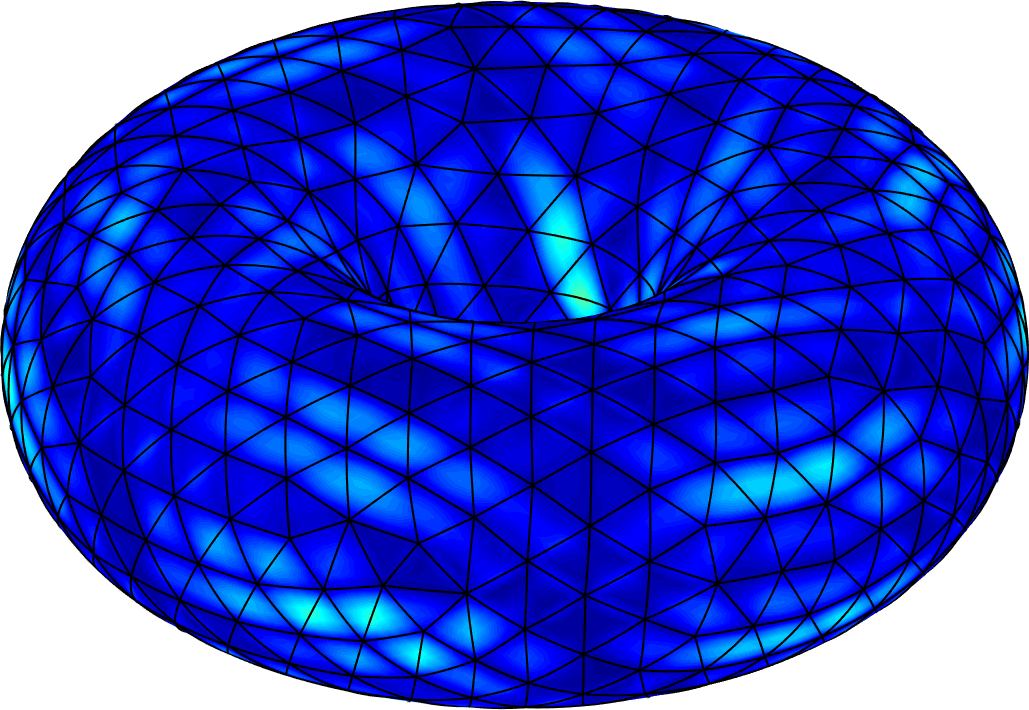}
\label{fig:error-subfig2}
}
\subfigure[$k_g=k_u=1$, $k_p=2$]{
\includegraphics[width=0.31\linewidth]{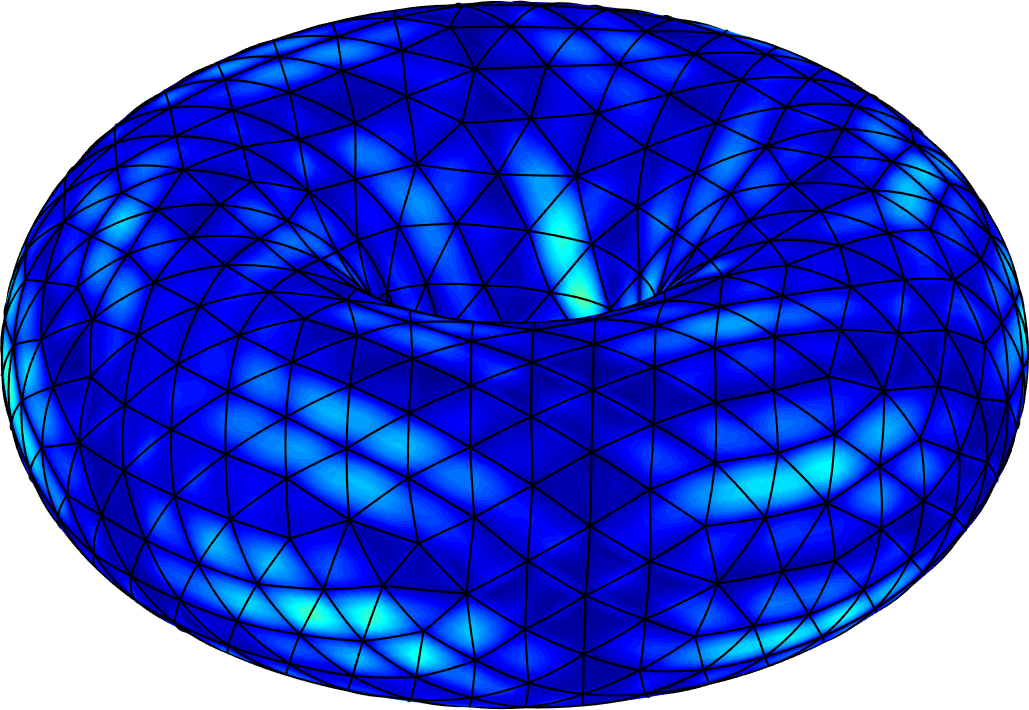}
\label{fig:error-subfig3}
}
\caption{\textbf{Error magnitude.} Error magnitudes $\| u - u_h \|_{\mathbb{R}^3}$ where blue is small and red is large for the standard formulation of the model problem using linear finite elements ($h=0.25$). Note that the increased order of geometry approximation in (b) yields very similar results to (c) where only the normal approximation order in the penalty term is increased, implicating the normal approximation in the penalty term as the dominant source of the error for the isoparametric elements in (a).}
\label{fig:error}
\end{figure}


\subsection{Convergence} \label{sec:convergence}

We perform convergence studies in $L^2(\Gamma_h)$ norm on the model problem for both the standard problem \eqref{eq:problem}, formulated using the full covariant derivative $\Dsh$, and the symmetric problem \eqref{eq:problem-sym}, formulated using the symmetric part of the covariant derivative $\epssh$. To detect mesh dependence we give results for both structured and perturbed meshes, see example meshes in Figure~\ref{fig:mesh}.

According to the $L^2(\Gammah)$ norm error estimate in Theorem~\ref{thm:error-L2} we have optimal order convergence if the normal in the penalty term is of one order better approximation than the normal we have using isoparametric elements. Also, the normal in the penalty term must have an approximation order at least as good as the normal to a piecewise quadratic interpolation of the surface.
Looking at the the convergence results in Figure~\ref{fig:usual}, and the corresponding results in Figure~\ref{fig:usual-pert}, it seems like the requirements in Theorem~\ref{thm:error-L2} are actually sharp. In particular we note a loss of convergence in the case of linear isoparametric elements and suboptimal convergence for higher-order elements. Using either superparametric elements, i.e., elements where the geometry approximation is one order higher than the finite element approximation, or improving the normal approximation in the penalty term, we see restored optimal order convergence.
While we in the analysis only prove Theorem~\ref{thm:error-L2} for the standard problem \eqref{eq:problem} we in Figure~\ref{fig:symmetric} note that the situation seems to be the same for the symmetric problem \eqref{eq:problem-sym}.


\begin{figure}
\centering
\subfigure[$k_g=k_u$]{
\includegraphics[width=0.31\linewidth]{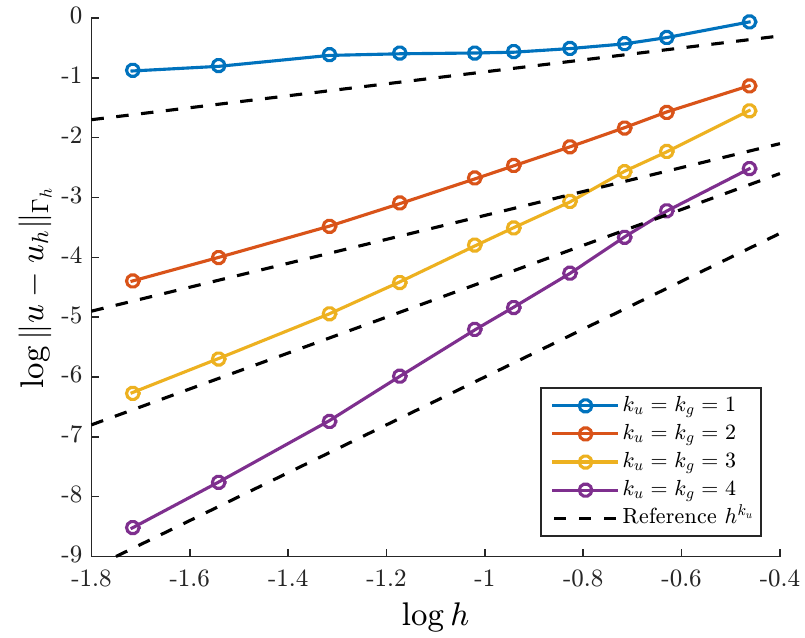}
\label{fig:usual-a}
}
\subfigure[$k_g=k_u+1$]{
\includegraphics[width=0.31\linewidth]{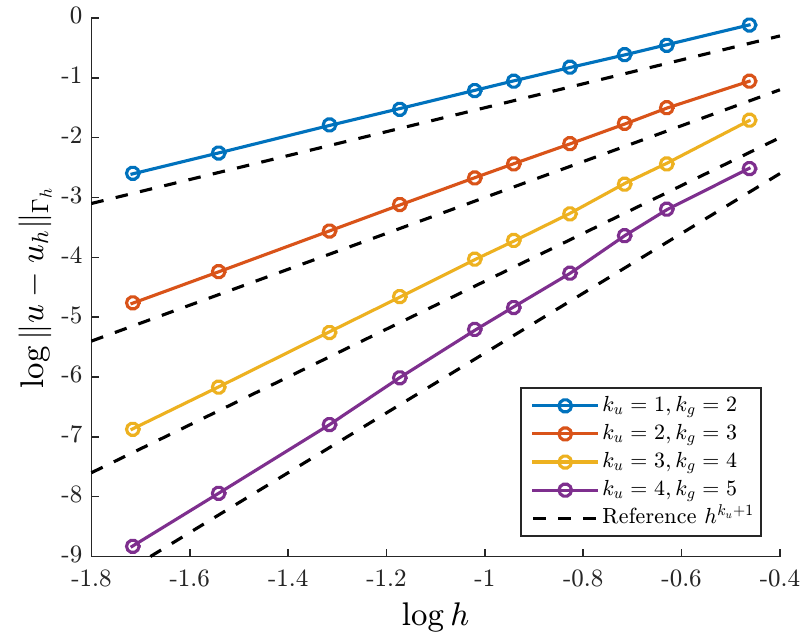}
\label{fig:usual-b}
}
\subfigure[$k_g=k_u$, $k_p=k_u+1$]{
\includegraphics[width=0.31\linewidth]{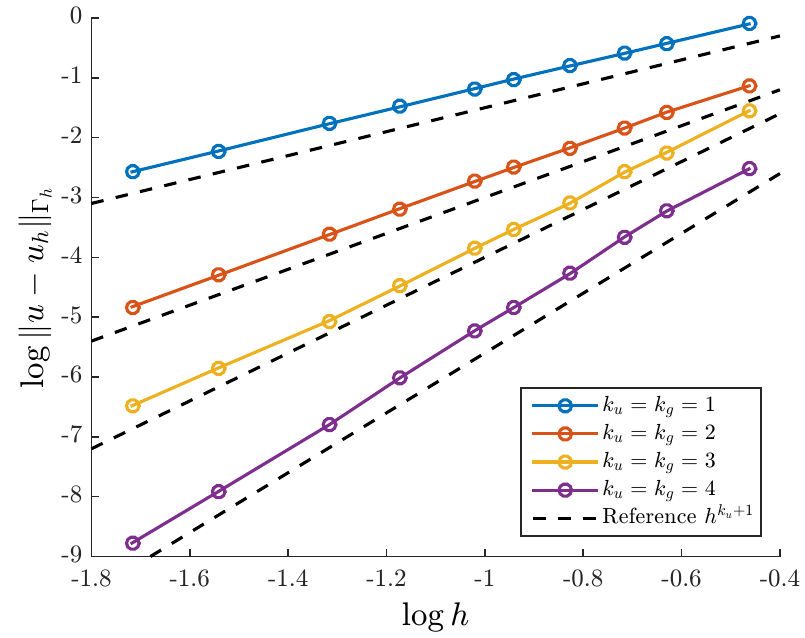}
\label{fig:usual-c}
}

\caption{\textbf{Convergence for standard problem.}
Convergence in $L^2(\Gammah)$ norm is in agreement with the error estimates.
(a) Using isoparametric elements we have no convergence for $k_u=1$ and a convergence rate of $h^{k_u}$ for higher-order elements. The loss of one order here is expected.
(b) Using superparametric elements we have optimal order convergence.
(c) Using isoparametric elements but with an improved normal approximation in the penalty term we again see optimal order convergence.
}
\label{fig:usual}
\end{figure}

\begin{figure}
\centering
\subfigure[$k_g=k_u$]{
\includegraphics[width=0.31\linewidth]{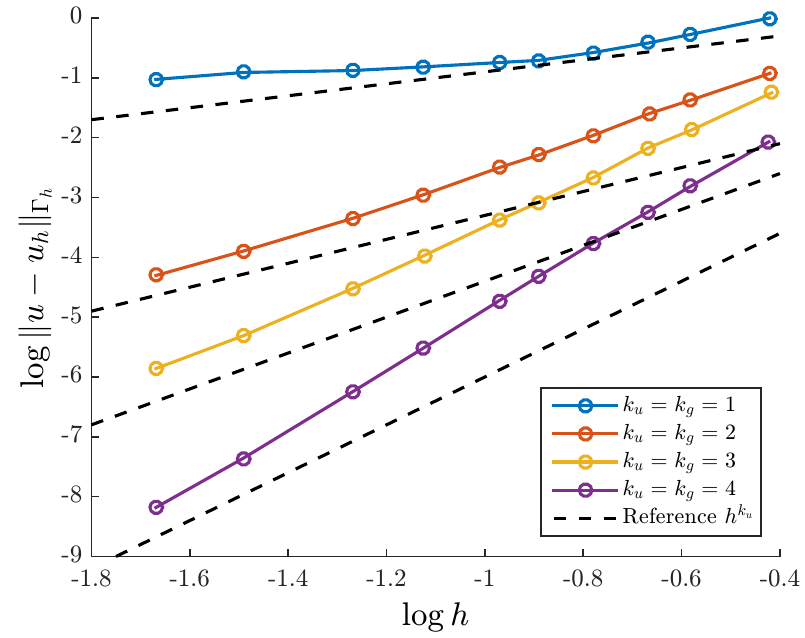}
\label{fig:usual-pert-a}
}
\subfigure[$k_g=k_u+1$]{
\includegraphics[width=0.31\linewidth]{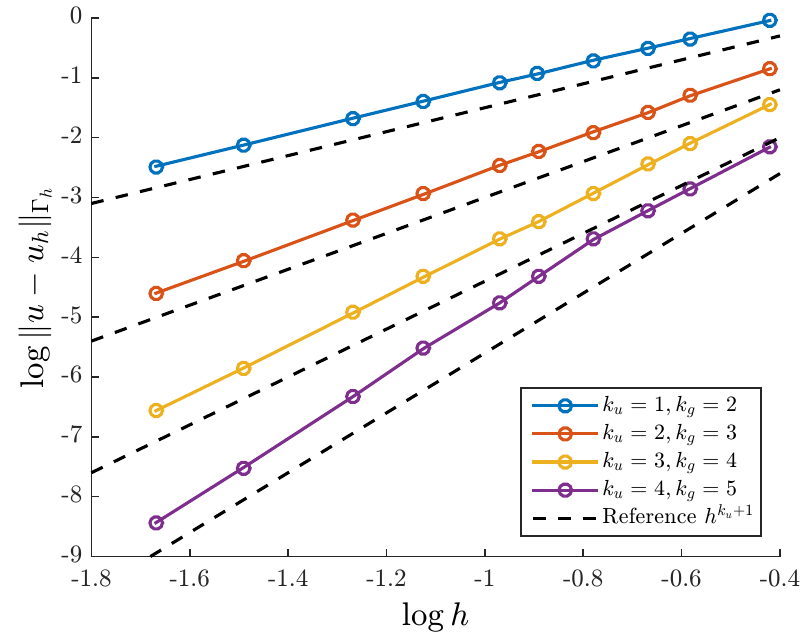}
\label{fig:usual-pert-b}
}
\subfigure[$k_g=k_u$, $k_p=k_u+1$]{
\includegraphics[width=0.31\linewidth]{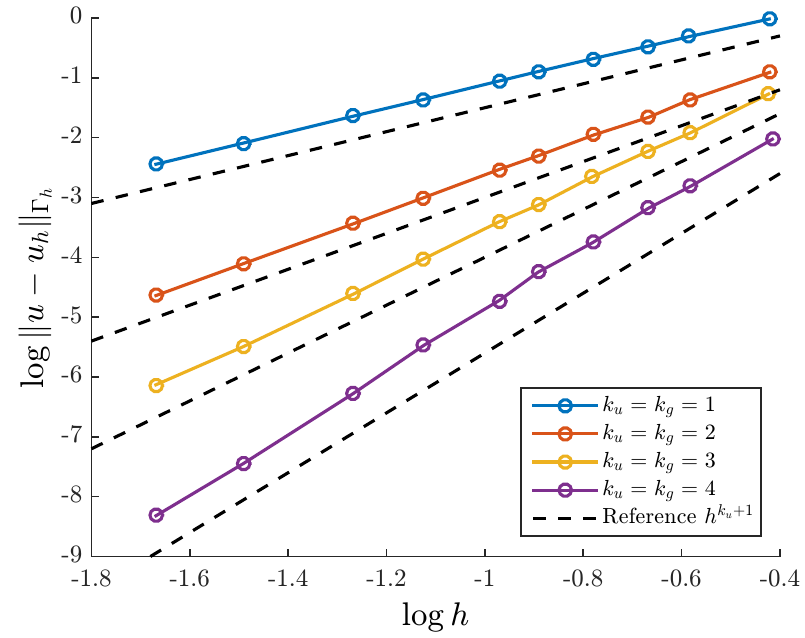}
\label{fig:usual-pert-c}
}

\caption{\textbf{Convergence on perturbed meshes.}
The method seems stable with respect to mesh structure as $L^2(\Gammah)$ convergence is unaffected by mesh perturbations, yielding results very similar to the structured case.
(a) Using isoparametric elements we have no convergence for $k_u=1$ and a convergence rate of $h^{k_u}$ for higher-order elements.
(b) Using superparametric elements we have optimal order convergence.
(c) Using isoparametric elements but with an improved normal approximation in the penalty term we again see optimal order convergence.}
\label{fig:usual-pert}
\end{figure}

\begin{figure}
\centering
\subfigure[$k_g=k_u$]{
\includegraphics[width=0.31\linewidth]{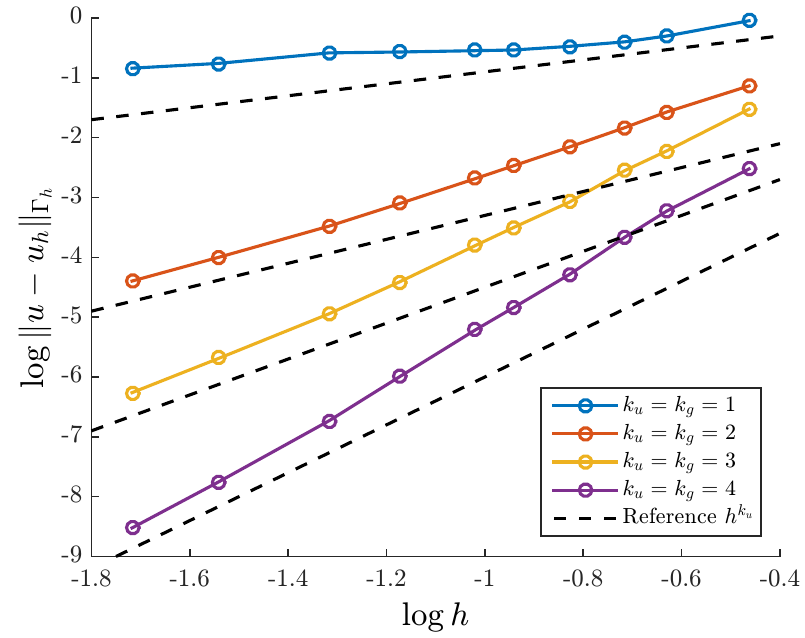}
\label{fig:symmetric-a}
}
\subfigure[$k_g=k_u+1$]{
\includegraphics[width=0.31\linewidth]{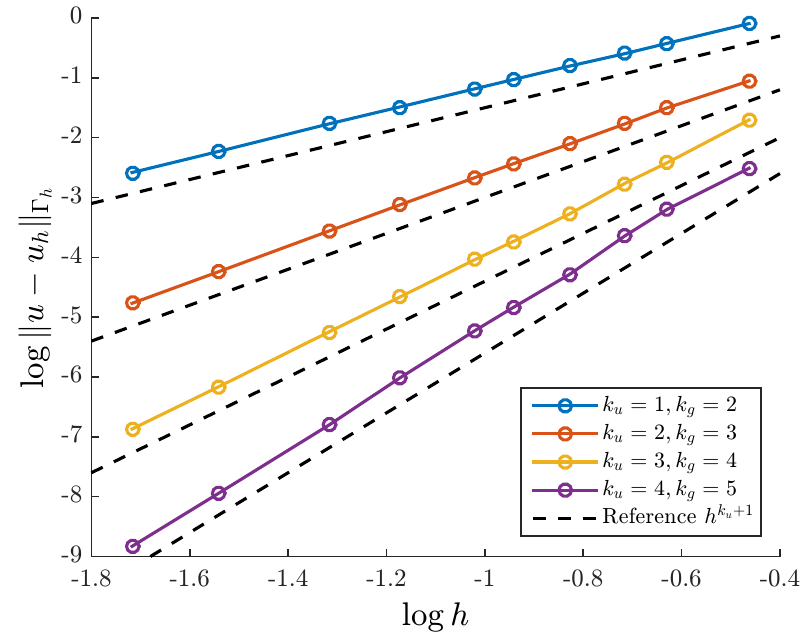}
\label{fig:symmetric-b}
}
\subfigure[$k_g=k_u$, $k_p=k_u+1$]{
\includegraphics[width=0.31\linewidth]{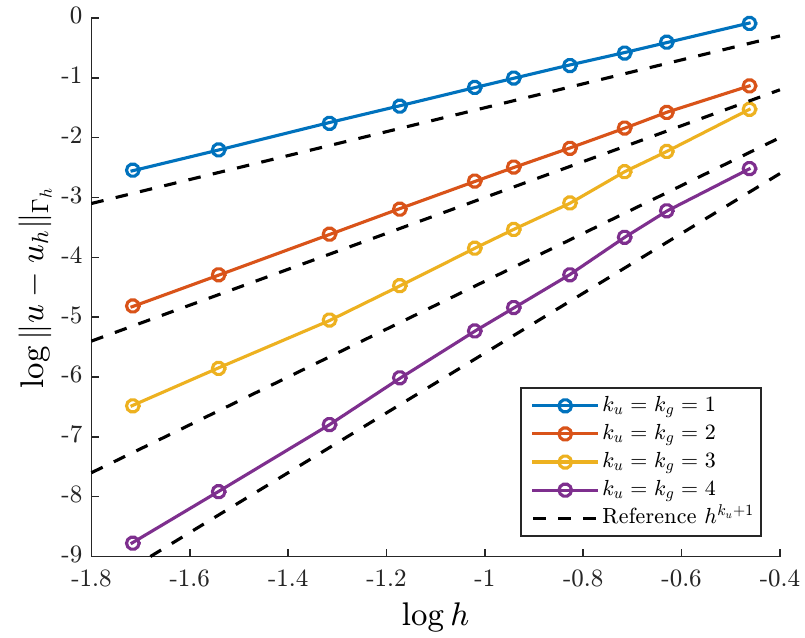}
\label{fig:symmetric-c}
}

\caption{\textbf{Convergence for symmetric problem.}
The method show the same $L^2(\Gammah)$ convergence behavior for the symmetric problem as for the standard problem.
(a) Using isoparametric elements we have no convergence for $k_u=1$ and a convergence rate of $h^{k_u}$ for higher-order elements. 
(b) Using superparametric elements we have optimal order convergence.
(c) Using isoparametric elements but with an improved normal approximation in the penalty term we again see optimal order convergence.
}
\label{fig:symmetric}
\end{figure}

\paragraph{Choice of $\boldsymbol\beta$.}
In the numerical results we have consistently used the normal penalty parameter $\beta=10$. That this is a reasonable choice is motivated by the numerical study presented in Figure~\ref{fig:betas} where we present results for the lowest order elements that exhibit optimal convergence, i.e., linear superparametric elements and linear isoparametric elements with an improved normal in the penalty term. We see that some large values for $\beta$ give a noticeably increased magnitude for the $L^2(\Gammah)$ error, albeit still with the correct asymptotic convergence rate.

\begin{figure}
\centering
\subfigure[$k_g=2$, $k_u=1$]{
\includegraphics[width=0.31\linewidth]{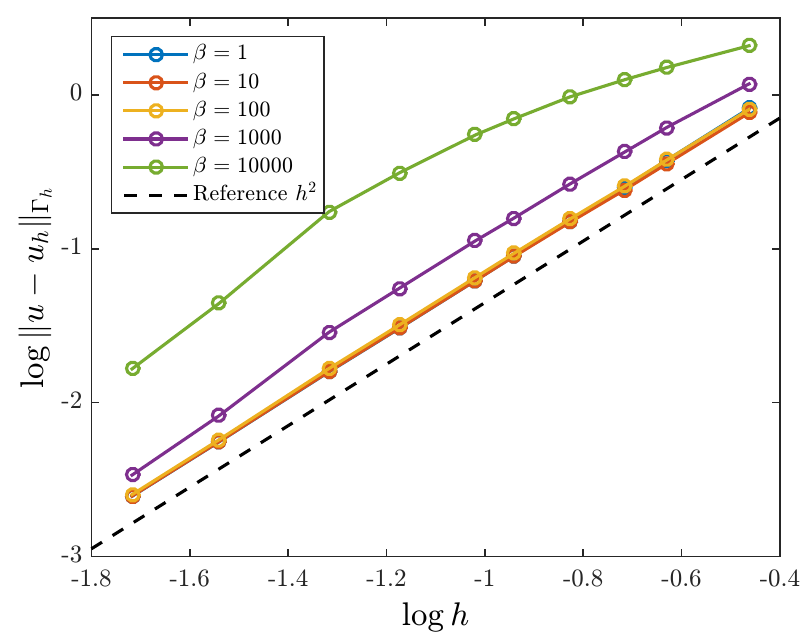}
\label{fig:betas-a}
}
\subfigure[$k_g=2$, $k_u=1$, normal component]{
\includegraphics[width=0.31\linewidth]{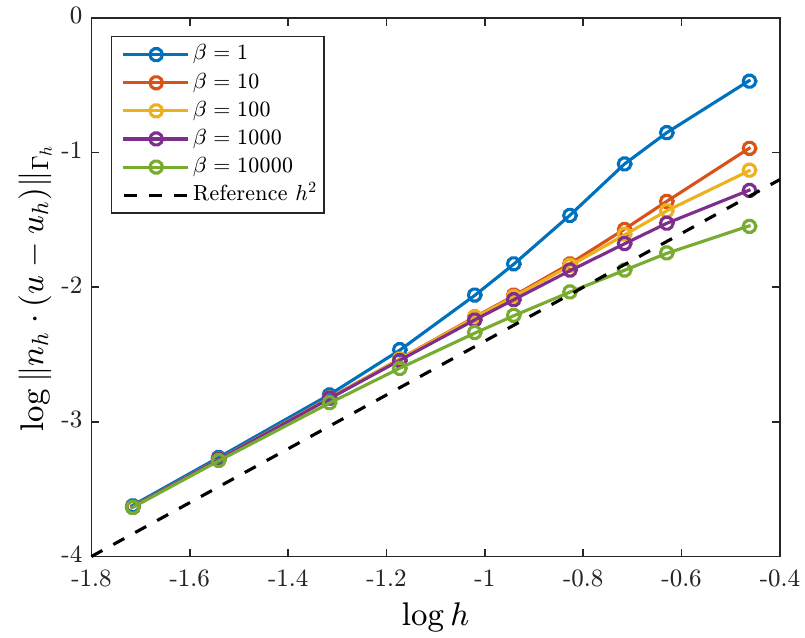}
\label{fig:betas-b}
}
\subfigure[$k_u=k_g=1$, $k_p=2$]{
\includegraphics[width=0.31\linewidth]{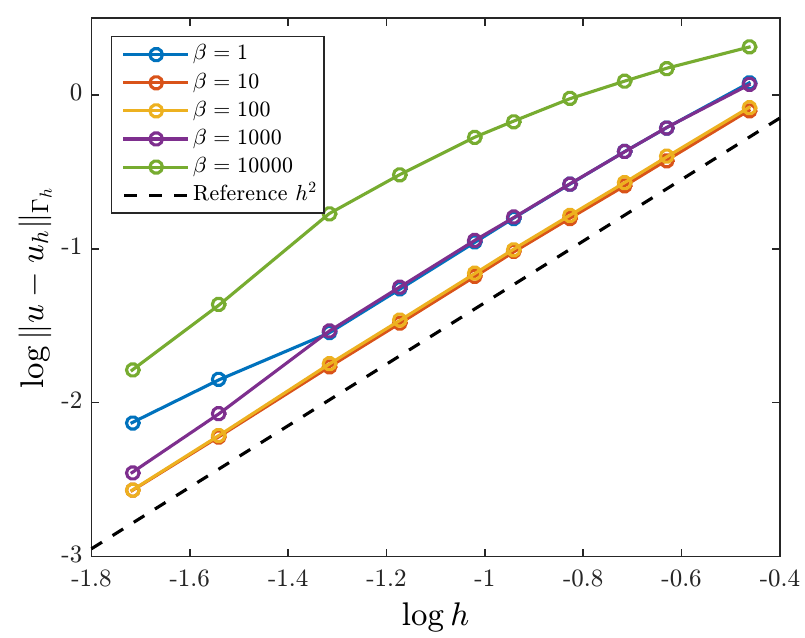}
\label{fig:betas-c}
}

\caption{\textbf{$\boldsymbol{\beta}$-Study.} Convergence studies in $L^2(\Gamma_h)$ norm on the model problem for the standard formulation using various values of the normal penalty parameter $\beta$. (a) Convergence for lowest order superparametric element. (b) Normal component convergence for lowest order superparametric element. (c) Convergence for lowest order isoparametric element with increased normal approximation order in the penalty term.}
\label{fig:betas}
\end{figure}

\paragraph{The Lagrange Multiplier Approach.}
An alternative to using the penalty term, which does not involve a choice of $\beta$, is the Lagrange multiplier approach described in Remark~\ref{rem:multiplier}.
This, more elaborate approach, is also numerically more expensive than the penalty term approach as it is posed as a saddle point problem and the size of the resulting sparse system of equations is increased by the dimension of the approximation space of the Lagrange multipliers.
In the convergence results in Figure~\ref{fig:multiplier} we note very similiar performance to the penalty term approach, with the notable exception that we now see convergence also for the lowest order isoparametric element. At first glance it might even seem like the $L^2(\Gammah)$ convergence for the lowest order isoparametric element is of optimal order, but looking at the normal component of the error presented in Figure~\ref{fig:multiplier-b} it is clear that the asymptotic behavior cannot be of optimal order.
Nevertheless, in cases where linear isoparametric elements must be used the Lagrange multiplier approach has a clear advantage.

\begin{figure}
\centering
\subfigure[$k_g=k_u$]{
\includegraphics[width=0.31\linewidth]{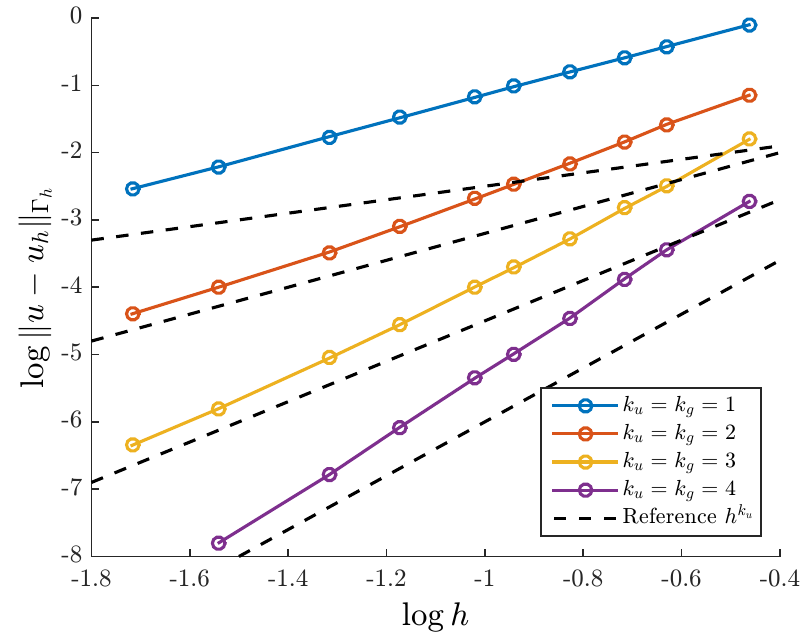}
\label{fig:multiplier-a}
}
\subfigure[$k_g=k_u$, normal component]{
\includegraphics[width=0.31\linewidth]{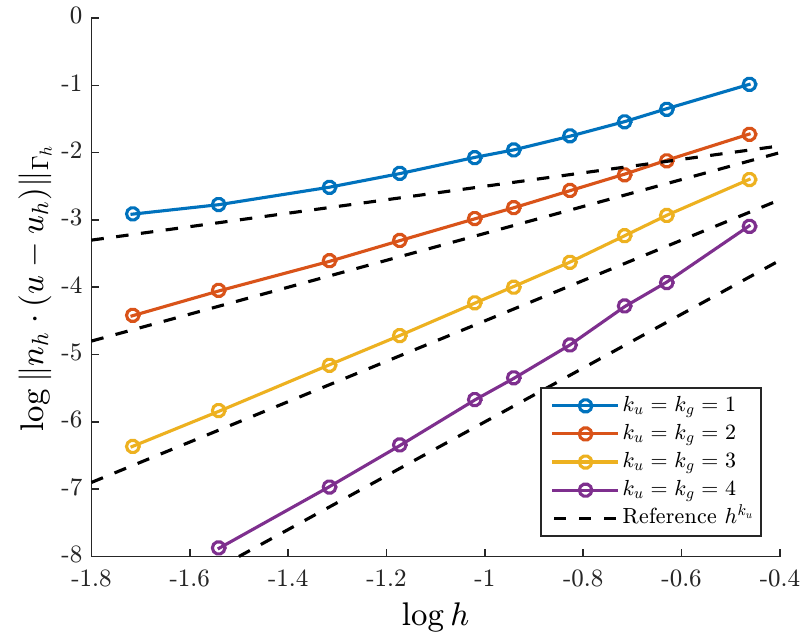}
\label{fig:multiplier-b}
}
\subfigure[$k_g=k_u+1$]{
\includegraphics[width=0.31\linewidth]{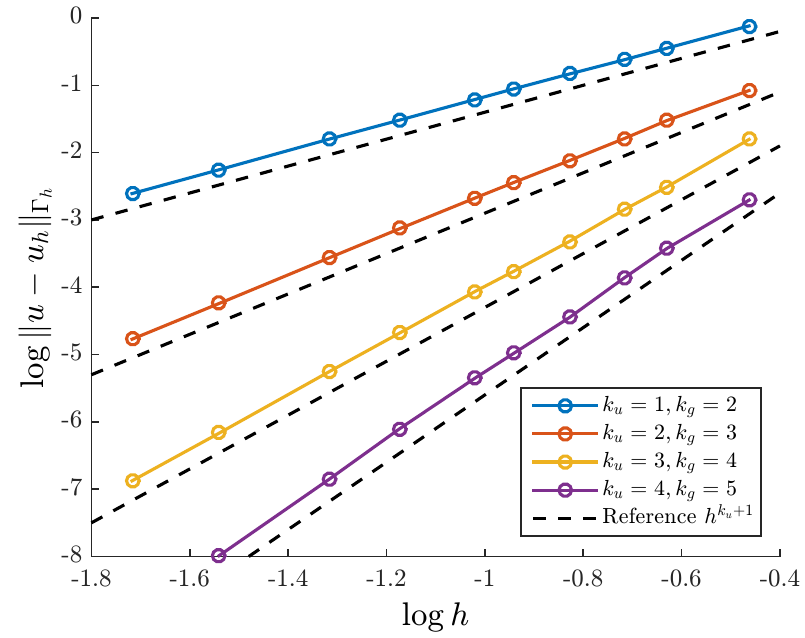}
\label{fig:multiplier-c}
}

\caption{\textbf{Convergence using Lagrange multiplier approach.} Convergence studies in $L^2(\Gammah)$ norm using the Lagrange multiplier approach.
(a) Isoparametric elements suffer from the loss of one order. It might, however, seem like the lowest order isoparametric element is of optimal order but in light of the results in
(b) it seems this behavior will eventually stop.
(c) Using superparametric elements we have optimal order convergence. 
Note that we were unable to compute the last data points for $k_u=4$ due to lack of memory (the mesh at this last data point consists of roughly $1.7\times 10^5$ elements and the finite element space $V_h$ with $k_u=4$ has $4.1\times 10^6$ DoFs).
}
\label{fig:multiplier}
\end{figure}

\paragraph{Tangential Convergence.}
While we in the analysis above prove convergence rates in energy and $L^2$ norms on the full vector field it is of course also of interest to investigate the convergence behavior for the tangential part of the solution.
In Figure~\ref{fig:tangent} we explore the tangential error in $L^2(\Gammah)$ norm when using isoparametric elements in the penalty term approach respectively in the Lagrange multiplier approach. With the exception of linear isoparametric elements using the penalty term approach where we still lack convergence, the convergence rates for the tangential part of the error seem to be of optimal order also for isoparametric elements.
Tangential convergence is arguably more natural to consider as we in the considered vector Laplace problems seek to approximate a tangential vector field.

\vfill 

\begin{figure}
\centering
\subfigure[Penalty term approach]{
\includegraphics[width=0.35\linewidth]{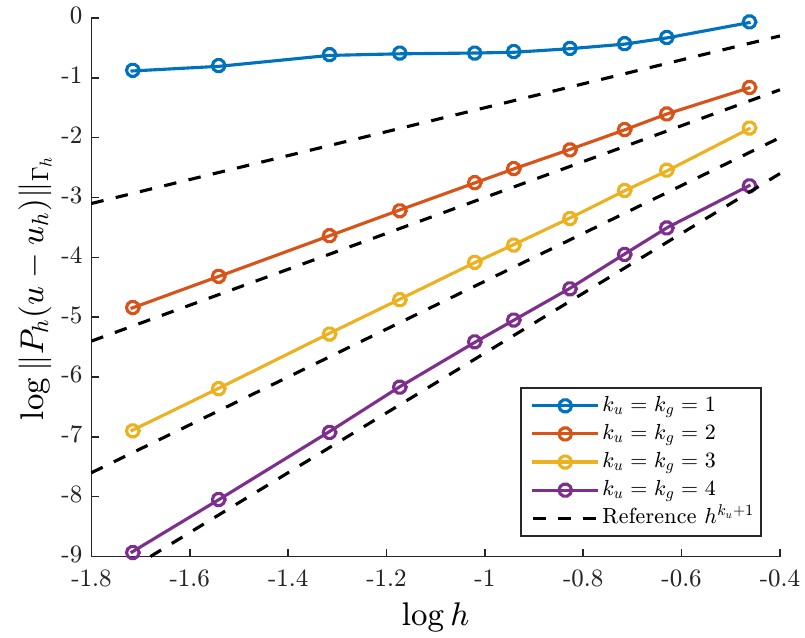}
\label{fig:tangent-a}
}
\subfigure[Lagrange multiplier approach]{
\includegraphics[width=0.35\linewidth]{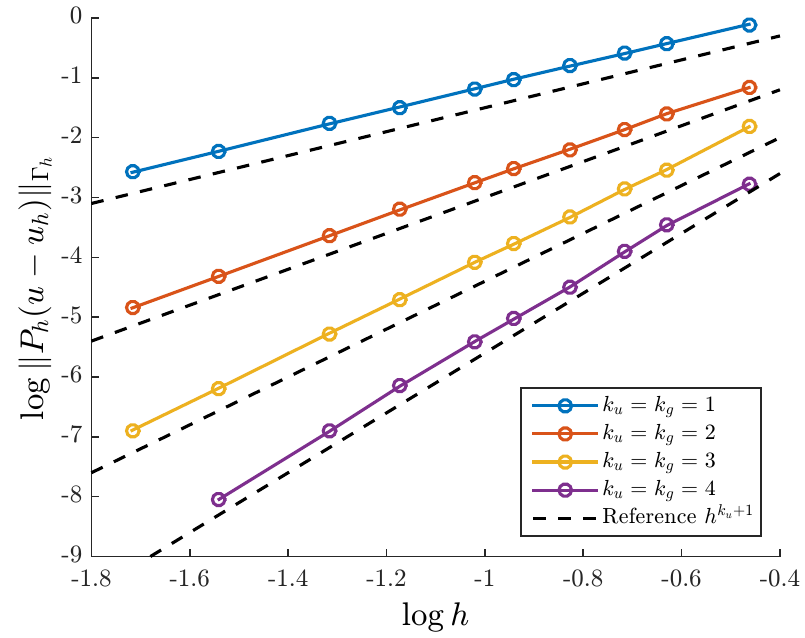}
\label{fig:tangent-c}
}

\caption{\textbf{Tangential convergence.} Convergence studies in $L^2(\Gamma_h)$ norm for the part of the solution tangential to $\Gammah$ in the standard problem. (a) All but the lowest order isoparametric element show optimal order convergence using the penalty term approach. (b) Using the Lagrange multiplier approach also the lowest order isoparametric element exhibits optimal convergence rates.}
\label{fig:tangent}
\end{figure}


\footnotesize{
\bibliographystyle{abbrv}
  \bibliography{ref}
}

\bigskip
\bigskip
\begin{samepage}
\noindent
\footnotesize {\bf Acknowledgments.}
This research was supported in part by the Swedish Foundation
for Strategic Research Grant No.\ AM13-0029,
the Swedish Research Council Grants Nos.\  2013-4708, 2017-03911,
and the Swedish strategic research programme eSSENCE.

\bigskip
\bigskip
\noindent
\footnotesize {\bf Authors' addresses:}

\smallskip
\noindent
Peter Hansbo,  \quad \hfill \addressjushort\\
{\tt peter.hansbo@ju.se}

\smallskip
\noindent
Mats G. Larson,  \quad \hfill \addressumushort\\
{\tt mats.larson@umu.se}

\smallskip
\noindent
Karl Larsson, \quad \hfill \addressumushort\\
{\tt karl.larsson@umu.se}
\end{samepage}

\end{document}